\documentclass[reqno,11pt]{amsart}
\usepackage{amsthm,amsfonts,amssymb,euscript,mathrsfs,graphics,color,amsmath,latexsym,marginnote,bm,mathtools}

\usepackage{soul}

\usepackage[dvips]{graphicx}
\usepackage[T1]{fontenc}
\usepackage{slashed} 
\usepackage{amssymb}
\usepackage{amsmath,amsthm}
\usepackage{amsfonts}
\usepackage{leftidx}
\usepackage{mathrsfs}
\usepackage{enumerate}  
\usepackage{abstract}
\usepackage{stmaryrd}
\usepackage{graphicx}

\usepackage[dvipsnames]{xcolor}
\usepackage{hyperref}
\hypersetup{colorlinks=true, pdfstartview=FitV, linkcolor=BrickRed,citecolor=BrickRed, urlcolor=BrickRed}
\definecolor{labelkey}{rgb}{0.6,0,0}

\usepackage[notcite,notref]{showkeys}
\mathtoolsset{showonlyrefs=true}

\usepackage[margin=1.40in]{geometry}
\numberwithin{equation}{section}

\newcommand{\Y}{\widetilde{Y}}

\providecommand{\ip}[1]{\langle#1\rangle}
\providecommand{\abs}[1]{\left\lvert#1\right\rvert}
\providecommand{\norm}[1]{\left\|#1\right\|}

\def\C{{\mathbb{C}}}

\def\be{{\beta}}

\def\eps{\varepsilon}

\def\varep{\varepsilon}

\def\al{\alpha}

\def\hat{\widehat}

\def\bar{\overline}
\def\R{{\mathbb R}}

\def\R{{\bf R}}

\def\Z{{\mathbb Z}}

\def\p{{\bf p}}

\def\W{\widetilde{W}}

\def\bar{\overline}

\def\R{\mathbb{R}}

\def\T{{\mathbb T}}

\def\R{\mathbb{R}}

\def\Z{\mathbb{Z}}

\def\C{\mathbb{C}}
\def\p{\partial}

\def\be{\begin{equation}}
\def\ee{\end{equation}}

\newtheorem{theorem}{Theorem}[section]
\newtheorem{lemma}[theorem]{Lemma}
\newtheorem{proposition}[theorem]{Proposition}

\newtheorem{remark}[theorem]{Remark}

\numberwithin{equation}{section}
\begin{document}

\title[Landau damping near the Poisson Equilibrium in $\R^3$]{Nonlinear Landau damping for the Vlasov-Poisson system in $\R^3$: the Poisson equilibrium}

\author{Alexandru D.\ Ionescu}
\address{Princeton University}
\email{aionescu@math.princeton.edu}
\author{Benoit Pausader}
\address{Brown University}
\email{benoit\_pausader@brown.edu}
\author{Xuecheng Wang}
\address{YMSC, Tsinghua University\& BIMSA}
\email{xuecheng@tsinghua.edu.cn}
\author{Klaus Widmayer}
\address{University of Zurich}
\email{klaus.widmayer@math.uzh.ch}

\thanks{A.I.\ was supported in part by NSF grant DMS-2007008; B.P.\ was partially supported by a Simons Fellowship, the CY-IAS fellowship program, and NSF grant DMS-1700282; X.W.\  was supported in part by NSFC-12141102, and MOST-2020YFA0713003. K.W.\ gratefully acknowledges support of the SNSF through grant PCEFP2\_203059. This material is based upon work supported by the National Science Foundation under Grant No. DMS-1929284 while A.I., B.P.\ and K.W.\ were in residence at the Institute for Computational and Experimental Research in Mathematics in Providence, RI, during the program ``Hamiltonian Methods in Dispersive and Wave Evolution Equations''. }
\maketitle

\begin{abstract}

We prove asymptotic stability of the Poisson homogeneous equilibrium among solutions of the Vlasov-Poisson system in the Euclidean space $\mathbb{R}^3$. More precisely, we show that small, smooth, and localized perturbations of the Poisson equilibrium lead to global solutions of the Vlasov-Poisson system, which scatter to linear solutions at a polynomial rate as $t\to\infty$.

The Euclidean problem we consider here differs significantly from the classical work on Landau damping in the periodic setting, in several ways. Most importantly, the linearized problem cannot satisfy a ``Penrose condition''. As a result, our system contains resonances (small divisors) and the electric field is a superposition of an electrostatic component and a larger oscillatory component, both with polynomially decaying rates.

\end{abstract}

\setcounter{tocdepth}{1}

\tableofcontents

 \section{Introduction}
It is believed that the vast majority of ordinary matter in the visible universe takes the form of a plasma, i.e.\  of an ionised gas, ranging from sparse intergalactic plasma to the interior of stars and neon
signs. Both from a theoretical and practical (e.g. nuclear fusion) point of view, the understanding of stability versus instability in plasmas is a formidable yet crucial challenge. We refer to e.g.\  \cite{fitzpatrick2014plasma,PlasmaBook} for physics references in book form.

In this article we consider a \textit{hot, unconfined, electrostatic} plasma in three dimensions of electrons on a uniform, static, background of ions. Here collisions are neglected and the associated distribution of electrons is modeled by a measure $M(x,v,t)dxdv$ on the phase space $\mathbb{R}^3_x\times\mathbb{R}^3_v$, where $M$ satisfies the \textit{Vlasov-Poisson} equation.
\begin{equation}\label{eq:VP_physics}
m_e(\partial_t+v\cdot\nabla_x)M+q{\bf E}\cdot\nabla_vM=0,\qquad \hbox{div}_x{\bf E}(x,t)=4\pi \epsilon_0\left\{n_0e+q\int_{\mathbb{R}^3} M dv\right\}.
\end{equation}
In \eqref{eq:VP_physics}, ${\bf E}$ denotes the self-generated electrostatic field, $q=-e<0$ the charge of an electron, $m_e$ its mass, $\epsilon_0$ the vacuum resistivity and $n_0e>0$ the uniform background charge density of the ions.

A particularly simple yet relevant class of solutions to \eqref{eq:VP_physics} are stationary, {\it spatially homogeneous} functions $M_0:\R^3\to[0,\infty)$ satisfying $q\int_{\R^3}M_0\,dv=n_0$. To understand the role of such equilibria in the overall dynamics of the Vlasov-Poisson equations, an important step is the investigation of their stability. Writing
\begin{equation}
 M(x,v,t)=n_0 e^{-1}\left(M_0(v)+F(x,v,t)\right)
\end{equation}
the equation for the perturbation $F$ becomes
\begin{equation}\label{VP}
\begin{split}
\left(\partial_t+v\cdot\nabla_x\right)F+\omega^2_e\nabla_x\psi\cdot\nabla_v M_0&+\omega_e^2\nabla_x\psi\cdot\nabla_vF=0,\qquad
\Delta_x\psi=\int_{\R^3} Fdv,
\end{split}
\end{equation}
with electrostatic potential $\psi$ and \textit{electron plasma frequency} $\omega_e$ given by $\omega_e^2:=\frac{4\pi\epsilon_0 n_0e^2}{m_e}$. Let
\begin{equation*}
 f(x,v,t):=F\big(\frac{x}{\omega_e},v,\frac{t}{\omega_e}\big),\qquad\phi(x,t):=\omega_e^2\cdot \psi\big(\frac{x}{\omega_e},\frac{t}{\omega_e}\big),
\end{equation*}
so we obtain the nondimensionalized Vlasov-Poisson system
\begin{equation}\label{NVP}
\begin{split}
\left(\partial_t+v\cdot\nabla_x\right)f+\nabla_x\phi\cdot \nabla_vf&=-\nabla_x\phi\cdot\nabla_vM_0,\\
\Delta_x\phi(x,t)&=\rho(x,t):=\int_{\R^3} f(x,v,t)\,dv,
\end{split}
\end{equation}
which will be the focus of the rest of this article.

\subsection{Main result}
In this paper we investigate the asymptotic stability of a particular homogeneous equilibrium, namely the ``Poisson equilibrium''
\begin{equation}\label{NVP.2}
M_0(v):=\frac{1}{\pi^2(1+|v|^2)^{2}},\qquad \widehat{M_0}(\xi)=e^{-|\xi|}.
\end{equation}
Our main result asserts that smooth perturbations of $M_0$ lead to global solutions which scatter to linear solutions, and exhibit two different dynamics in their associated electric fields:
 
 \begin{theorem}\label{thm:main_simple}
 There exists $\bar\eps>0$ such that if the initial particle distribution $f_0$ satisfies
 \begin{equation}\label{eq:init}
  \sum_{\abs{\alpha}+\abs{\beta}\leq 1}\norm{\ip{v}^{4.5}\partial_x^\alpha\partial_v^\beta f_0(x,v)}_{L^\infty_x L^\infty_v}+\norm{\ip{v}^{4.5}\partial_x^\alpha\partial_v^\beta f_0(x,v)}_{L^1_x L^\infty_v}\leq\eps_0\leq \bar\eps,
 \end{equation}
 then the Vlasov-Poisson system \eqref{NVP} with $M_0$ defined as in \eqref{NVP.2} has a global unique solution $f\in C^1_{x,v,t}(\mathbb{R}^{3+3}\times\mathbb{R}_+)$ that scatters linearly, i.e.\ there exists $f_\infty\in L^\infty_{x,v}$ such that
 \begin{equation}\label{eq:lin_scatter}
\begin{split}
\Vert f(x,v,t)-f_\infty(x-tv,v)\Vert_{L^\infty_{x,v}}\lesssim\varepsilon_0\langle t\rangle^{-1/2}.
\end{split}
\end{equation}
Moreover, the electric field $E(x,t)=\nabla_x\phi(x,t)$ decomposes into a ``static'' and an ``oscillatory'' component with different decay rates
 \begin{equation}\label{eq:E-decomp}
\begin{split}
&E(t)=E^{stat}(t)+\Re(e^{-it}E^{osc}(t)),\\
&\ip{t}\norm{E^{stat}(t)}_{L^\infty}+\norm{E^{osc}(t)}_{L^\infty}\lesssim \eps_0\ip{t}^{-2+\delta},
\end{split}
 \end{equation}
where $\delta=10^{-4}$ is a small parameter.
\end{theorem}

\begin{remark}
(1)  Theorem \ref{thm:main_simple} appears to be the first nonlinear asymptotic stability result for the Vlasov-Poisson system in $\R^3$ in a neighborhood of a smooth non-trivial equilibrium (see \cite{PW2020} for the case of a repulsive point charge). We work with the Poisson homogeneous equilibrium $M_0$ mostly for the sake of simplicity, as it leads to explicit formulas such as \eqref{eq:rho_self2}--\eqref{eq:rho_self3}. However, we expect that the linear scattering conclusion of the theorem  and the underlying analysis extend to more general smooth homogeneous equilibria, at least as long as the resonances of the system (the set where the kernel $1+K(\xi,\lambda)$ defined as in \eqref{eq:PenroseR3} vanishes) are not too severe. See sections \ref{sec:PW} and \ref{MainIdeas} for further discussion.


(2) The statement of the theorem, in particular the crucial decay estimates \eqref{eq:E-decomp}, depend on the fact that we work in dimension $d=3$. The decay is stronger in higher dimensions $d\geq 4$ (so the proof becomes easier), and weaker and insufficient in dimension $d=2$. This is due mainly to the dimension-dependent dispersive estimates Lemma \ref{disper2}, and is in sharp contrast with the periodic case $x\in\T^d$. See section \ref{sec:PW} for further discussion.

\end{remark}

\subsection{Prior works}\label{sec:PW}

The literature of broadly related results is too vast to be surveyed here, so we will focus on a selection of more directly related results. In particular, we restrict our attention to the setting of three spatial dimensions.

\subsubsection{Perturbations of vacuum} The simplest equilibrium of \eqref{eq:VP_physics} is the \textit{vacuum} $n_0\equiv 0$ and $M_0\equiv 0$. The resulting equations
\begin{equation}\label{VVP}
\left(\partial_t+v\cdot\nabla_x\right)f+\lambda\nabla_x\phi\cdot \nabla_vf=0,\qquad \Delta_x\phi(x,t)=\int f(x,v,t)dv,\qquad\lambda\in\{-1,1\},
\end{equation}
are also relevant in astrophysics, where the self-interactions through a gravitational potential are attractive ($\lambda=-1$) rather than repulsive ($\lambda=+1$). In the former case, \eqref{VVP} possesses a large variety of localized equilibria, see e.g.\ \cite{GR1999,LMR2008,Rei2007,Mou2013}. 

For sufficiently localized initial data, the system \eqref{VVP} has been extensively studied (see e.g.\ \cite{Gla1996,Rei2007} for book references). Solutions are global in time under reasonably general regularity assumptions \cite{BD1985,LP1991,Sch1991,Pfa1992}. For small perturbations, the electric field and charge density decay over time \cite{BD1985}, but the precise dynamics were only recently clarified (after a series of preliminary works \cite{IR1996,Perthame1996,HRV2011,Smu2016,CK2016,Wang2018}): the \textit{long-range effects} of the electrostatic field lead to asymptotic dynamics that feature a logarithmic correction of linear scattering, a phenomenon known as \textit{modified scattering} \cite{IPWW2020,FOPW2021}.

\subsubsection{Nonlinear Landau damping}
As is well-known, the free transport equation $\partial_t f+v\cdot\nabla_x f=0$ exhibits \textit{phase mixing}, which can manifest itself as time decay in the spatial density $\rho(t,x)$. It was an observation of Landau \cite{Lan1946} (see also \cite[Chapter 3, Section 30]{LLX1981}) that an interesting mechanism of decay exists also in the linearized Vlasov-Poisson equations near homogeneous equilibria satisfying certain conditions (nowadays called ``the Penrose criterion'').   
 
$\bullet$ The periodic setting $x\in\mathbb{T}^3$, is a natural model for a \textit{confined} plasma and has been studied extensively. After some preliminary works \cite{CM1998,HV2009}, \textit{nonlinear Landau damping} was proved in the pioneering work \cite{MV11} (see also \cite{BMM2016,GNR2020} for refinements and simplifications). More precisely, for suitable homogeneous equilibria $M_0$ that satisfy the \textit{Penrose criterion}
\begin{equation}\label{eq:Penrose}
\inf_{k\in\Z^3\setminus\{0\},\,\lambda\in\C,\,\Im(\lambda)<0}\abs{1+\int_0^\infty e^{-i\lambda s}s\widehat{M_0}(ks)ds} >0,
\end{equation}
it was shown that small, highly regular (analytic or Gevrey) perturbations lead to a global solution which \textit{scatters linearly}, with an electric field that decays exponentially.

The confined nature of the system induces strong \textit{nonlinear echoes} which, for smooth enough perturbations, are compensated by the fast decay of the linearized equation. Rougher perturbations (e.g.\ Sobolev) can lead to other stationary solutions \cite{BGK1957,LZ2011} or still damp \cite{GNR2020-2}. We note that related mechanisms are at play in $2d$ \textit{fluid mixing}, in particular near monotone shear flows \cite{BM15,Ionescu2018,IJ20,MZ20} or point vortices \cite{IJ2022}.

$\bullet$ In the \textit{unconfined} setting $x\in\R^3$, Theorem \ref{thm:main_simple} is the first nonlinear asymptotic stability result of non-trivial, smooth homogeneous equilibria in 3D. Previous results are mainly concerned with the linearized system, see \cite{GS1994,GS1995a}, and the screened case, see  \cite{BMM2018,HNR2019}, in which the low frequency part is screened out.

The key difficulty for  the unconfined case is caused by the fact that the Penrose criterion cannot hold. More precisely, for any normalized homogeneous equilibrium $M_0$ we have 
\begin{equation}\label{eq:PenroseR3}
\begin{split}
&1+K(0,\lambda)=1-\lambda^{-2}\,\,\,\text{ if }\,\,\,\Im(\lambda)<0\,\,\,\text{thus}\,\,\,
\lim_{\xi\to 0,\,\lambda\to\pm 1,\,\Im(\lambda)<0}\abs{1+K(\xi,\lambda)} =0,\\
&\text{where}\qquad K(\xi,\lambda):=\int_0^\infty e^{-i\lambda s}s\widehat{M_0}(\xi s)\,ds.
\end{split}
\end{equation}
In particular, the key lower bound in the Penrose criterion \eqref{eq:Penrose} cannot hold. This critical difficulty has already been observed in \cite{GS1994,GS1995a,BMM2020,HNR2020} in the study of the linearized system. 

In the nonlinear setting the failure of the Penrose condition leads to the presence of small denominators and resonances in the system. As a consequence we have much slower decay and convergence (depending on the dimension of the ambient space), and the global dynamics is completely different compared to the periodic case. In fact our analysis reveals two different types of dynamics for the electric field: an oscillatory part, which decays at almost the critical rate $\ip{t}^{-2}$ and oscillates like $e^{-it}$ over time, and a static part, which decays faster than the critical rate $\ip{t}^{-2}$. See section \ref{MainIdeas} below for more details.

We note that a suitable analogue of the Penrose condition \eqref{eq:Penrose} still holds in the screened case investigated in \cite{BMM2018,HNR2019}. As a consequence there are no resonances, and one can prove sufficiently rapid decay of the electric field (of order $\langle t\rangle^{-4+}$) for smooth perturbations. For comparison, the decay of the electric field in the case of vacuum perturbations is of the order of $\langle t\rangle^{-2}$.

\subsubsection{Comparison with Euler-Poisson models}

Finally, we note the analogy with the related case of {\it warm plasmas}, when one can use fluid models instead of kinetic ones: the Euler-Poisson equation for electron is classically asymptotically stable (for irrotational data) \cite{Guo1998}, while the stability of the ion equation was more recently obtained \cite{GP2011}. In this case, the two-fluid models involving both electrons and ions lead to new and interesting phenomena \cite{GM2014,IoPau1,GIP2016}.

\subsection{Main ideas}\label{MainIdeas}
As in \cite{HNR2019} we use a Lagrangian approach. The left-hand side of \eqref{NVP} is a transport equation for $f$, the (backwards) characteristics of which are the curves
\begin{equation}\label{eq:chars}
\begin{alignedat}{3}
\partial_s X(x,v,s,t)&=V(x,v,s,t),\qquad &X(x,v,t,t)&=x,\\
\partial_s V(x,v,s,t)&=\nabla_x\phi(X(x,v,s,t),s),\qquad &V(x,v,t,t)&=v.
\end{alignedat}
\end{equation}
Notice that we can integrate \eqref{NVP} to obtain an exact formula for $f(x,v,t)$, namely
\begin{equation}\label{eq:NVPchar}
\begin{split}
 f(x,v,t)&=f_0(X(x,v,0,t),V(x,v,0,t))-\int_{0}^t\nabla_x\phi(X(x,v,s,t),s)\cdot\nabla_vM_0(V(x,v,s,t))ds.
\end{split}
\end{equation}
It thus suffices to study the characteristics, which are given through the electric field $\nabla_x\phi$, which in turn arises from the density $\rho$. Through integration in $v$ in \eqref{eq:NVPchar}, together with \eqref{eq:chars} we recast \eqref{NVP} as a closed system for the density $\rho$ satisfying
\begin{equation}\label{eq:rho_self}
 \rho(x,t)+\int_0^t\int_{\R^3}(t-s)\rho(x-(t-s)v,s)M_0(v)\,dvds=\mathcal{N}(x,t),
\end{equation}
where $\mathcal{N}=\mathcal{N}_1+\mathcal{N}_2$ is a sum of initial data contribution $\mathcal{N}_1$ and a nonlinear expression $\mathcal{N}_2$ in $\rho$ (see \eqref{Lan4}-\eqref{Lan5} below for precise formulas). This formulation is amenable to a bootstrap approach that we detail further.

\subsubsection{Linear analysis and the Penrose condition}
The (linear) left-hand side of \eqref{eq:NVPchar} is a Volterra equation, which can be integrated exactly. Using the Laplace-Fourier transform and the explicit formula \eqref{NVP.2}, after a few algebraic manipulations it follows from \eqref{eq:rho_self} that
\begin{equation}\label{eq:rho_self2}
\begin{split}
&\left(1+K(\xi,\lambda)\right)\widetilde{\rho}(\xi,\lambda)=\widetilde{\mathcal{N}}(\xi,\lambda),\\
&1+K(\xi,\lambda)=1+\int_0^\infty r\widehat{M_0}(r\xi)e^{-i\lambda r}\,dr=\frac{(|\xi|+i(\lambda-1))(|\xi|+i(\lambda+1))}{(\abs{\xi}+i\lambda)^2}.
\end{split}
\end{equation}
Notice that there is no uniform lower bound for $|1+K|$ as $\xi\to 0$ and $\lambda\to\pm 1$, i.e.\ a Penrose condition similar to \eqref{eq:Penrose} cannot be satisfied. However, we can still obtain an explicit solution of the linear problem,
\begin{equation}\label{eq:rho_self3}
 \widehat{\rho}(\xi,t)=\widehat{\mathcal{N}}(\xi,t)-\int_0^t\widehat{\mathcal{N}}(\xi,\tau)e^{-(t-\tau)|\xi|}\sin(t-\tau)\,d\tau.
\end{equation}
This gives a first ``solution formula'' (in the same spirit as a Duhamel formula) for $\rho$ in terms of the initial data and a nonlinear expression of itself.

We write $\sin(t-\tau)=\sin(t)\cos(\tau)-\cos(t)\sin(\tau)$ and do integration by parts in $\tau$ in \eqref{eq:rho_self3} (similar to a normal form analysis) when the oscillation effect is strong, i.e.\ in the case $|v||\xi|\ll 1$. As a result, we derive 
a second formula for the localized density, see \eqref{sug13.1}-\eqref{sug13.4}. 

The identity \eqref{eq:rho_self3} naturally leads to a decomposition of the density 
\begin{equation}\label{eq:decrho}
\rho=\Re(e^{-it}\rho^{osc})+\rho^{stat}
\end{equation}
into an oscillatory mode $\Re(e^{-it}\rho^{osc})$ and a static component $\rho^{stat}$. The corresponding decomposition at the level of the electric field follows by setting $E^{\ast}=\nabla\Delta^{-1}\rho^{\ast}$, $\ast\in\{osc,stat\}$. 

\subsubsection{The bootstrap argument and nonlinear analysis} Our main bootstrap argument is at the level of the density function $\rho$. The main idea is to use the decomposition \eqref{eq:decrho} and bound the two components $\rho^{osc}$ and $\rho^{stat}$ in two different Banach spaces: a stronger space for the static component $\rho^{stat}$ and a weaker space for the oscillatory component $\rho^{osc}$. The choice of these two spaces is very important, see \eqref{normA}--\eqref{may12eqn21} for the precise definitions.\footnote{The choice of these bootstrap spaces is in fact the most important choice in the paper, in the spirit of the ``$Z$-norm method'' used by the authors in earlier work \cite{IoPau1, GIP2016, IPWW2020}. } The point is that even though the oscillatory component $\rho^{osc}$ satisfies weaker bounds, the presence of the factor $e^{-it}$ allows integration by parts in time arguments (normal forms), which lead to improved decay and convergence.

To prove the main bootstrap proposition \ref{MainBootstrapProp} we start from the identity \eqref{eq:rho_self3}. The main nonlinear contribution comes from the {\textit{reaction term}}
\begin{equation}\label{eq:reation}
\begin{split}
 \mathcal{N}_2(x,t)=\int_0^t&\int_{\R^3}\big\{E(x-(t-s)v,s)\cdot \nabla_vM_0(v)\\
 &-E(X(x,v,s,t),s)\cdot \nabla_vM_0(V(x,v,s,t))\big\}\,dvds,
 \end{split}
\end{equation}
where the characteristics $(X,V)(x,v,s,t)$ are defined as in \eqref{eq:chars}. 

The bootstrap assumptions on $\rho$ and the definitions  \eqref{eq:chars} can be used to derive bounds on the deviations $V(x,v,s,t)-v$ and $X(x,v,s,t)-[x-(t-s)v]$ of the linear characteristics. To estimate the reaction term we further localize in $|v|\approx 2^j$ and $x$-frequency $|\xi|\approx 2^k$. We then examine the resulting nonlinear interactions, and analyze them according to the size of the parameters $j$ and $k$. In some cases it is beneficial to integrate by parts in time (normal forms) or in $v$ (improved dispersive estimates), particularly when the oscillatory components $E^{osc}$ are involved.  The analysis of the reaction term \eqref{eq:reation} is the most elaborate part of our proof and covers sections 5, 6, and 7.

Integration by parts in time produces quadratic ``main terms'' and a large number of ``cubic remainders''. Fortunately, in many cases the cubic remainders are of lower order and we can deal with them in systematic fashion using Lemma \ref{keybilinearlemma}.

\subsubsection{General equilibria} It is natural to ask if the main conclusions of Theorem \ref{thm:main_simple} hold for more general homogeneous equilibria $M_0(v)$. Our choice of the Poisson equilibrium simplifies the analysis, since it leads to explicit formulas such as \eqref{eq:rho_self2}--\eqref{eq:rho_self3}, but we expect that the framework we construct here extends to more general situations.

At the technical level, in the case of more general homogeneous equilibria there is an additional difficulty identified in \cite{BMM2020,HNR2020}: even at the linearized level the oscillatory component of the electric field contains an entire family of frequency dependent oscillations instead of just the two discrete modes we have here. This requires a more careful decomposition of the density function and more precise $Z$-norm analysis. We hope to return to these issues soon.

\subsection{Organization} The rest of the paper is organized as follows: in section 2 we derive our main formulas for the density function, identify precisely its decomposition (Corollary \ref{rhodeco}), and state our main bootstrap proposition (Proposition \ref{MainBootstrapProp}). In section 3 we use the bootstrap assumptions to prove bounds on the modified characteristics (Lemmas \ref{derivativeschar} and \ref{derichar2}), general dispersive estimates (Lemma \ref{disper2} and \ref{omegaioLem2}), and bounds on operators defined by Fourier multipliers (Lemma \ref{omegaioLem}). In section 4 we prove improved bootstrap estimates on the contributions of the initial data. In sections 5, 6, and 7 we prove improved bootstrap estimates on the contributions of the reaction term. Finally, in section 8 we complete the proof of the main Theorem \ref{thm:main_simple}.

 \section{Dynamics of the density and bootstrap setup}
Recall from \eqref{NVP} that the perturbation $f$ we study satisfies the equation
\begin{equation}\label{tul2}
\begin{split}
&(\partial_t+v\cdot\nabla_x)f+E\cdot\nabla_vM_0+E\cdot\nabla_vf=0,\\
&\rho(x,t)=\int_{\mathbb{R}^3}f(x,v,t)\,dv,\qquad E:=\nabla_x\Delta_x^{-1}\rho,
\end{split}
\end{equation}
where
\begin{equation}\label{sug2}
M_0(v):=\frac{1}{\pi^2(1+|v|^2)^{2}},\qquad \widehat{M_0}(\xi)=e^{-|\xi|}.
\end{equation}

To recast this as a problem for the density $\rho$ we introduce the ``backwards characteristics'' of \eqref{tul2}: these are the functions $X,V:\R^3\times\R^3\times\mathcal{I}^2_T\to\R^3$ obtained by solving the ODE system
\begin{equation}\label{Lan1}
\begin{alignedat}{2}
&\partial_sX(x,v,s,t)=V(x,v,s,t),\qquad &X(x,v,t,t)&=x,\\
&\partial_sV(x,v,s,t)=E(X(x,v,s,t),s),\qquad &V(x,v,t,t)&=v,
\end{alignedat}
\end{equation}
where $\mathcal{I}^2_T:=\{(s,t)\in[0,T]^2:\,s\leq t\}$. The main equation \eqref{tul2} gives
\begin{equation*}
\begin{split}
\frac{d}{ds}f(X(x,v,s,t),V(x,v,s,t),s)&=[(\partial_s+v\cdot\nabla_x+E\cdot\nabla_v)f](X(x,v,s,t),V(x,v,s,t),s)\\
&=-E(X(x,v,s,t),s)\cdot\nabla_vM_0(V(x,v,s,t)).
\end{split}
\end{equation*}
Integrating over $s\in[0,t]$ and letting $M'_0:=\nabla_v M_0$, we have 
\begin{equation}\label{Lan2}
f(x,v,t)=f_0(X(x,v,0,t),V(x,v,0,t))-\int_0^tE(X(x,v,s,t),s)\cdot M'_0(V(x,v,s,t))\,ds,
\end{equation}
for any $(x,v,t)\in\R^3\times\R^3\times[0,T]$. Adding $\int_0^t\int_{\R^3}E(x-(t-s)v,s)\cdot M'_0(v)$ on both sides of \eqref{Lan2}, a further integration in $v$ shows that
\begin{equation}\label{Lan4}
\rho(x,t)+\int_0^t\int_{\R^3}(t-s)\rho(x-(t-s)v,s)M_0(v)\,dvds=\mathcal{N}(x,t),
\end{equation}
for any $(x,t)\in\R^3\times[0,T]$, where
\begin{equation}\label{Lan5}
\begin{split}
\mathcal{N}(x,t)&:=\mathcal{N}_1(x,t)+\mathcal{N}_2(x,t),\\
\mathcal{N}_1(x,t)&:=\int_{\R^3}f_0(X(x,v,0,t),V(x,v,0,t))\,dv,\\
\mathcal{N}_2(x,t)&:=\int_0^t\int_{\R^3}\big\{E(x-(t-s)v,s)\cdot M'_0(v)-E(X(x,v,s,t),s)\cdot M'_0(V(x,v,s,t))\big\}\,dvds.
\end{split}
\end{equation}
Since $E=\nabla_x\Delta_x^{-1}\rho$ can be recovered directly from the density $\rho$, together with the characteristic equations \eqref{Lan1}, these equations yield a closed system.

\subsection{Solving the density equation}
In view of \eqref{Lan4}, the Fourier transform $\widehat{\rho}$ of the density satisfies a forced Volterra equation
\begin{equation}\label{ForcedLearizedEq}
\begin{split}
 \widehat{\rho}(\xi,t)+\int_{0}^t(t-s)\widehat{M}_0((t-s)\xi)\widehat{\rho}(\xi,s)ds=\widehat{H}(\xi,t),\qquad t\geq 0,
\end{split}
\end{equation}
where $H$ is a forcing term. Such equations can be solved as follows:
\begin{lemma}\label{LinDensity}
Assume $\rho$ and $H$ satisfy \eqref{ForcedLearizedEq} and define the convolution kernel $G$ by
\begin{equation}\label{DefKernel1}
\begin{split}
G(\xi,\tau)&:=\delta_0(\tau)-\sin(\tau)e^{-\tau\vert\xi\vert}\mathbf{1}_+(\tau),
\end{split}
\end{equation}
where $\mathbf{1}_+$ denotes the characteristic function of the interval $[0,\infty)$. Then
\begin{equation}\label{FormulaForRho}
\begin{split}
\widehat{\rho}(\xi,t)&=\int_{0}^t G(\xi,s)\widehat{H}(\xi,t-s)\,ds.
\end{split}
\end{equation}
\end{lemma}

\begin{proof}
Recall the formulas \eqref{sug2}. For any $\xi\neq 0$ and $\lambda\in\C$ with $\Im\lambda\leq 0$ we have
\begin{equation}\label{sug3}
K(\xi,\lambda):=\int_0^\infty r\widehat{M_0}(r\xi)e^{-i\lambda r}\,dr=\int_0^\infty re^{-r(|\xi|+i\lambda)}\,dr=\frac{1}{(|\xi|+i\lambda)^2}.
\end{equation}
Then \eqref{ForcedLearizedEq} shows that
\begin{equation}
 \mathcal{F}_t(\widehat{\rho}(\xi,t))(\lambda)+K(\xi,\lambda)\mathcal{F}_t(\widehat{\rho}(\xi,t))(\lambda)=\mathcal{F}_t(\widehat{H}(\xi,t))(\lambda).
\end{equation}
The claim follows by computing that
\begin{equation}\label{sug4}
\frac{1}{1+K(\xi,\lambda)}=1-\frac{1}{(|\xi|+i\lambda)^2+1}=1-\frac{1}{2i}\Big[\frac{1}{|\xi|+i(\lambda-1)}-\frac{1}{|\xi|+i(\lambda+1)}\Big],
\end{equation}
so
\begin{equation}\label{sug5}
\mathcal{F}_\lambda^{-1}\Big(\frac{1}{1+K(\xi,\lambda)}\Big)(\tau)=\delta_0(\tau)-e^{-\tau|\xi|}\sin\tau \mathbf{1}_+(\tau)=:G(\xi,\tau),
\end{equation}
as desired.
\end{proof}

Using this we can solve the equations for forcing terms as they appear in \eqref{Lan4}-\eqref{Lan5}. This introduces the new, real operator
\begin{equation}\label{DefD}
\begin{split}
\mathcal{D}(\nabla,v):=\vert\nabla_x\vert-v\cdot\nabla_x,
\end{split}
\end{equation}
and reveals a discretely oscillatory dynamic: 

\begin{lemma}\label{LemSolvingFrocingTerm}
For a given $h:\R^3\times\R^3\times\R\to\R$ let
\begin{equation}\label{eq:Hh}
H(x,t):=\int_{\mathbb{R}^3}h(x-tv,v,t)dv.
\end{equation}
Then we have two alternative expressions $\mathcal{S}^{I}[h]$ and $\mathcal{S}^{II}[h]$ for the solution of the associated equation \eqref{ForcedLearizedEq}, with
\begin{equation}
 \mathcal{S}^{\ast}[h](x,t)=R^{\ast}(x,t)+\Re\{e^{-it}T^{\ast}(x,t)\},\qquad \ast\in\{I,II\},
\end{equation}
where
\begin{equation}
\begin{aligned}
 R^{I}(x,t)&=\int_{\mathbb{R}^3} h(x-tv,v,t)dv,\\
 T^{I}(x,t)&=-i\int_{0}^te^{is}e^{-(t-s)\vert\nabla\vert}\int_{\mathbb{R}^3}h(x-sv,v,s)dvds,\\
\end{aligned}
\end{equation}
and
\begin{equation}
\begin{aligned}
 R^{II}(x,t)&=\int_{\mathbb{R}^3} \frac{\mathcal{D}^2}{1+\mathcal{D}^2}h(x-tv,v,t)dv,\\
 T^{II}(x,t)&=e^{-t\vert\nabla\vert}\int_{\mathbb{R}^3}\frac{1}{1-i\mathcal{D}}h(x,v,0)dv\\
 &\quad+\int_{0}^t\int_{\mathbb{R}^3}\frac{e^{is}}{1-i\mathcal{D}}e^{-(t-s)\vert\nabla\vert}(\partial_s h)(x-sv,v,s)dvds.
\end{aligned} 
\end{equation}
\end{lemma}

\begin{proof}[Proof of Lemma \ref{LemSolvingFrocingTerm}]
The expressions for $\mathcal{S}^I[h]$ follow directly from the solution formula \eqref{FormulaForRho}. For $\mathcal{S}^{II}[h]$, we have
\begin{equation}
 \mathcal{F}\{h(\cdot-vt,v,t)\}=e^{-itv\cdot\xi}\hat{h}(\xi,v,t),
\end{equation}
and we observe that
\begin{equation}\label{NewFormG}
\begin{split}
G(\xi,\lambda)e^{i\lambda v\cdot\xi}\mathbf{1}_{[0,t]}(\lambda)&=\frac{\mathcal{D}^2}{1+\mathcal{D}^2}\delta_0(\lambda)+\frac{\cos(t)+\mathcal{D}\sin(t)}{1+\mathcal{D}^2}e^{-t\mathcal{D}}\delta_0(\lambda-t)\\
&\quad +\frac{1}{1+\mathcal{D}^2}\frac{d}{d\lambda}\left\{(\cos(\lambda)+\mathcal{D}\sin(\lambda))e^{-\lambda\mathcal{D}}\mathbf{1}_{[0,t]}(\lambda)\right\},\\
\end{split}
\end{equation}
which gives the claim.
\end{proof}

\subsection{The main decomposition} 
Applying Lemma \ref{LemSolvingFrocingTerm} to our equations \eqref{Lan4}-\eqref{Lan5}, the density $\rho$ naturally decomposes into ``static'' and ``oscillatory'' components. We will control these parts using a bootstrap argument and different norms. 

To implement this, we further decompose the nonlinear terms of \eqref{Lan5} in frequency, velocity space, and time. For this we fix an even smooth function $\varphi: \R\to[0,1]$ supported in $[-8/5,8/5]$ and equal to $1$ in $[-5/4,5/4]$,
and define
\begin{equation}\label{phik*}
\varphi_k(x) := \varphi(|x|/2^k) - \varphi(|x|/2^{k-1}) , \quad \varphi_{\leq k}(x):=\varphi(|x|/2^k),
  \quad\varphi_{>k}(x) := 1-\varphi(x/2^{k}),
\end{equation}
for any $k\in\mathbb{Z}$ and $x\in\R^d$, $d\geq 1$. Let $P_k$, $P_{\leq k}$, and $P_{> k}$ denote the operators on $\R^3$ defined by the Fourier multipliers $\varphi_k$, $\varphi_{\leq k}$, and $\varphi_{> k}$ respectively.

Let $\widetilde{\varphi}_0=\varphi_{\leq 0}$ and $\widetilde{\varphi}_j=\varphi_{j}$ if $j\geq 1$. For any interval $I\subseteq\R$ let
\begin{equation}\label{phisum}
\varphi_I(x):=\sum_{k\in\Z\cap I}\varphi_k(x),\qquad \widetilde{\varphi}_I(x):=\sum_{j\in\Z_+\cap I}\widetilde{\varphi}_j(x).
\end{equation}

 We define the functions $L_1,L_{1,j}:\R^3\times\R^3\times[0,T]\to\R$ and $L_{2,j}:\R^3\times\R^3\times\mathcal{I}_T^2\to\R$  by
\begin{equation}\label{qwp1}
\begin{split}
L_{1,j}(x,v,t)&:=\widetilde{\varphi}_j(v)L_1(x,v,t),\quad L_1(x,v,t):=f_0(X(x+tv,v,0,t),V(x+tv,v,0,t)),\\
L_{2,j}(x,v,s,t)&:=E(x,s)\cdot M'_j(v)-E(X(x+(t-s)v,v,s,t),s)\cdot M'_j(V(x+(t-s)v,v,s,t)),
\end{split}
\end{equation} 
where $j\in\Z_+$ and $M'_j(v):=\widetilde{\varphi}_j(v)\nabla_vM_0(v)$. Then for any $j\in\Z_+$ and $k\in\Z$ we define the functions $L_{1,j,k}:\R^3\times\R^3\times[0,T]\to\R$ and $L_{2,j,k}:\R^3\times\R^3\times\mathcal{I}_T^2\to\R$ by
\begin{equation}\label{qwp2}
L_{1,j,k}(x,v,t):=P_kL_{1,j}(x,v,t),\qquad L_{2,j,k}(x,v,s,t):=P_kL_{2,j}(x,v,s,t),
\end{equation} 
where the projections $P_k$ act in the $x$ variable. It follows from \eqref{Lan5} that
\begin{equation}\label{Lan10}
\begin{split}
&\mathcal{N}_1=\sum_{j\in\Z_+,\,k\in\Z}\mathcal{N}_{1,j,k},\qquad\mathcal{N}_{1,j,k}(x,t):=\int_{\R^3}L_{1,j,k}(x-tv,v,t)\,dv,\\
&\mathcal{N}_2=\sum_{j\in\Z_+,\,k\in\Z}\mathcal{N}_{2,j,k},\qquad\mathcal{N}_{2,j,k}(x,t):=\int_0^t\int_{\R^3}L_{2,j,k}(x-(t-s)v,v,s,t)\,dvds.
\end{split}
\end{equation}

\subsubsection*{Contributions from initial data}
For the contributions from $\mathcal{N}_1$, using the expressions $I$ of Lemma \ref{LemSolvingFrocingTerm} we thus we have the first decomposition
\begin{equation}\label{deco1}
\begin{split}
&\rho_{1,j,k}(x,t)=R^I_{1,j,k}(x,t)+\Re\big\{e^{-it}T^I_{1,j,k}(x,t)\big\},\\
&\mathcal{F}(R^I_{1,j,k})(\xi,t):=\int_{\R^3}\widehat{L_{1,j,k}}(\xi,v,t)e^{-itv\cdot\xi}\,dv,\\
&\mathcal{F}(T^I_{1,j,k})(\xi,t):=\int_0^t\int_{\R^3}e^{-(t-s)|\xi|}e^{-isv\cdot\xi}(-i)e^{is}\cdot\widehat{L_{1,j,k}}(\xi,v,s)\,dvds.
\end{split}
\end{equation}
Alternatively, formulation $II$ of Lemma \ref{LemSolvingFrocingTerm} yields a decomposition as
\begin{equation}\label{sug13.1}
\rho_{1,j,k}(x,t)=R^{II}_{1,j,k}(x,t)+\Re\big\{e^{-it}T^{II}_{1,j,k}(x,t)\big\},
\end{equation}
where
\begin{equation}\label{sug13.2}
\mathcal{F}(R^{II}_{1,j,k})(\xi,t):=\int_{\R^3}\widehat{L_{1,j,k}}(\xi,v,t)e^{-itv\cdot\xi}\frac{\mathcal{D}^2}{1+\mathcal{D}^2}\,dv,
\end{equation}
\begin{equation}\label{sug13.4}
\begin{split}
\mathcal{F}(T_{1,j,k}^{II})(\xi,t)&:=e^{-t|\xi|}\int_{\R^3}\frac{1}{1-i\mathcal{D}}\widehat{L_{1,j,k}}(\xi,v,0)\,dv\\
&+\int_0^t \int_{\R^3}e^{-(t-s)|\xi|}e^{-isv\cdot\xi}\frac{e^{is}}{1-i\mathcal{D}}(\partial_s\widehat{L_{1,j,k}})(\xi,v,s)\,dvds
\end{split}
\end{equation}
and we have slightly abused notation to denote by $\mathcal{D}$ also its Fourier symbol 
\begin{equation}\label{sug12}
\mathcal{D}(\xi,v):=|\xi|-iv\cdot\xi.
\end{equation}

\subsubsection*{Contributions from nonlinearity: ``reaction terms''}
Similarly, contributions from $\mathcal{N}_2$ can be decomposed with formulation $I$ of Lemma \ref{LemSolvingFrocingTerm} to obtain
\begin{equation}\label{deco2}
\begin{split}
&\rho_{2,j,k}(x,t)=R^I_{2,j,k}(x,t)+\Re\big\{e^{-it}T_{2,j,k}^I(x,t)\big\},\\
&\mathcal{F}(R^I_{2,j,k})(\xi,t):=\int_0^t\int_{\R^3}\widehat{L_{2,j,k}}(\xi,v,\tau,t)e^{-i(t-\tau)v\cdot\xi}\,dv d\tau,\\
&\mathcal{F}(T^I_{2,j,k})(\xi,t):=\int_0^t\int_0^t\int_{\R^3}\mathbf{1}_+(s-\tau)e^{-(t-s)|\xi|}(-i)e^{is}\widehat{L_{2,j,k}}(\xi,v,\tau,s)e^{-i(s-\tau)v\cdot\xi}\,dv d\tau ds.
\end{split}
\end{equation}
Alternatively, formulation $II$ gives the expressions
\begin{equation}\label{sug23.1}
\rho_{2,j,k}(x,t)=R^{II}_{2,j,k}(x,t)+\Re\big\{e^{-it}T^{II}_{2,j,k}(x,t)\big\},
\end{equation}
where
\begin{equation}\label{sug23.2}
\mathcal{F}(R^{II}_{2,j,k})(\xi,t):=\int_0^t\int_{\R^3}\widehat{L_{2,j,k}}(\xi,v,\tau,t)e^{-i(t-\tau)v\cdot\xi}\frac{\mathcal{D}^2}{1+\mathcal{D}^2}\,dvd\tau,
\end{equation}
\begin{equation}\label{sug23.4}
\begin{split}
\mathcal{F}(T^{II}_{2,j,k})(\xi,t):=\int_0^t\int_0^t \int_{\R^3}&\mathbf{1}_+(s-\tau)e^{-(t-s)|\xi|}e^{-i(s-\tau)v\cdot\xi}\\
&\times\frac{e^{is}}{1-i\mathcal{D}}(\partial_s\widehat{L_{2,j,k}})(\xi,v,\tau,s)\,dvd\tau ds.
\end{split}
\end{equation}

\subsubsection*{Full decomposition} An important issue is whether to use the first representation I or the second representation II to recover the density $\rho$. Ideally, we would like to use the second representation, for two reasons: (1) the static terms $R^{II}$ contain additional favorable factors at low frequencies, compared to the terms $R^I$, and (2) the derivative $\partial_sh$ of the ``profile" $h$ is expected to be smaller than the profile itself. 

However, the representation II contains the potentially small denominator $1-\langle v,\xi\rangle-i\vert\xi\vert$, coming from the normal form. To avoid the associated singularity, our basic idea is to use the representation I when this denominator is small (essentially $\big|1-\langle v,\xi\rangle-i\vert\xi\vert\big|\lesssim 1$) and then use the representation II when the denominator is large. More precisely we define the sets
\begin{equation}\label{sug30}
\begin{split}
A^I&:=\big\{(j,k,m)\in\Z_+\times\Z\times\Z_+:\,m<\delta^{-4}\text{ or }j>19m/20\text{ or }k+j+\delta m/3> 0\big\},\\
A^{II}&:=\big\{(j,k,m)\in\Z_+\times\Z\times\Z_+:\,m\geq\delta^{-4}\text{ and }j\leq 19m/20\text{ and }k+j+\delta m/3\leq 0\big\},
\end{split}
\end{equation}
where $\delta:=10^{-4}$ is a small parameter. We have thus established the following full decomposition of the density $\rho$:

\begin{proposition}\label{rhodeco}
Assume that $f:\R^3\times\R^3\times[0,T]\to\R$ is a regular solution of the system \eqref{tul2} and define the function $\rho$ as before. Then we can decompose
\begin{equation}\label{sug31}
\rho=\rho^{stat}+\Re\big\{e^{-it}\rho^{osc}\big\},
\end{equation}
where, with the definitions \eqref{deco1}--\eqref{sug13.4} and \eqref{deco2}--\eqref{sug30}
\begin{equation}\label{sug34}
\begin{split}
\rho^{stat}(x,t)=\sum_{\ast\in\{I,II\}}\sum_{(j,k,m)\in A^\ast}\widetilde{\varphi}_m(t)[R^{\ast}_{1,j,k}(x,t)+R^{\ast}_{2,j,k}(x,t)],\\
\rho^{osc}(x,t)=\sum_{\ast\in\{I,II\}}\sum_{(j,k,m)\in A^\ast}\widetilde{\varphi}_m(t)[T^{\ast}_{1,j,k}(x,t)+T^{\ast}_{2,j,k}(x,t)].
\end{split}
\end{equation}
\end{proposition}

The main point of this decomposition is that we will be able to prove stronger control of the low frequencies of the stationary component $\rho^{stat}$ compared to the time-oscillatory component $\rho^{osc}$. See Proposition \ref{MainBootstrapProp} below.

\subsection{Norms and the bootstrap proposition}\label{bootstrapass} 

We are now ready to define our main norms and state the main bootstrap proposition. Let $B_T$ denote the space of continuous functions on $\R^3\times [0,T]$ defined by the norm
\begin{equation}\label{normA}
\begin{split}
&\|f\|_{B_T}:=\sup_{t\in[0,T]}\|f(t)\|_{B^0_t},\\
&\|f(t)\|_{B^0_t}:=\sup_{k\in\Z}\big\{\langle t\rangle^3\|P_kf(t)\|_{L^\infty}+\|P_kf(t)\|_{L^1}\big\}.
\end{split}
\end{equation}
Assume that $\delta=10^{-4}$ is a small parameter and define the norms
\begin{equation}\label{may12eqn21}
\begin{split}
\Vert f\Vert_{Stat_\delta}&:=\Vert  \langle t\rangle^{1-2\delta}\langle \nabla_x\rangle      f(t)\Vert_{B_T},\\
\Vert f\Vert_{Osc_\delta}&:=\Vert  \langle t\rangle^{-\delta}f(t)\Vert_{B_T}+\Vert   \langle t\rangle^{1-2\delta}\nabla_{x,t}f(t)\Vert_{B_T}.
\end{split}
\end{equation}

We are now ready to state our main bootstrap proposition:
\begin{proposition}\label{MainBootstrapProp} 
There exist $\overline{\varepsilon}\in(0,1]$ and $C\geq 1$ such that the following is true:

Assume that $f:\R^3\times\R^3\times[0,T]\to\R$ is a regular solution of the system \eqref{tul2} for some $T>0$ with initial data $f(0)=f_0:\R^3\times\R^3\to\R$ satisfying the smallness condition
\begin{equation}\label{bootinit}
\|\langle v\rangle^{4.5}(\partial^\al_x\partial_v^\beta f_0)(x,v)\|_{(L^\infty_x\cap L^1_x) L^\infty_v}\leq\varep_0\leq\overline{\varep},\quad \alpha,\beta\in\mathbb{N}_0^3,\quad \abs{\alpha}+\abs{\beta}\leq 1.
\end{equation}
Assume that the associated density $\rho$ decomposes as $\rho=\rho^{stat}+\Re\{e^{-it}\rho^{osc}\}$ as in Proposition \ref{rhodeco}, where for some $0<\varep_1\leq\varep_0^{3/4}$
\begin{equation}\label{YW12}
\Vert \rho^{stat}\Vert_{Stat_\delta}+\Vert \rho^{osc}\Vert_{Osc_\delta}\leq\varep_1.
\end{equation}
Then the functions $\rho^{stat}$ and $\rho^{osc}$ satisfy the improved bounds
\begin{equation}\label{YW12improved}
\Vert \rho^{stat}\Vert_{Stat_\delta}+\Vert \rho^{osc}\Vert_{Osc_\delta}\leq \varep_0+C\varep_1^{3/2}.
\end{equation}
\end{proposition}

%

The proof of Proposition \ref{MainBootstrapProp} will take up sections 4, 5, 6, and 7. We decompose
\begin{equation}\label{bvn0}
\begin{split}
\rho^{stat}&=\rho^{stat}_{1,I}+\rho^{stat}_{1,II}+\rho^{stat}_{2,I}+\rho^{stat}_{2,II},\\
\rho^{osc}&=\rho^{osc}_{1,I}+\rho^{osc}_{1,II}+\rho^{osc}_{2,I}+\rho^{osc}_{2,II},\\
\end{split}
\end{equation}
where for $a\in\{1,2\}$ and $\ast\in\{I,II\}$ we define 
\begin{equation}\label{bvn2}
\begin{split}
\rho^{stat}_{a,\ast}(x,t)&:=\sum_{(j,k,m)\in A^\ast}\widetilde{\varphi}_m(t)R^{\ast}_{a,j,k}(x,t),\\
\rho^{osc}_{a,\ast}(x,t)&:=\sum_{(j,k,m)\in A^\ast}\widetilde{\varphi}_m(t)T^{\ast}_{a,j,k}(x,t).
\end{split}
\end{equation}

Finally, in Section \ref{ProofMainThm} we show how Proposition \ref{MainBootstrapProp} implies our main result Theorem \ref{thm:main_simple}.

\section{Preliminary estimates}\label{Prelims}

In this section we assume that $f:\R^3\times\R^3\times[0,T]\to\R$ is a regular solution of the system \eqref{tul2}, and the density function $\rho$ satisfies the bootstrap assumptions \eqref{YW12}.

To measure the deviation of the characteristic flow from the free flow we define the functions $\widetilde{Y},\widetilde{W}:\R^3\times\R^3\times\mathcal{I}^2_T\to\R^3$ by
\begin{equation}\label{Lan6}
\begin{split}
\widetilde{Y}(x,v,s,t)&:=X(x+tv,v,s,t)-x-sv,\\
\widetilde{W}(x,v,s,t)&:=V(x+tv,v,s,t)-v.
\end{split}
\end{equation}
The definitions \eqref{Lan1} show that $\widetilde{Y}(x,v,t,t)=0,\widetilde{W}(x,v,t,t)=0$ and
\begin{equation}\label{Lan7}
\begin{split}
\widetilde{W}(x,v,s,t)&=-\int_s^tE(x+\tau v+\widetilde{Y}(x,v,\tau,t),\tau)\,d\tau,\\
\widetilde{Y}(x,v,s,t)&=\int_s^t(\tau-s)E(x+\tau v+\widetilde{Y}(x,v,\tau,t),\tau)\,d\tau.
\end{split}
\end{equation}
Moreover, for any $s\leq t\in[0,T]$ and $x,v\in\R^3$ we have
\begin{equation}\label{Lan7.5}
\partial_s\widetilde{Y}(x,v,s,t)=\widetilde{W}(x,v,s,t),\qquad \widetilde{Y}(x,v,s,t)=-\int_s^t\widetilde{W}(x,v,\tau,t)\,d\tau.
\end{equation}

\subsection{Bounds on the density function $\rho$ and electric field E}\label{ElecFi} The decomposition $\rho=\rho^{stat}+\Re\{e^{-it}\rho^{osc}\}$ in Proposition \ref{rhodeco} induces a natural decomposition of the electric field
\begin{equation}\label{Laga5}
\begin{split}
&E(t)=E^{stat}(t)+\Re\{e^{-it} E^{osc}(t)\},\\
&E^{stat}(t):=(\nabla_x\Delta_x^{-1}\rho^{stat})(t),\qquad E^{osc}(t):=(\nabla_x\Delta_x^{-1}\rho^{osc})(t).
\end{split}
\end{equation}
We start with some bounds on the density components $\rho^{stat}$, $\rho^{osc}$ and on the electric fields $E^{stat}$, $E^{osc}$.

\begin{lemma}\label{Laga10}
(i) For any $k\in\Z$ and $t\in[0,T]$ we have
\begin{equation}\label{Laga2}
\langle t\rangle^{3}\|(P_k\rho^{stat})(t)\|_{L^\infty}+\|(P_k\rho^{stat})(t)\|_{L^1}\lesssim\varep_12^{-k^+}\langle t\rangle^{-1+2\delta},
\end{equation}
\begin{equation}\label{Laga3}
\langle t\rangle^{3}\|(P_k\rho^{osc})(t)\|_{L^\infty}+\|(P_k\rho^{osc})(t)\|_{L^1}\lesssim\frac{\varep_1}{\langle t\rangle^{1-2\delta}2^{k}+\langle t\rangle^{-\delta}},
\end{equation}
\begin{equation}\label{Laga4}
\langle t\rangle^{3}\|(P_k\nabla_{x,t}\rho^{osc})(t)\|_{L^\infty}+\|(P_k\nabla_{x,t}\rho^{osc})(t)\|_{L^1}\lesssim \varep_1\langle t\rangle^{-1+2\delta}.
\end{equation}

(ii) Let $R_j$ denote the Riesz transforms $R_j=\partial_j|\nabla|^{-1}$, $j\in\{1,2,3\}$. Then, for any $t\in[0,T]$
\begin{equation}\label{Laga11}
\begin{split}
&\|E^{stat}(t)\|_{L^\infty}+\sum_{j\in\{1,2,3\}}\|R_jE^{stat}(t)\|_{L^\infty}\lesssim \varep_1\langle t\rangle^{-3+2\delta},\\
&\|\nabla_xE^{stat}(t)\|_{L^\infty}\lesssim \varep_1\langle t\rangle^{-4+2\delta}\ln(2+t),
\end{split}
\end{equation}
and
\begin{equation}\label{Laga12}
\begin{split}
&\|E^{osc}(t)\|_{L^\infty}+\sum_{j\in\{1,2,3\}}\|R_jE^{osc}(t)\|_{L^\infty}\lesssim \varep_1\langle t\rangle^{-2+\delta},\\
&\|\nabla_xE^{osc}(t)\|_{L^\infty}\lesssim \varep_1\langle t\rangle^{-3+\delta}\ln(2+t),\qquad \|\partial_tE^{osc}(t)\|_{L^\infty}\lesssim \varep_1\langle t\rangle^{-3+2\delta}.
\end{split}
\end{equation}
\end{lemma}

\begin{proof} The bounds \eqref{Laga2}--\eqref{Laga4} follow directly from the bootstrap assumptions \eqref{YW12} and the definitions \eqref{normA}--\eqref{may12eqn21}. To prove \eqref{Laga11}--\eqref{Laga12} we use \eqref{Laga2}--\eqref{Laga4} and the simple estimates $\|P_kg\|_{L^\infty}\lesssim 2^{3k}\|P_kg\|_{L^1}$. Therefore
\begin{equation*}
\begin{split}
&\|E^{stat}(t)\|_{L^\infty}+\sum_{j\in\{1,2,3\}}\|R_jE^{stat}(t)\|_{L^\infty}\\
&\lesssim \sum_{2^k\geq \langle t\rangle^{-1}}2^{-k}\|(P_k\rho^{stat})(t)\|_{L^\infty}+\sum_{2^k\leq \langle t\rangle^{-1}}2^{2k}\|(P_k\rho^{stat})(t)\|_{L^1}\lesssim \varep_1\langle t\rangle^{-3+2\delta}
\end{split}
\end{equation*}
and
\begin{equation*}
\begin{split}
\|\nabla_xE^{stat}(t)\|_{L^\infty}&\lesssim \sum_{2^k\geq \langle t\rangle^{-1}}\|(P_k\rho^{stat})(t)\|_{L^\infty}+\sum_{2^k\leq \langle t\rangle^{-1}}2^{3k}\|(P_k\rho^{stat})(t)\|_{L^1}\\
&\lesssim \varep_1\langle t\rangle^{-4+2\delta}\ln(2+t).
\end{split}
\end{equation*}
The bounds \eqref{Laga11} follow. The bounds \eqref{Laga12} are similar, using \eqref{Laga3}--\eqref{Laga4}.
\end{proof}

\subsection{Fourier multipliers} Some of our main identities involve Fourier multipliers, such as the ``derivative'' operators $\mathcal{D}$. To estimate these contributions efficiently we need to localize in the Fourier space. 

For any $d\geq 1$ let $S^\infty=S^\infty(\R^d)$ denote the space of continuous compactly supported multipliers $m:\R^d\to\mathbb{C}$ defined by the norm
\begin{equation}\label{locl0}
\|m\|_{S^\infty}:=\|\mathcal{F}^{-1}m\|_{L^1}<\infty.
\end{equation}
Notice that if $m,m'\in S^\infty(\R^d)$ then
\begin{equation}\label{locl0.1}
\|mm'\|_{S^\infty}\leq\|m\|_{S^\infty}\|m'\|_{S^\infty}.
\end{equation}
Moreover if $m\in S^\infty(\R^d)$ and $f\in L^p(\R^d)$, $p\in[1,\infty]$, then
\begin{equation}\label{locl0.2}
\|\mathcal{F}^{-1}(m\widehat{f})\|_{L^p}\lesssim \|m\|_{S^\infty}\|f\|_{L^p}.
\end{equation} 

We often use integration by parts to estimate oscillatory integrals with non-degenerate phases, according to the following general lemma:

\begin{lemma}\label{tech5} Assume that $0<\eps\leq 1/\eps\leq K$, $N\geq 1$ is an integer, and $f,g\in C^N_0(\mathbb{R}^d)$, $d\geq 1$. Then
\begin{equation}\label{ln1}
\Big|\int_{\mathbb{R}^d}e^{iKf}g\,dx\Big|\lesssim_N (K\eps)^{-N}\big[\sum_{|\alpha|\leq N}\eps^{|\alpha|}\|D^\alpha_xg\|_{L^1}\big],
\end{equation}
provided that $f$ is real-valued, 
\begin{equation}\label{ln2}
|\nabla_x f|\geq \mathbf{1}_{{\mathrm{supp}}\,g},\quad\text{ and }\quad\|D_x^\alpha f \cdot\mathbf{1}_{{\mathrm{supp}}\,g}\|_{L^\infty}\lesssim_N\eps^{1-|\alpha|},\,2\leq |\alpha|\leq N+1.
\end{equation}
\end{lemma}

We define the resonant projectors $\Pi_{\iota,k,p}$, $\iota\in\{+1,-1\}$, $k\leq 4$, $p\in\Z$, by the Fourier multipliers $\varphi_k(\xi)\varphi_p(1+\iota v\cdot\xi)$, i.e.
\begin{equation}\label{omeio0}
\Pi_{\iota,k,p}f(x,v):=\mathcal{F}^{-1}\big\{\widehat{f}(\xi)\cdot\varphi_k(\xi)\varphi_p(1+\iota v\cdot\xi)\big\}(x,v).
\end{equation}
Similarly, we define the projectors $\Pi_{\iota,k,\leq p}$, $\Pi_{\iota,k,>p}$, $k\leq 4$, $p\in\Z$. These projectors are important in arguments involving normal forms. We define also the functions
\begin{equation}\label{omeio1}
\omega_{\iota,\kappa}(\xi,v):=1+\iota v\cdot\xi+i\kappa |\xi|,\qquad\iota\in\{-1,1\},\,\kappa\in\{-1,0,1\},
\end{equation}
and the operators $\omega_{\iota,\kappa}^{-1}$ defined by the multipliers $(\xi,v)\to \omega_{\iota,\kappa}^{-1}(\xi,v)$. We prove the following:

\begin{lemma}\label{omegaioLem}
(i) If $\iota\in\{-1,1\}$, $k\leq 4$ and $p\leq -4$ then
\begin{equation}\label{omegaio2}
\Pi_{\iota,k,\leq p} g(x,v)=\int_{\R^3}g(x-y)K_{k,\leq p}^\iota(y,v)\,dy,
\end{equation}
for some kernels $K_{k,\leq p}^\iota$ that satisfy the bounds
\begin{equation}\label{omegaio3}
\begin{split}
&\sup_{|v|\in [2^{l-1},2^{l+1}]}\|K_{k,\leq p}^\iota(.,v)\|_{L^1_y}\lesssim 1,\qquad \sup_{|v|\in [2^{l-1},2^{l+1}]}\|K_{k,\leq p}^\iota(.,v)\|_{L^\infty_y}\lesssim 2^{2k}2^{p-l},
\end{split}
\end{equation}
for any $l\in\Z$. Moreover, $\mathfrak{1}_{[2^{l-1},2^{l+1}]}(|v|)K_{k,\leq p}^\iota(y,v)\equiv 0$ if $k+l\leq -4$.

As a consequence, for any $q\leq r\in[1,\infty]$ and $g\in L^q(\R^3)$
\begin{equation}\label{omegaio4}
\sup_{|v|\in [2^{l-1},2^{l+1}]}\|\Pi_{\iota,k,\leq p} g(.,v)\|_{L^r_x}\lesssim 2^{(2k+p-l)(1/q-1/r)}\|g\|_{L^q_x}.
\end{equation}

(ii) If $\iota\in\{-1,1\}$, $\kappa\in\{-1,0,1\}$, $k\leq 4$, and $p\in\Z$ then
\begin{equation}\label{omegaio2x}
(\Pi_{\iota,k,p} \omega_{\iota,\kappa}^{-1}g)(x,v)=\int_{\R^3}g(x-y)\widetilde{K}_{k,p}^{\iota,\kappa}(y,v)\,dy,
\end{equation}
for some kernels $\widetilde{K}_{k,p}^{\iota,\kappa}$ that satisfy the bounds
\begin{equation}\label{omegaio3x}
\|\widetilde{K}_{k,p}^{\iota,\kappa}(.,v)\|_{L^1_y}\lesssim 2^{-p},
\end{equation}
for any $v\in\R^3$. As a consequence, for any $q\in[1,\infty]$, $g\in L^q(\R^3)$, and $v\in\R^3$,
\begin{equation}\label{omegaio4x}
\|\Pi_{\iota,k,p} \omega_{\iota,\kappa}^{-1}g(.,v)\|_{L^q_x}\lesssim 2^{-p}\|g\|_{L^q_x}.
\end{equation}
\end{lemma} 

\begin{proof} (i) We have the formula
\begin{equation}\label{omegaio6}
K_{k,\leq p}^\iota(y,v)=\frac{1}{(2\pi)^3}\int_{\R^3}e^{iy\cdot\xi}\varphi_k(\xi)\varphi_{\leq p}(1+\iota v\cdot \xi)\,d\xi.
\end{equation}
Clearly $K_{k,\leq p}^\iota(y,v)=0$ if $2^{k}|v|\leq 1/8$ and $p\leq -4$. On the other hand, if $2^k|v|\geq 1/8$ and $p\leq -4$ then we may assume $v=(v_1,0,0)$, $v_1\in[2^{l-1},2^{l+1}]$, and integrate by parts (using Lemma \ref{tech5}) to see that
\begin{equation*}
\big|K_{k,\leq p}^\iota(y,v)\big|\lesssim 2^{2k+p-l}(1+2^k|y_2|+2^k|y_3|+2^{p-l}|y_1|)^{-4}.
\end{equation*}
The $L^1$ bounds in \eqref{omegaio3} follow as well, and the bounds \eqref{omegaio4} follow by interpolation and H\'{o}lder's inequality.

(ii) We may assume that $v=(v_1,0,0)$, $v_1\geq 0$. Notice that 
\begin{equation}\label{omegaio9}
\widetilde{K}_{k,p}^{\iota,\kappa}(y,v)=\frac{1}{(2\pi)^3}\int_{\R^3}e^{iy\cdot\xi}\varphi_k(\xi)\frac{\varphi_0(2^{-p}(1+\iota v_1\xi_1))}{1+\iota v_1\xi_1+i\kappa|\xi|}\,d\xi.
\end{equation}
Notice that, for any $\alpha_1,\alpha_2,\alpha_3\in[0,4]$ and $\xi$ in the support of the integral in \eqref{omegaio9},
\begin{equation}\label{omegaio9.5}
\Big|\frac{d^{\alpha_1}}{d\xi_1^{\alpha_1}}\frac{d^{\alpha_2}}{d\xi_2^{\alpha_2}}\frac{d^{\alpha_3}}{d\xi_3^{\alpha_3}}\Big(\varphi_k(\xi)\frac{\varphi_0(2^{-p}(1+\iota v_1\xi_1))}{1+\iota v_1\xi_1+i\kappa|\xi|}\Big)\Big|\lesssim (2^p+|\kappa|2^k)^{-1}2^{-k(\alpha_2+\alpha_3)}(2^{-k}+|v_1|2^{-p})^{\alpha_1}.
\end{equation}
Therefore, using integration by parts (Lemma \ref{tech5}), for any $p\in\Z$ we have
\begin{equation}\label{omegaio8}
\big|\widetilde{K}_{k,p}^{\iota,\kappa}(y,v)\big|\lesssim 2^{-p}2^{2k}(2^{-k}+2^{-p}|v_1|)^{-1}\big\{1+2^k|y_2|+2^k|y_3|+(2^{-k}+2^{-p}|v_1|)^{-1}|y_1|\big\}^{-4}.
\end{equation}
In particular $\|\widetilde{K}_{k,p}^{\iota,\kappa}(.,v)\|_{L^1_y}\lesssim 2^{-p}$, and the desired bounds \eqref{omegaio3x}--\eqref{omegaio4x} follow.
\end{proof}

\subsection{Linear dispersive estimates} For $\lambda\in\mathbb{R}$ and multipliers $\mathfrak{m}:\R^3\times\R^3\to\mathbb{C}$ we define the operators
\begin{equation}\label{disper1}
\mathcal{I}_{\mathfrak{m}}(g;\lambda)(x):=\frac{1}{(2\pi)^3}\int_{\R^3\times\R^3}e^{ix\cdot\xi}e^{-i\lambda v\cdot\xi}\widehat{g}(\xi,v)\mathfrak{m}(\xi,v)\,d\xi dv.
\end{equation}

We have the following lemma:

\begin{lemma}\label{disper2}
(i) If $j\in\{1,2,3\}$ and $\lambda\neq 0$ then
\begin{equation}\label{disper3}
\mathcal{I}_{i\xi_j\mathfrak{m}}(g;\lambda)=\frac{1}{\lambda}\big\{\mathcal{I}_{\partial_{v_j}\mathfrak{m}}(g;\lambda)+\mathcal{I}_{\mathfrak{m}}(\partial_{v_j}g;\lambda)\big\}.
\end{equation}

(ii) Let $K$ denote the inverse Fourier transform in $\xi$ of the multiplier $\mathfrak{m}$,
\begin{equation}\label{disper4}
K(y,v):=\frac{1}{(2\pi)^3}\int_{\R^3}\mathfrak{m}(\xi,v)e^{iy\cdot \xi}\,d\xi.
\end{equation}
Then, if $q\in[1,\infty]$,
\begin{equation}\label{disper5}
\begin{split}
\|\mathcal{I}_{\mathfrak{m}}(g;\lambda)\|_{L^q_x}&\lesssim\|g\|_{L^1_vL^q_y}\|K\|_{L^\infty_vL^1_y},\\
|\lambda|^3\|\mathcal{I}_{\mathfrak{m}}(g;\lambda)\|_{L^\infty_x}&\lesssim\|g\|_{X_\lambda}\|K\|_{L^1_yL^\infty_v},
\end{split}
\end{equation}
where, by definition,
\begin{equation}\label{disper5.2}
\|g\|_{X_\lambda}:=\sup_{p\in\R^3}\int_{\R^3}|g(w,p-w/\lambda)|\,dw.
\end{equation}

(iii) For $k\in\Z$ let $\mathfrak{m}_k(\xi,v):=\mathfrak{m}(\xi,v)\varphi_k(\xi)$ and $K_k:=\mathcal{F}^{-1}(\mathfrak{m}_k)$. Then, if $q\in[1,\infty]$,
\begin{equation}\label{disper5.5}
\begin{split}
|\lambda|\|\mathcal{I}_{\mathfrak{m}_k}(g;\lambda)\|_{L^q_x}&\lesssim 2^{-k}\big\{\|g\|_{L^1_vL^q_y}\|\nabla_vK_k\|_{L^\infty_vL^1_y}+\|\nabla_vg\|_{L^1_{v}L^q_y}\|K_k\|_{L^\infty_vL^1_y}\big\},\\
|\lambda|^{4}\|\mathcal{I}_{\mathfrak{m}_k}(g;\lambda)\|_{L^\infty_x}&\lesssim2^{-k}\big\{\|g\|_{X_\lambda}\|\nabla_vK_k\|_{L^1_yL^\infty_v}+\|\nabla_vg\|_{X_\lambda}\|K_k\|_{L^1_yL^\infty_v}\big\}.
\end{split}
\end{equation}
\end{lemma}

\begin{proof} (i) The identities \eqref{disper3} follow from the definition and integration by parts in $v$.

(ii) To prove the bounds \eqref{disper5} we rewrite
\begin{equation}\label{disper6}
\begin{split}
\mathcal{I}_{\mathfrak{m}}(g;\lambda)(x)&=\int_{\R^d\times\R^d}K(y,v)g(x-y-\lambda v,v)\,dydv\\
&=|\lambda|^{-d}\int_{\R^d\times\R^d}K\big(y,\frac{x-y-w}{\lambda}\big)g\big(w,\frac{x-y-w}{\lambda}\big)\,dydw.
\end{split}
\end{equation}
The bounds in \eqref{disper5} follow from these identities.

(iii) To prove \eqref{disper5.5} we write
\begin{equation*}
\mathfrak{m}_k(\xi,v)=\xi_a\cdot \mathfrak{m}_k(\xi,v)\cdot (\xi_a/|\xi|^2)\varphi_{[k-2,k+2]}(\xi),
\end{equation*}
and then combine the identity \eqref{disper3} and the first two bounds in \eqref{disper5}.
\end{proof}

We will also need dispersive bounds on the resonant projectors $\Pi_{\iota,k,\leq p}$ and $\Pi_{\iota,k,p}\omega_{\iota,\kappa}^{-1}$.

\begin{lemma}\label{omegaioLem2}
If $\iota\in\{1,-1\}$, $\kappa\in\{-1,0,1\}$, $k\leq 4$, $p\in\Z$, $|\lambda|\geq 1+2^{-p(1+\delta)}$, and $x\in\R^3$ then
\begin{equation}\label{omegaio2y}
\begin{split}
\int_{|v|\leq |\lambda|^8}\big|(\Pi_{\iota,k,\leq p} g)(x+\lambda v,v)\big|\,dv&\lesssim |\lambda|^{-3}\|g\|_{L^1},\\
\int_{|v|\leq |\lambda|^8}\big|(\Pi_{\iota,k,p} \omega_{\iota,\kappa}^{-1}g)(x+\lambda v,v)\big|\,dv&\lesssim 2^{-p}|\lambda|^{-3}\|g\|_{L^1}.
\end{split}
\end{equation}
\end{lemma} 

\begin{proof} The two estimates are similar, so we focus on the harder bounds in the second line. Using \eqref{omegaio2x} we have
\begin{equation*}
\int_{|v|\leq |\lambda|^8}\big|(\Pi_{\iota,k,p}\omega_{\iota,\kappa}^{-1}g)(x+\lambda v,v)\big|\,dv\leq\int_{\R^3}\int_{|v|\leq |\lambda|^8}|g(y)\widetilde{K}_{k, p}^{\iota,\kappa}(x+\lambda v-y,v)|\,dv dy.
\end{equation*}
To prove the bounds in the second line of \eqref{omegaio2y} it suffices to show that
\begin{equation}\label{omegaio3y}
\sup_{x\in\R^3}\int_{|v|\leq |\lambda|^8}\big|\widetilde{K}_{k,p}^{\iota,\kappa}(x+\lambda v,v)\big|\,dv\lesssim 2^{-p}|\lambda|^{-3}.
\end{equation}

As in \eqref{omegaio9} we write
\begin{equation}\label{omegaio4y}
\widetilde{K}_{k,p}^{\iota,\kappa}(x+\lambda v,v)=\frac{1}{(2\pi)^3}\int_{\R^3}e^{i(x+\lambda v)\cdot\xi}\varphi_k(\xi)\frac{\varphi_0(2^{-p}(1+\iota v_1\xi_1))}{1+\iota v_1\xi_1+i\kappa|\xi|}\,d\xi.
\end{equation}
Let $H_{p}^{\iota,\kappa}$ denote the Fourier transform of part of the integrand, 
\begin{equation}\label{omegaio5y}
H_p^{\kappa}(\xi,\rho):=\int_{\R}e^{-i\rho\beta}\frac{\varphi_0(\beta)}{2^p\beta+i\kappa|\xi|}\,d\beta,
\end{equation}
therefore, for any $\beta\in\R$,
\begin{equation}\label{omegaio6y}
\frac{\varphi_0(\beta)}{2^p\beta+i\kappa|\xi|}=\frac{1}{2\pi}\int_{\R}e^{i\rho\beta}H_p^{\kappa}(\xi,\rho)\,d\rho.
\end{equation}
We substitute this identity into \eqref{omegaio4y}, with $\beta=2^{-p}(1+\iota v\cdot\xi)$, so
\begin{equation}\label{omegaio7y}
\begin{split}
\widetilde{K}_{k, p}^{\iota,\kappa}(x+\lambda v,v)&=\frac{1}{(2\pi)^4}\int_{\R^3}\int_{\R}e^{i(x+\lambda v)\cdot\xi}\varphi_k(\xi)e^{i\rho 2^{-p}(1+\iota v\cdot\xi)}H_p^\kappa(\xi,\rho)\,d\rho d\xi\\
&=\frac{1}{2\pi}\int_{\R}e^{i\rho 2^{-p}}G_{k,p}^\kappa(x+\lambda v+\iota\rho 2^{-p}v,\rho)\,d\rho,
\end{split}
\end{equation}
where, using also the definition \eqref{omegaio5y},
\begin{equation}\label{omegaio8y}
\begin{split}
G_{k,p}^\kappa(y,\rho)&:=\frac{1}{(2\pi)^3}\int_{\R^3}e^{iy\cdot\xi}\varphi_k(\xi)H_p^\kappa(\xi,\rho)\,d\xi\\
&=\frac{1}{(2\pi)^3}\int_{\R^3}\int_{\R}e^{iy\cdot\xi}e^{-i\rho\beta}\frac{\varphi_k(\xi)\varphi_0(\beta)}{2^p\beta+i\kappa|\xi|}\,d\beta d\xi.
\end{split}
\end{equation}

We notice that
\begin{equation}\label{omegaio9y}
\Big|D^a_\beta D^\alpha_\xi\Big(\frac{\varphi_k(\xi)\varphi_0(\beta)}{2^p\beta+i\kappa|\xi|}\Big)\Big|\lesssim_N2^{-p}2^{-|\alpha|k}
\end{equation}
if $|\xi|\in[2^{k-1},2^{k+1}]$, $|\beta|\in[1/2,2]$, $a\leq N$, and $|\alpha|\leq 4$. We can therefore integrate by parts in the formula \eqref{omegaio8y} to see that
\begin{equation}\label{omegaio10y}
|G_{k,p}^\kappa(y,\rho)|\lesssim_N2^{-p}2^{3k}(1+2^{k}|y|)^{-4}(1+|\rho|)^{-N}\qquad \text{ for any }y\in\R^3,\,\rho\in\R.
\end{equation}

We can now prove the bounds \eqref{omegaio3y}. For any $x\in\R^3$ we use the formula \eqref{omegaio7y} to estimate
\begin{equation*}
\int_{|v|\leq |\lambda|^8}\big|\widetilde{K}_{k, p}^{\iota,\kappa}(x+\lambda v,v)\big|\,dv\lesssim \int_{\R}\int_{|v|\leq |\lambda|^8}\big|G_{k,p}^\kappa(x+\lambda v+\iota\rho 2^{-p}v,\rho)\big|\,dvd\rho\lesssim I+II,
\end{equation*}
where
\begin{equation*}
\begin{split}
I&:=\int_{|\rho|\leq 2^p|\lambda|/2}\int_{|v|\leq |\lambda|^8}\big|G_{k,p}^\kappa(x+(\lambda+\iota\rho 2^{-p})v,\rho)\big|\,dvd\rho,\\
II&:=\int_{|\rho|\geq 2^p|\lambda|/2}\int_{|v|\leq |\lambda|^8}\big|G_{k,p}^\kappa(x+(\lambda+\iota\rho 2^{-p})v,\rho)\big|\,dvd\rho.
\end{split}
\end{equation*}

It follows from \eqref{omegaio10y} that $\|G_{k,p}^\kappa(.,\rho)\|_{L^1_y}\lesssim 2^{-p}(1+|\rho|)^{-2}$. Moreover $|\lambda+\iota\rho 2^{-p}|\in[|\lambda|/2,2|\lambda|]$ if $|\rho|\leq 2^p|\lambda|/2$. We can therefore make a suitable change of variables and integrate in $v$ to see that $|I|\lesssim 2^{-p}|\lambda|^{-3}$. In addition, it follows from \eqref{omegaio10y} that $\|G_{k,p}^\kappa(.,\rho)\|_{L^\infty_y}\lesssim_N 2^{-p}(1+|\rho|)^{-N}$. Notice that $2^p|\lambda|\gtrsim |\lambda|^{\delta/2}$, due to the assumption $|\lambda|\geq 1+2^{-p(1+\delta)}$. We can thus take $N$ large enough to see that $|II|\lesssim 2^{-p}|\lambda|^{-3}$ and the desired bounds \eqref{omegaio3y} follow.
\end{proof}

\subsubsection{The $S^\infty L^\infty$ class} To apply dispersive bounds in certain cases, we need to introduce the space $S^\infty L^\infty (\R^d\times\R^d)$ of multipliers $m:\R^d\times\R^d\to\mathbb{C}$ defined by the norm
\begin{equation}\label{loc25}
\|m\|_{S^\infty L^\infty}:=\|\mathcal{F}^{-1}m\|_{L^1_xL^\infty_v}=\Big\|\sup_{v\in\R^d}\Big|\frac{1}{(2\pi)^d}\int_{\R^d}m(\xi,v)e^{ix\cdot\xi}\,d\xi\Big|\Big\|_{L^1_x}.
\end{equation}

\begin{lemma}\label{ProjAB} (i) If $m,m'\in S^\infty L^\infty$ then $mm'\in S^\infty L^\infty$ and
\begin{equation}\label{loc26}
\|mm'\|_{S^\infty L^\infty}\leq \|m\|_{S^\infty L^\infty}\|m'\|_{S^\infty L^\infty}.
\end{equation}

(ii) If $k\in\Z$, and $j\in\Z_+$ then
\begin{equation}\label{loc27}
\|\mathcal{D}\varphi_k(\xi)\widetilde{\varphi}_j(v)\|_{S^\infty L^\infty}\lesssim 2^{j+k},
\end{equation}
\begin{equation}\label{loc28}
\|(\nabla_v\mathcal{D})\varphi_k(\xi)\widetilde{\varphi}_j(v)\|_{S^\infty L^\infty}\lesssim 2^{k}.
\end{equation}
Moreover, if $j+k\leq -10$ then
\begin{equation}\label{loc29}
\big\|(1+\mathcal{D}^2)^{-1}\varphi_k(\xi)\widetilde{\varphi}_j(v)\big\|_{S^\infty L^\infty}\lesssim 1.
\end{equation}
\end{lemma}

\begin{proof}
The bounds \eqref{loc26}, \eqref{loc27}, and \eqref{loc28} follow directly from definitions. To prove \eqref{loc29} we decompose $(1+\mathcal{D}^2)=(\mathcal{D}+i)(\mathcal{D}-i)$ and integrate by parts in $\xi$ to show that
\begin{equation*}
\Big|\int_{\R^3}\frac{1}{\mathcal{D}(\xi,v)\pm i}e^{ix\cdot\xi}\varphi_k(\xi)\widetilde{\varphi}_j(v)\,d\xi\Big|\lesssim \widetilde{\varphi}_j(v)2^{3k}(1+2^k|x|)^{-4}
\end{equation*}
for any $x,v\in\R^3$. The bounds \eqref{loc29} follow.
\end{proof}

\subsection{Pointwise estimates of characteristics}\label{CharBounds}

We prove now estimates on the functions $\widetilde{Y}$ and $\widetilde{W}$ defined in \eqref{Lan6}, which measure the nonlinear deviation of the characteristic flow. 
 
\begin{lemma}\label{derivativeschar}
For any $(s,t)\in \mathcal{I}_T^2=\{(s,t)\in[0,T]^2:\,0\leq s\leq t\leq T\}$ we have
\be\label{nov28eqn2}
  \sup_{x, v\in \R^3}   |\p_t \widetilde{Y}(x,v,s,t)-(t-s) E( x+tv,t)|   \lesssim \varepsilon_1\langle t-s\rangle \langle t\rangle^{-2+\delta}\langle s \rangle^{-1+1.1\delta},
\ee
\be\label{dec5eqn51}
  \sup_{x, v\in \R^3}   |\p_t\widetilde{W}(x,v,s,t)+E( x+tv,t)| \lesssim \varepsilon_1\langle t-s\rangle \langle t\rangle^{-2+\delta}\langle s \rangle^{-2+1.1\delta}. 
\ee
\end{lemma}

\begin{proof}

It follows from \eqref{Lan7} and the assumption $\widetilde{Y}(x,v,t,t)=0$ that
\begin{equation}\label{nov19eqn2}
\begin{split}
\p_t \widetilde{Y}(x,v,s,t)&= (t-s) E( x+tv,t)\\
&+\int_{s}^t (\tau-s) \nabla_x E(x+\tau v +\widetilde{Y}(x,v,\tau,v),\tau )\cdot \p_t \widetilde{Y}(x,v,\tau,t)\,d \tau
\end{split}
\end{equation}
and
\be\label{dec5eqn41}
\p_t\widetilde{W}(x,v,s,t)=-E( x+tv,t)-\int_{s}^t\nabla_x E(x+\tau v +\widetilde{Y}(x,v,\tau,v),\tau )\cdot \p_t \widetilde{Y}(x,v,\tau,t)\,d \tau.
\ee

For $t\in[0,T]$ let 
\[
Z_0(t):=\sup_{s\in [0, t]} \sup_{x, v\in \R^3} ({t-s})^{-1} |\p_t \widetilde{Y}(x,v,s,t)|. 
\]
Using \eqref{nov19eqn2} and Lemma \ref{Laga10} (ii) we have
\begin{equation*}
\begin{split}
Z_0(t)&\lesssim  \varepsilon_1 \langle t \rangle^{-2+\delta}  + \frac{\varepsilon_1}{t-s}  \int_{s}^t (\tau-s) \langle \tau\rangle^{-3+2\delta} (t-\tau) Z_0(t) d \tau\\
&\lesssim \varepsilon_1 \langle t \rangle^{-2+\delta}  + \varepsilon_1 \int_{s}^{t}  \langle \tau\rangle^{-2+2\delta} Z_0(t) d \tau \lesssim  \varepsilon_1 \langle t \rangle^{-2+\delta}  + \varepsilon_1 Z_0(t).
\end{split}
\end{equation*}
Therefore $Z_0(t)\lesssim \varepsilon_1 \langle t \rangle^{-2+\delta}$. Using again \eqref{nov19eqn2} and Lemma \ref{Laga10} (ii) we have
\begin{equation*}
\begin{split}
|\p_t \widetilde{Y}(x,v,s,t)-(t-s) E( x+tv,t)|&\lesssim\int_{s}^t\varep_1\langle \tau\rangle^{-3+1.1\delta}\langle\tau-s\rangle\langle t-\tau\rangle \langle t\rangle^{-2+\delta}\,d \tau\\
&\lesssim\varepsilon_1 \langle t-s\rangle \langle t\rangle^{-2+\delta}\langle s \rangle^{-1+1.1\delta},
\end{split}
\end{equation*}
which gives the desired estimate \eqref{nov28eqn2}. The estimates \eqref{dec5eqn51} follow in a similar way.
\end{proof}

\begin{lemma}\label{derichar2}
For any $(s,t)\in\mathcal{I}^2_T$ and $x,v\in \R^3$ we have
\begin{equation}\label{cui6}
\begin{split}
|\widetilde{Y}(x,v,s,t)|&\lesssim    \varep_1\min\{\langle s\rangle^{-1+2\delta}\langle v\rangle, \langle s\rangle^{-1/6}\},\\
|\widetilde{W}(x,v,s,t)|&\lesssim    \varep_1\min\{\langle s\rangle^{-2+2\delta}\langle v\rangle, \langle s\rangle^{-7/6}\},
\end{split}
\end{equation}
\begin{equation}\label{cui7}
\begin{split}
|\nabla_x\widetilde{Y}(x,v,s,t)|&\lesssim \varep_1\min\{\langle s\rangle^{-2+2.1\delta}\langle v\rangle, \langle s\rangle^{-7/6}\},\\
|\nabla_x\widetilde{W}(x,v,s,t)|&\lesssim \varep_1\min\{\langle s\rangle^{-3+2.1\delta}\langle v\rangle, \langle s\rangle^{-13/6}\},
\end{split}
\end{equation}
\begin{equation}\label{cui7.5}
\begin{split}
|\nabla_v\widetilde{Y}(x,v,s,t)|&\lesssim \varep_1\min\{\langle s\rangle^{-1+2.1\delta}\langle v\rangle, \langle s\rangle^{-1/6}\},\\
|\nabla_v\widetilde{W}(x,v,s,t)|&\lesssim \varep_1\min\{\langle s\rangle^{-2+2.1\delta}\langle v\rangle, \langle s\rangle^{-7/6}\}.
\end{split}
\end{equation}
\end{lemma}

\begin{proof} {\bf{Step 1.}} To prove the bounds \eqref{cui6} we use the decomposition \eqref{Laga5} of the electric field, and rewrite the identity in the first line of \eqref{Lan7} in the form
\begin{equation}\label{uui1}
\begin{split}
\widetilde{W}(x,v,s,t)&= \widetilde{W}^{stat}(x,v,s,t;t)+ \widetilde{W}^{osc}(x,v,s,t;t),\\
\widetilde{W}^{stat}(x,v,s_1,s_2;t)&:=-\int_{s_1}^{s_2}\big[E^{stat}+\Re(e^{-i\tau}P_{\geq 0}E^{osc})\big](x+\tau v + \widetilde{Y}(x,v,\tau,t),\tau)\,d \tau,\\
\widetilde{W}^{osc} (x,v,s_1,s_2;t)&:=-\int_{s_1}^{s_2}\Re(e^{-i\tau}P_{<0}E^{osc})(x+\tau v + \widetilde{Y}(x,v,\tau,t),\tau)\,d \tau.
\end{split}
\end{equation}

In view of \eqref{Laga3} we have $\|P_{\geq 0}E^{osc}(\tau)\|_{L^\infty}\lesssim\varep_1\langle\tau\rangle^{-4+2\delta}$. Thus, using also \eqref{Laga11}, if $m\geq 0$ and $s_1\leq s_2\in[2^m-1,2^{m+1}]\cap[0,t]$ then
\begin{equation}\label{uui2}
|\widetilde{W}^{stat}(x,v,s_1,s_2;t)|\lesssim \int_{s_1}^{s_2}\varepsilon_1\langle \tau \rangle^{-3+2\delta}\,d \tau\lesssim  \varep_1 2^{-2m+2\delta m}. 
\end{equation}

To bound $|\widetilde{W}^{osc}|$ we integrate by parts in $\tau$. Notice that $\partial_\tau\widetilde{Y}(x,v,\tau,t)=\widetilde{W}(x,v,\tau,t)$, see \eqref{Lan7.5}, and $|\widetilde{W}(x,v,\tau,t)|\lesssim \varep_1\langle\tau\rangle^{-1+\delta}$ (due to \eqref{Lan7} and \eqref{Laga11}--\eqref{Laga12}). Therefore
\begin{equation}\label{uui3}
\begin{split}
|\widetilde{W}^{osc}(x,v,s_1,s_2;t)|&\lesssim \int_{s_1}^{s_2}\big\{\|(\partial_\tau P_{<0}E^{osc})(\tau)\|_{L^\infty}+\|(\nabla_xP_{<0}E^{osc})(\tau)\|_{L^\infty}\\
&\qquad\quad\,\,\times(|v|+|\widetilde{W}(x,v,\tau,t)|)\big\}\,d\tau+\sum_{\tau\in\{s_1,s_2\}}\|(P_{<0}E^{osc})(\tau)\|_{L^\infty}\\
&\lesssim\varep_12^{-2m+2\delta m}\langle v\rangle, 
\end{split}
\end{equation}
using \eqref{Laga12} again in the last line.

We prove now that if $m\geq 0$, $|v|\geq 2^{4m/5+10}$, and $s_1\leq s_2\in[2^m-1,2^{m+1}]\cap[0,t]$ then
\begin{equation}\label{uui4}
|\widetilde{W}^{osc} (x,v,s_1,s_2;t)|\lesssim\varep_12^{-7m/6}.
\end{equation}
Indeed, we decompose
\begin{equation}\label{uui5}
\begin{split}
&\widetilde{W}^{osc} (x,v,s_1,s_2;t)=-\sum_{k\leq -1}\Re I^{(1)}_k(x,v,s_1,s_2;t),\\
&I^{(1)}_k(x,v,s_1,s_2;t):=\int_{s_1}^{s_2}e^{-i\tau}(P_{k}E^{osc})(x+\tau v + \widetilde{Y}(x,v,\tau,t),\tau)\,d \tau.
\end{split}
\end{equation}
Using \eqref{Laga3} we have $\|P_kE^{osc}(\tau)\|_{L^\infty}\lesssim\varep_1\min\{\langle\tau\rangle^{-3},2^{3k}\}\langle\tau\rangle^{-1+2\delta}2^{-2k}$ for any $k\leq 0$, thus
\begin{equation}\label{uui6}
\sum_{k\leq -6m/5+60}|I^{(1)}_k(x,v,s_1,s_2;t)|\lesssim \sum_{k\leq -6m/5+60}\varep_12^{2\delta m}2^{k}\lesssim \varep_12^{-6m/5+2\delta m},
\end{equation}
for any $s_1\leq s_2\in[2^{m}-1,2^{m+1}]\cap [0,t]$, consistent with the desired bounds in \eqref{uui4}. 

On the other hand, if $k\in[-6m/5+60,0]$ (in particular $m\geq 50$) then we write
\begin{equation*}
e^{-i\tau}(P_{k}E^{osc})(x+\tau v + \widetilde{Y}(x,v,\tau,t),\tau)=\frac{1}{(2\pi)^3}\int_{\R^3}e^{-i\tau}\varphi_k(\xi)\widehat{E^{osc}}(\xi,\tau)e^{i\xi\cdot (x+\tau v + \widetilde{Y}(x,v,\tau,t))}\,d\xi.
\end{equation*}
We insert cutoff functions of the form $\varphi_{\leq p_0}(1-\xi\cdot v)$ and $\varphi_{\leq p_0}(1-\xi\cdot v)$, and decompose
\begin{equation}\label{uui7}
\begin{split}
I^{(1)}_k&=J^{(1)}_{k,\leq p_0}+J^{(1)}_{k,>p_0},\\
J^{(1)}_{k,\ast}(x,v,s_1,s_2;t)&:=\frac{1}{(2\pi)^3}\int_{s_1}^{s_2} \int_{\R^3}\varphi_{\ast}(1-\xi\cdot v)e^{-i\tau(1-\xi\cdot v)}\\
&\qquad\qquad\qquad\qquad\times\varphi_k(\xi)\widehat{E^{osc}}(\xi,\tau)e^{i\xi\cdot (x+\widetilde{Y}(x,v,\tau,t))}\,d\xi d\tau,
\end{split}
\end{equation}
where $p_0:=-2m/5$ and $\ast\in\{\leq p_0,>p_0\}$. We have $\|\widehat{P_kE^{osc}}(\tau)\|_{L^\infty}\lesssim \|P_kE^{osc}(\tau)\|_{L^1}\lesssim \varep_12^{-2k}\langle \tau\rangle^{-1+2\delta}$, using \eqref{Laga3}. Recalling that $|v|\geq 2^{4m/5+10}$, we have
\begin{equation}\label{uui8}
\big|J^{(1)}_{k,\leq p_0}(x,v,s_1,s_2;t)\big|\lesssim \int_{s_1}^{s_2} \int_{\R^3}|\varphi_{\leq p_0}(1-\xi\cdot v)||\varphi_k(\xi)||\widehat{E^{osc}}(\xi,\tau)|\,d\xi d\tau\lesssim \varep_1 2^{-6m/5+2\delta m}.
\end{equation}

On the other hand, to bound $|J^{(1)}_{k,>p_0}|$ we integrate by parts in $\tau$ and estimate
\begin{equation*}
\begin{split}
|J^{(1)}_{k,>p_0}&(x,v,s_1,s_2;t)|\lesssim\negmedspace\sum_{\tau\in \{s_1,s_2\}}\negmedspace\Big|\int_{\R^3}\frac{\varphi_{>p_0}(1-\xi\cdot v)}{1-\xi\cdot v}e^{-i\tau(1-\xi\cdot v)}\varphi_k(\xi)\widehat{E^{osc}}(\xi,\tau)e^{i\xi\cdot (x+\widetilde{Y}(x,v,\tau,t))}\,d\xi\Big|\\
&+\int_{s_1}^{s_2} \Big|\int_{\R^3}\frac{\varphi_{>p_0}(1-\xi\cdot v)}{1-\xi\cdot v}e^{-i\tau(1-\xi\cdot v)}\varphi_k(\xi)\frac{d}{d\tau}\big\{\widehat{E^{osc}}(\xi,\tau)e^{i\xi\cdot (x+\widetilde{Y}(x,v,\tau,t))}\big\}\,d\xi\Big| d\tau.
\end{split}
\end{equation*}
In view of \eqref{omegaio4x}, for $F\in\{E^{osc}(\tau),\,(\partial_\tau E^{osc})(\tau)\}$ we have
\begin{equation*}
\Big|\int_{\R^3}\frac{\varphi_{>p_0}(1-\xi\cdot v)}{1-\xi\cdot v}e^{-i\tau(1-\xi\cdot v)}\varphi_k(\xi)\widehat{F}(\xi)e^{i\xi\cdot (x+\widetilde{Y}(x,v,\tau,t))}\,d\xi\Big|\lesssim 2^{-p_0}\|P_kF\|_{L^\infty}.
\end{equation*}
Since $\|P_kE^{osc}(\tau)\|_{L^\infty}\lesssim \varep_1\langle \tau\rangle^{-2+2\delta}$ and $\|P_k(\partial_\tau E^{osc})(\tau)\|_{L^\infty}\lesssim \varep_1\langle \tau\rangle^{-3+2\delta}$ (due to \eqref{Laga3}--\eqref{Laga4}) and $|\partial_\tau\widetilde{Y}(x,v,\tau,t)|=|\widetilde{W}(x,v,\tau,t)|\lesssim \varep_1\langle\tau\rangle^{-1+\delta}$, it follows from the last two estimates that
\begin{equation}\label{uui9}
|J^{(1)}_{k,>p_0}(x,v,s_1,s_2;t)|\lesssim \varep_12^{-p_0}2^{-2m+4\delta m}
\end{equation}
for any $s_1\leq s_2\in[2^{m}-1,2^{m+1}]\cap [0,t]$. 

Recalling that $p_0=-2m/5$ and using \eqref{uui7}--\eqref{uui8} it follows that $|I^{(1)}_k(x,v,s_1,s_2;t)|\lesssim \varep_12^{-6m/5+2\delta m}$. The desired bounds \eqref{uui4} follow using also \eqref{uui6}. 

We combine now the bounds \eqref{uui2}--\eqref{uui4} and divide the interval $[s,t]$ dyadically to conclude that $|\widetilde{W}(x,v,s,t)|\lesssim    \varep_1\min\{\langle s\rangle^{-2+2\delta}\langle v\rangle, \langle s\rangle^{-7/6}\}$ for any $s\leq t\in[0,T]$. This gives the bounds in the second line of \eqref{cui6}. The bounds in the first line then follow using the identity $\partial_s\widetilde{Y}=\widetilde{W}$ and integrating on the interval $[s,t]$.

{\bf{Step 2.}} To prove the bounds \eqref{cui7} we define, for any $t\in[0,T]$,
\begin{equation}\label{uui0}
Z_1(t):=\sup_{s\in[0,t]} \sup_{x,v\in \R^3} \big[\langle v \rangle^{-1}\langle s \rangle^{2-2.1\delta}+\langle s\rangle^{7/6}\big] \big[|\nabla_x\widetilde{Y}(x,v,s,t)|+\langle s\rangle\nabla_x\widetilde{W}(x,v,s,t)|\big].
\end{equation}
In view of \eqref{Lan7} we have
\begin{equation*}
\partial_{x^a}\widetilde{W}(x,v,s,t)=-\int_{s}^t (\partial_{x^b} E)(x+\tau v + \widetilde{Y}(x,v,\tau,t),\tau)(\delta_{ab}+ \partial_{x^a} \widetilde{Y}_b(x,v,\tau,t))\,d \tau,
\end{equation*}
so we can decompose, as in \eqref{uui1}
\begin{equation}\label{uui20}
\begin{split}
\partial_{x^a}\widetilde{W}(x,v,s,t)&= \widetilde{W}^{stat}_{x,1}(x,v,s,t;t)+\widetilde{W}^{stat}_{x,2}(x,v,s,t;t) +\widetilde{W}^{osc}_x(x,v,s,t;t),\\
\widetilde{W}^{stat}_{x,1}(x,v,s_1,s_2;t)&:=-\int_{s_1}^{s_2}\big[\partial_{x^b}E^{stat}+\Re(e^{-i\tau}\partial_{x^b}P_{\geq 0}E^{osc})\big](x+\tau v + \widetilde{Y}(x,v,\tau,t),\tau)\\
&\qquad\qquad\qquad\times (\delta_{ab}+ \partial_{x^a} \widetilde{Y}_b(x,v,\tau,t))\,d \tau,\\
\widetilde{W}^{stat}_{x,2}(x,v,s_1,s_2;t)&:=-\int_{s_1}^{s_2}\Re(e^{-i\tau}P_{<0}\partial_{x^b}E^{osc})(x+\tau v + \widetilde{Y}(x,v,\tau,t),\tau)\partial_{x^a} \widetilde{Y}_b(x,v,\tau,t)\,d \tau,\\
\widetilde{W}^{osc}_x(x,v,s_1,s_2;t)&:=-\int_{s_1}^{s_2}\Re(e^{-i\tau}P_{<0}\partial_{x^a}E^{osc})(x+\tau v + \widetilde{Y}(x,v,\tau,t),\tau)\,d \tau.
\end{split}
\end{equation}

Assume that $m\geq 0$ and $s_1\leq s_2\in[2^m-1,2^{m+1}]\cap[0,t]$. Since $|\nabla_x\widetilde{Y}(x,v,\tau,t)|\leq Z_1(t)\langle \tau\rangle^{-7/6}$ it follows from \eqref{Laga2}--\eqref{Laga3} that
\begin{equation}\label{uui21}
|\widetilde{W}^{stat}_{x,1}(x,v,s_1,s_2;t)|+|\widetilde{W}^{stat}_{x,2}(x,v,s_1,s_2;t)|\lesssim \varep_12^{-3m+2.1\delta m}+\varep_1Z_1(t)2^{-3m}.
\end{equation}
Moreover, as in \eqref{uui3}, we integrate by parts in $\tau$ to show that
\begin{equation}\label{uui22}
|\widetilde{W}^{osc}_x(x,v,s_1,s_2,t)|\lesssim\varep_12^{-3m+2.1\delta m}\langle v\rangle.
\end{equation}
Finally, one can estimate as in the proof of \eqref{uui4} to show that if $|v|\geq 2^{4m/5+10}$ then
\begin{equation}\label{uui23}
|\widetilde{W}^{osc}_x(x,v,s_1,s_2;t)|\lesssim\varep_12^{-13m/6}.
\end{equation}

We combine \eqref{uui21}--\eqref{uui23} and sum over integers $m$ satisfying $2^m\gtrsim \langle s\rangle$ to conclude that
\begin{equation*}
|\nabla_x\widetilde{W}(x,v,s,t)|\lesssim \varep_1\min\{\langle s\rangle^{-3+2.1\delta}\langle v\rangle,\langle s\rangle^{-13/6}\}+\varep_1Z_1(t)\langle s\rangle^{-3}.
\end{equation*}
Using \eqref{Lan7.5} it follows that
\begin{equation}\label{uui24}
\langle s\rangle|\nabla_x\widetilde{W}(x,v,s,t)|+|\nabla_x\widetilde{Y}(x,v,s,t)|\lesssim \varep_1\min\{\langle s\rangle^{-2+2.1\delta}\langle v\rangle,\langle s\rangle^{-7/6}\}+\varep_1Z_1(t)\langle s\rangle^{-2}.
\end{equation}
In particular, using the definition \eqref{uui0}, $Z_1(t)\lesssim\varep_1$, and the desired bounds \eqref{cui7} follow.

{\bf{Step 3.}} The proof of the bounds \eqref{cui7.5} is similar. For $t\in [0,T]$ we define
\begin{equation}\label{uui30}
Z_2(t):=\sup_{s\in[0,t]} \sup_{x,v\in \R^3} \big[\langle v \rangle^{-1}\langle s \rangle^{1-2.1\delta}+\langle s\rangle^{1/6}\big] \big[|\nabla_v\widetilde{Y}(x,v,s,t)|+\langle s\rangle\nabla_v\widetilde{W}(x,v,s,t)|\big].
\end{equation}

Using the formulas \eqref{Lan7} and the decomposition \eqref{Laga5} we write
\be\label{feb15eqn32}
\partial_{v^a}\widetilde{W}(x,v,s,t)= \widetilde{W}^{stat}_{v,1}(x,v,s,t;t)+\widetilde{W}^{stat}_{v,2}(x,v,s,t;t) +\widetilde{W}^{osc}_v(x,v,s,t;t),
\ee 
where
\begin{equation*}
\begin{split}
\widetilde{W}^{stat}_{v,1}(x,v,s_1,s_2;t)&:=-\int_{s_1}^{s_2}\big[\partial_{x^b}E^{stat}+\Re(e^{-i\tau}\partial_{x^b}P_{\geq 0}E^{osc})\big](x+\tau v + \widetilde{Y}(x,v,\tau,t),\tau)\\
&\qquad\qquad\qquad\times (\tau\delta_{ab}+ \partial_{v^a} \widetilde{Y}_b(x,v,\tau,t))\,d \tau,\\
\widetilde{W}^{stat}_{v,2}(x,v,s_1,s_2;t)&:=-\int_{s_1}^{s_2}\Re(e^{-i\tau}P_{<0}\partial_{x^b}E^{osc})(x+\tau v + \widetilde{Y}(x,v,\tau,t),\tau)\partial_{v^a} \widetilde{Y}_b(x,v,\tau,t)\,d \tau,\\
\widetilde{W}^{osc}_v(x,v,s_1,s_2;t)&:=-\int_{s_1}^{s_2}\tau\Re(e^{-i\tau}P_{<0}\partial_{x^a}E^{osc})(x+\tau v + \widetilde{Y}(x,v,\tau,t),\tau)\,d \tau.
\end{split}
\end{equation*}

As in \eqref{uui21}, it follows from \eqref{Laga2}--\eqref{Laga3} that
\begin{equation*}
|\widetilde{W}^{stat}_{v,1}(x,v,s_1,s_2;t)|+|\widetilde{W}^{stat}_{v,2}(x,v,s_1,s_2;t)|\lesssim \varep_12^{-2m+2.1\delta m}+\varep_1Z_2(t)2^{-2m},
\end{equation*}
provided that $m\geq 0$ and $s_1\leq s_2\in[2^m-1,2^{m+1}]\cap[0,t]$. Moreover, as in \eqref{uui22}--\eqref{uui23}, we can integrate by parts in $\tau$ to show that
\begin{equation*}
|\widetilde{W}^{osc}_v(x,v,s_1,s_2,t)|\lesssim\varep_1\min\{2^{-2m+2.1\delta m}\langle v\rangle,2^{-7m/6}\}.
\end{equation*}
As before we combine these bounds, sum over integers $m$ satisfying $2^m\gtrsim \langle s\rangle$, and use \eqref{Lan7.5} to conclude that
\begin{equation*}
\langle s\rangle|\nabla_v\widetilde{W}(x,v,s,t)|+|\nabla_v\widetilde{Y}(x,v,s,t)|\lesssim \varep_1\min\{\langle s\rangle^{-1+2.1\delta}\langle v\rangle,\langle s\rangle^{-1/6}\}+\varep_1Z_2(t)\langle s\rangle^{-1}.
\end{equation*}
In particular $Z_2(t)\lesssim\varep_1$, using the definition \eqref{uui30}, and the desired bounds \eqref{cui7.5} follow.
\end{proof}

\section{The contributions of the initial data}\label{InitDataCon}

 In this section we bound the contributions of the initial data $f_0$. These contributions are the terms $R^\ast_{1,j,k}$ and $T^\ast_{1,j,k}$, $\ast\in\{I,II\}$, in the decomposition \eqref{sug31}--\eqref{sug34}, originating from the term $L_1$ defined in \eqref{qwp1}. Our main result in this section is the following:

\begin{proposition}\label{closeboot1}
With the notation in \eqref{bvn0}--\eqref{bvn2}, we have
\begin{equation}\label{bvn3}
\Vert \rho^{stat}_{1,I}\Vert_{Stat_\delta}+\Vert \rho^{stat}_{1,II}\Vert_{Stat_\delta}+\Vert \rho^{osc}_{1,I}\Vert_{Osc_\delta}+\Vert \rho^{osc}_{1,II}\Vert_{Osc_\delta}\lesssim\varep_0.
\end{equation}
\end{proposition}

We would like to use the defining formulas \eqref{deco1} and \eqref{sug13.2}--\eqref{sug13.4} and Lemma \ref{disper2} (with $d=3$). For this we need several bounds on the functions $L_{1,j,k}$.

\begin{lemma} \label{L1kjBou}
With $L_{1,j,k}=L_{1,j,k}(x,v,s)$ defined as in \eqref{qwp1}--\eqref{qwp2}, we have
\begin{equation}\label{bvn8}
\begin{split}
&\sup_{s\in[0,T]}\big\{\big\|\partial_x^\alpha\partial_v^\beta L_{1,j,k}(s)\big\|_{L^\infty_vL^1_x}
+\langle s\rangle^{1-\delta}\big\|(\partial_sL_{1,j,k})(s)\big\|_{L^\infty_vL^1_x}\big\}\lesssim\varep_0 2^{-4.5j},\\
&\sup_{s\in[0,T]}\big\{\big\|\partial_x^\alpha\partial_v^\beta L_{1,j,k}(s)\big\|_{L^\infty_vL^\infty_x}+\langle s\rangle^{1-\delta}\big\|(\partial_sL_{1,j,k})(s)\big\|_{L^\infty_vL^\infty_x}\big\}\lesssim\varep_0 2^{-4.5j},\\
&\sup_{s\in[1,T]}\big\{\big\|\partial_x^\alpha\partial_v^\beta L_{1,j,k}(s)\big\|_{X_s}+\langle s\rangle^{1-\delta}\big\|(\partial_sL_{1,j,k})(s)\big\|_{X_s}\big\}\lesssim\varep_0 2^{-4.5j},
\end{split}
\end{equation}
for any $(j,k)\in\Z_+\times\Z$ and multi-indices $\alpha,\beta$ with $|\alpha|+|\beta|\leq 1$.
\end{lemma}

\begin{proof} We use the initial-data assumptions \eqref{bootinit}. Let
\begin{equation}\label{bvn8.1}
F_0(x,v):=|f_0(x,v)|+|\nabla_xf_0(x,v)|+|\nabla_vf_0(x,v)|,\qquad F_0^\ast(x):=\sup_{v\in\R^3} |F_0(x,v)|\langle v\rangle^{4.5}.
\end{equation}
Notice that, for any $(x,v,s)\in\R^3\times\R^3\times[0,T]$ we have
\begin{equation*}
L_{1,j,k}(x,v,s)=\widetilde{\varphi}_j(v)\int_{\R^3}L_1(x-y,v,s)K_k(y)\,dy,
\end{equation*}
where $K_k:=\mathcal{F}^{-1}(\varphi_k)$. Thus
\begin{equation}\label{bvn8.2}
\begin{split}
|\partial_x^\alpha\partial_v^\beta &L_{1,j,k}(x,v,s)|+\langle s\rangle^{1-\delta}|\partial_s L_{1,j,k}(x,v,s)|\\
&\lesssim \widetilde{\varphi}_{[j-2,j+2]}(v)\sum_{\mathcal{T}\in\{Id,\nabla_x,\nabla_v,\langle s\rangle^{1-\delta}\partial_s\}}\int_{\R^3}|\mathcal{T}L_1(x-y,v,s)||K_k(y)|\,dy.
\end{split}
\end{equation}
We examine the formula in the first line of \eqref{qwp1} and use the bounds \eqref{nov28eqn2}, \eqref{dec5eqn51}, \eqref{cui7}, and \eqref{cui7.5}. 
It follows that
\begin{equation*}
|\mathcal{T}L_1(z,v,s)|\lesssim F_0(z+\widetilde{Y}(z,v,0,s),v+\widetilde{W}(z,v,0,t))\lesssim F_0^\ast(z+\widetilde{Y}(z,v,0,s))\langle v\rangle^{-4.5},
\end{equation*}
for any $\mathcal{T}\in\{Id,\nabla_x,\nabla_v,\langle s\rangle^{1-\delta}\partial_s\}$. Using also \eqref{bvn8.2} we have
\begin{equation}\label{bvn8.3}
\begin{split}
|\partial_x^\alpha\partial_v^\beta &L_{1,j,k}(x,v,s)|+\langle s\rangle^{1-\delta}|\partial_s L_{1,j,k}(x,v,s)|\\
&\lesssim \widetilde{\varphi}_{[j-2,j+2]}(v)\langle v\rangle^{-4.5}\int_{\R^3}F_0^\ast(x-y+\widetilde{Y}(x-y,v,0,s))|K_k(y)|\,dy.
\end{split}
\end{equation}

The assumptions \eqref{bootinit} show that $\|F_0^\ast\|_{L^\infty_x}+\|F_0^\ast\|_{L^1_x}\lesssim\varep_0$. The $L^\infty_vL^\infty_x$ bounds in the second line of \eqref{bvn8} then follow. 
The $L^\infty_vL^1_x$ bounds in the first line also follow, once we notice that the mapping $x\mapsto x-y+\widetilde{Y}(x-y,v,0,s)$ is a global change of coordinates on $\R^3$ with Jacobian $\approx 1$ for any $y,v\in\R^3$, due to the bounds \eqref{cui7}. 

Finally, the bounds on the $X_s$ norms in the third line of \eqref{bvn8} also follow using the definition \eqref{disper5.2} and the observation that the mapping $x\mapsto x-y+\widetilde{Y}(x-y,p-x/s,0,s)$ is a global change of coordinates on $\R^3$ with Jacobian $\approx 1$ for any $y,p\in\R^3$, $s\geq 1$, due to the bounds \eqref{cui7}--\eqref{cui7.5}. This completes the proof of the lemma.
\end{proof}

\noindent \textit{Proof of Proposition \ref{closeboot1}.} In view of the definitions, it suffices to prove that for $\ast\in\{I,II\}$, $(j,k,m)\in A^\ast$, $t\in[2^{m-1},2^{m+1}]$ (or $t\in[0,2]$ if $m=0$) we have
\begin{align}
2^{m(1-2\delta)}2^{k^+}\big[\|R^{\ast}_{1,j,k}(t)\|_{L^1_x}+2^{3m}\|R^{\ast}_{1,j,k}(t)\|_{L^\infty_x}\big]&\lesssim\varep_02^{-\delta^2j},\label{bvn4}\\
(2^{-\delta m}+2^{m(1-2\delta)}2^{k})\big[\|T^{\ast}_{1,j,k}(t)\|_{L^1_x}+2^{3m}\|T^{\ast}_{1,j,k}(t)\|_{L^\infty_x}\big]&\lesssim\varep_02^{-\delta^2j},\label{bvn5}\\
2^{m(1-2\delta)}\big[\|(\partial_t+|\nabla|)T^{\ast}_{1,j,k}(t)\|_{L^1_x}+2^{3m}\|(\partial_t+|\nabla|) T^{\ast}_{1,j,k}(t)\|_{L^\infty_x}\big]&\lesssim\varep_02^{-\delta^2j}.\label{bvn6}
\end{align}
We prove these bounds in several steps.

{\bf{Proof of \eqref{bvn4}.}} Assume first that $\ast=I$, so we start from the formula \eqref{deco1}. We use also Lemma \ref{disper2} with $d=3$ and $m(\xi,v)=\varphi_{[k-4,k+4]}(\xi)\widetilde{\varphi}_{[j-4,j+4]}(v)$, $\|m\|_{S^\infty L^\infty}\lesssim 1$. Thus
\begin{equation}\label{bvn9}
\begin{split}
2^{k^+}\|R^I_{1,j,k}(t)\|_{L^1_x}&\lesssim \|L_{1,j,k}(t)\|_{L^1_xL^1_v}+\|\nabla_xL_{1,j,k}(t)\|_{L^1_xL^1_v}\lesssim\varepsilon_02^{-j},\\
2^{k^+}\|R^I_{1,j,k}(t)\|_{L^\infty_x}&\lesssim \|L_{1,j,k}(t)\|_{L^1_vL^\infty_x}+\|\nabla_xL_{1,j,k}(t)\|_{L^1_vL^\infty_x}\lesssim\varepsilon_02^{-j},\\
\end{split}
\end{equation}
using the bounds \eqref{disper5} and Lemma \ref{L1kjBou}. These bounds suffice if $2^m\lesssim 1$.

 On the other hand, if $m\geq \delta^{-6}$, after using \eqref{disper5}, \eqref{disper5.5}, and Lemma \ref{L1kjBou}, we have 
\begin{equation}\label{bvn10}
\begin{split}
(1+2^k2^m)\|R^I_{1,j,k}(t)\|_{L^1_x}&\lesssim \|L_{1,j,k}(t)\|_{L^1_xL^1_v}+\|\nabla_vL_{1,j,k}(t)\|_{L^1_xL^1_v}\lesssim\varepsilon_02^{-4j/3},\\
(1+2^k2^m)2^{3m}\|R^I_{1,j,k}(t)\|_{L^\infty_x}&\lesssim \|L_{1,j,k}(t)\|_{X_t}+\|\nabla_vL_{1,j,k}(t)\|_{X_t}\lesssim\varepsilon_02^{-4j/3},\\
\end{split}
\end{equation}
  The desired bounds \eqref{bvn4} follow for $\ast=I$, since $2^{4j/3}(1+2^{k+m})\gtrsim 2^{\delta^2j}2^{k^+}2^{m(1-2\delta)}$ for any $(j,k,m)\in A^I$ with $m\geq \delta^{-6}$.

Assume now that $\ast=II$, so we start from the formula \eqref{sug13.2}. We use also Lemma \ref{disper2} with $d=3$ and
\begin{equation*}
m'(\xi,v)=\varphi_{[k-4,k+4]}(\xi)\widetilde{\varphi}_{[j-4,j+4]}(v)\frac{\mathcal{D}(\xi,v)^2}{1+\mathcal{D}(\xi,v)^2},
\end{equation*}
which satisfies $\|m'\|_{S^\infty L^\infty}+\|\nabla_vm'\|_{S^\infty L^\infty}\lesssim 2^{2j+2k}$ if $j+k\leq -20$ due to Lemma \ref{ProjAB}. Thus, if $(j,k,m)\in A^{II}$ (in particular $j+k\leq 0$), we have 
\begin{equation*}
\begin{split}
(1+2^k2^m)\|R^{II}_{1,j,k}(t)\|_{L^1_x}&\lesssim 2^{k+j}\big[\|L_{1,j,k}(t)\|_{L^1_xL^1_v}+\|\nabla_vL_{1,j,k}(t)\|_{L^1_xL^1_v}\big]\lesssim\varepsilon_02^{-4j/3}2^{k+j},\\
(1+2^k2^m)2^{3m}\|R^{II}_{1,j,k}(t)\|_{L^\infty_x}&\lesssim 2^{k+j}\big[\|L_{1,j,k}(t)\|_{X_t}+\|\nabla_vL_{1,j,k}(t)\|_{X_t}\big]\lesssim\varepsilon_02^{-4j/3}2^{k+j}.
\end{split}
\end{equation*}
Since $2^{j/3}2^{-k}(1+2^{k+m})\gtrsim 2^{\delta^2j}2^{k^+}2^{m(1-2\delta)}$ for any $(j,k,m)\in A^{II}$, the desired bounds \eqref{bvn4} follow if $\ast=II$. 

{\bf{Proof of \eqref{bvn5} when $\ast=I$.}} Using \eqref{disper5} and Lemma \ref{L1kjBou} we have
\begin{equation*}
\begin{split}
2^{k^+}\Big\|\frac{1}{(2\pi)^3}\int_{\R^3}\int_{\R^3}e^{ix\cdot\xi}e^{-(t-s)|\xi|}&e^{-isv\cdot\xi}\widehat{L_{1,j,k}}(\xi,v,s)\,d\xi dv\Big\|_{L^1_x}\\
&\lesssim \|L_{1,j,k}(s)\|_{L^1_xL^1_v}+\|\nabla_xL_{1,j,k}(s)\|_{L^1_xL^1_v}\lesssim\varepsilon_02^{-j}
\end{split}
\end{equation*}
and
\begin{equation*}
\begin{split}
2^{k^+}\Big\|\frac{1}{(2\pi)^3}\int_{\R^3}\int_{\R^3}e^{ix\cdot\xi}e^{-(t-s)|\xi|}&e^{-isv\cdot\xi}\widehat{L_{1,j,k}}(\xi,v,s)\,d\xi dv\Big\|_{L^\infty_x}\\
&\lesssim \|L_{1,j,k}(s)\|_{L^1_vL^\infty_x}+\|\nabla_xL_{1,j,k}(s)\|_{L^1_vL^\infty_x}\lesssim\varepsilon_02^{-j}
\end{split}
\end{equation*}
for any $s\in[0,t]$. The bounds \eqref{bvn5} follow if $\ast=I$ and $2^m\lesssim 1$, using the identities \eqref{deco1}.

On the other hand, if $m\geq\delta^{-6}$, $t\in[2^{m-2},2^{m+2}]$ and $s\geq t/2$ then we use \eqref{disper5}, \eqref{disper5.5}, and Lemma \ref{L1kjBou} to estimate
\begin{equation*}
\begin{split}
(1+2^{k+m})\Big\|\frac{1}{(2\pi)^3}&\int_{\R^3}\int_{\R^3}e^{ix\cdot\xi}e^{-(t-s)|\xi|}e^{-isv\cdot\xi}\widehat{L_{1,j,k}}(\xi,v,s)\,d\xi dv\Big\|_{L^1_x}\\
&\lesssim (1+2^k|t-s|)^{-6}\big[\|L_{1,j,k}(s)\|_{L^1_xL^1_v}+\|\nabla_vL_{1,j,k}(s)\|_{L^1_xL^1_v}\big]\\
&\lesssim \varep_0(1+2^k|t-s|)^{-6}2^{-4j/3}
\end{split}
\end{equation*}
and
\begin{equation*}
\begin{split}
(1+2^{k+m})2^{3m}\Big\|\frac{1}{(2\pi)^3}&\int_{\R^3}\int_{\R^3}e^{ix\cdot\xi}e^{-(t-s)|\xi|}e^{-isv\cdot\xi}\widehat{L_{1,j,k}}(\xi,v,s)\,d\xi dv\Big\|_{L^\infty_x}\\
&\lesssim (1+2^k|t-s|)^{-6}\big[\|L_{1,j,k}(s)\|_{X_s}+\|\nabla_vL_{1,j,k}(s)\|_{X_s}\big]\\
&\lesssim\varepsilon_0(1+2^k|t-s|)^{-6}2^{-4j/3}.
\end{split}
\end{equation*}
Moreover, if $s\in[0,t/2]$ then we use \eqref{disper5} and Lemma \ref{L1kjBou} to estimate
\begin{equation*}
\begin{split}
\Big\|\frac{1}{(2\pi)^3}\int_{\R^3}\int_{\R^3}e^{ix\cdot\xi}e^{-(t-s)|\xi|}e^{-isv\cdot\xi}\widehat{L_{1,j,k}}(\xi,v,s)\,d\xi dv\Big\|_{L^\infty_x}&\lesssim (1+2^k|t-s|)^{-6}\|L_{1,j,k}(s)\|_{L^1_xL^1_v}\\
&\lesssim \varep_0(1+2^{k+m})^{-6}2^{-4j/3}.
\end{split}
\end{equation*}
Since $\|P_lg\|_{L^\infty}\lesssim \|\widehat{P_lg}\|_{L^1}\lesssim 2^{3l}\|\widehat{P_lg}\|_{L^\infty}\lesssim 2^{3l}\|P_lg\|_{L^1}$ for any $l\in\Z$, it follows that 
\begin{equation*}
\begin{split}
2^{3m}\Big\|\frac{1}{(2\pi)^3}\int_{\R^3}\int_{\R^3}e^{ix\cdot\xi}e^{-(t-s)|\xi|}e^{-isv\cdot\xi}\widehat{L_{1,j,k}}(\xi,v,s)\,d\xi dv\Big\|_{L^\infty_x}\lesssim \varep_0(1+2^{k+m})^{-3}2^{-4j/3}.
\end{split}
\end{equation*}
Therefore, using the identities \eqref{deco1} and the last four inequalities, we have
\begin{equation*}
\begin{split}
(1+2^{k+m})\big[\|T^I_{1,j,k}(t)\|_{L^1_x}+2^{3m}\|T^I_{1,j,k}(t)\|_{L^\infty_x}\big]&\lesssim \int_0^t\varep_0(1+2^k|t-s|)^{-2}2^{-4j/3}\,ds\\
&\lesssim\varep_02^{-4j/3}\min\{2^m,2^{-k}\}.
\end{split}
\end{equation*}
Moreover, $2^{-4j/3}\min\{2^m,2^{-k}\}\lesssim 2^{-\delta^2j}$ if $(j,k,m)\in A^I$ and $m\geq \delta^{-6}$, and the desired bounds \eqref{bvn5} follow if $\ast=I$. 

{\bf{Proof of \eqref{bvn5} when $\ast=II$.}} Assume that $(j,k,m)\in A^{II}$, so $m\geq \delta^{-4}$, and $t\in[2^{m-2},2^{m+2}]$. If $s\in[0,t]$ we use \eqref{disper5}, Lemma \ref{ProjAB}, and Lemma \ref{L1kjBou} to estimate
\begin{equation}\label{bvn11.1}
\begin{split}
\Big\|\frac{1}{(2\pi)^3}&\int_{\R^3}\int_{\R^3}e^{ix\cdot\xi}e^{-(t-s)|\xi|}e^{-isv\cdot\xi}\frac{\mathcal{D'}}{1+\mathcal{D}^2}(\partial_s\widehat{L_{1,j,k}})(\xi,v,s)\,d\xi dv\Big\|_{L^1_x}\\
&\lesssim (1+2^k|t-s|)^{-6}\|\partial_sL_{1,j,k}(s)\|_{L^1_xL^1_v}\\
&\lesssim \varep_0(1+2^k|t-s|)^{-6}2^{-j}\langle s\rangle^{-1+\delta},
\end{split}
\end{equation}
where $\mathcal{D}'\in\{1,\mathcal{D}\}$. If $s\geq t/2$ then we use \eqref{disper5}, Lemma \ref{ProjAB}, and Lemma \ref{L1kjBou} to estimate
\begin{equation*}
\begin{split}
2^{3m}\Big\|\frac{1}{(2\pi)^3}&\int_{\R^3}\int_{\R^3}e^{ix\cdot\xi}e^{-(t-s)|\xi|}e^{-isv\cdot\xi}\frac{\mathcal{D'}}{1+\mathcal{D}^2}(\partial_s\widehat{L_{1,j,k}})(\xi,v,s)\,d\xi dv\Big\|_{L^\infty_x}\\
&\lesssim (1+2^k|t-s|)^{-6}\|\partial_sL_{1,j,k}(s)\|_{X_s}\\
&\lesssim\varepsilon_0(1+2^k|t-s|)^{-6}2^{-j}2^{-m+\delta m}.
\end{split}
\end{equation*}
As before, since $\|P_lg\|_{L^\infty}\lesssim 2^{3l}\|P_lg\|_{L^1}$ for any $l\in\Z$ we can use \eqref{bvn11.1} to estimate
\begin{equation*}
\begin{split}
2^{3m}\Big\|\frac{1}{(2\pi)^3}&\int_{\R^3}\int_{\R^3}e^{ix\cdot\xi}e^{-(t-s)|\xi|}e^{-isv\cdot\xi}\frac{\mathcal{D'}}{1+\mathcal{D}^2}(\partial_s\widehat{L_{1,j,k}})(\xi,v,s)\,d\xi dv\Big\|_{L^\infty_x}\\
&\lesssim \varep_0(1+2^k|t-s|)^{-3}2^{-j}\langle s\rangle^{-1+\delta}
\end{split}
\end{equation*}
for $s\in[0,t/2]$. We use these last three inequalities and the formulas \eqref{sug13.4}, and integrate from $0$ to $t$. The contribution of the term in the first line of \eqref{sug13.4} is easier to estimate, similar to \eqref{bvn11.1} with $s=0$. Therefore
\begin{equation*}
\|T^{II}_{1,j,k}(t)\|_{L^1_x}+2^{3m}\|T^{II}_{1,j,k}(t)\|_{L^\infty_x}\lesssim\varep_02^{-j}2^{-m+\delta m}\min\{2^m,2^{-k}\}.
\end{equation*}
This suffices to prove the desired bounds \eqref{bvn5}.

{\bf{Proof of \eqref{bvn6}.}} The formulas \eqref{deco1} show that
\begin{equation*}
(\partial_t+|\nabla|)T^I_{1,j,k}(t)=(-i)e^{it}\cdot R^I_{1,j,k}(t),
\end{equation*}
so the bounds \eqref{bvn6} follow from the bounds \eqref{bvn4} if $\ast=I$. Moreover, using \eqref{sug13.4},
\begin{equation*}
\mathcal{F}\big\{(\partial_t+|\nabla|)T^{II}_{1,j,k}\big\}(t)=\int_{\R^3}e^{-itv\cdot\xi}\frac{e^{it}}{1-i\mathcal{D}}(\partial_t\widehat{L_{1,j,k}})(\xi,v,t)\,dv.
\end{equation*}
Using \eqref{disper5}, Lemma \ref{ProjAB}, and \eqref{bvn8}, if $(j,k,m)\in A^{II}$, we have
\begin{equation*}
\begin{split}
\|(\partial_t+|\nabla|)&T^{II}_{1,j,k}(t)\|_{L^1_x}+2^{3m}\|(\partial_t+|\nabla|)T^{II}_{1,j,k}(t)\|_{L^\infty_x}\\
&\lesssim \|(\partial_tL_{1,j,k})(t)\|_{L^1_xL^1_v}+\|(\partial_tL_{1,j,k})(t)\|_{X_t}\lesssim\varep_02^{-j}2^{-m+\delta m},
\end{split}
\end{equation*}
  This gives the desired bounds \eqref{bvn6} if $\ast=II$, which completes the proof of Proposition \ref{closeboot1}.
\qed

\section{Bounds on the static terms  and the type-I reaction term}\label{StatCont}

In this section we estimate the $B_T$-norms of the static terms, which consist of the type-I  components $R_{2,j,k}^I$  defined in \eqref{deco2} and the type-II components $R_{2,j,k}^{II}$ defined in \eqref{sug23.2}. Moreover, the estimate of the type-I oscillatory terms  $T_{2,j,k}^{I}$ will be obtained as a byproduct.


It turns out that many error terms arising not only in this section but also in the next one fit into a general framework of trilinear operators. We define these next and then record some of the relevant estimates. We note that these bounds improve over crude estimates through the use of the density decomposition \eqref{rhodeco} into static and oscillatory components. 

Assume that $\theta_1,\theta_2 \in [0, 1]$ and $s\in[0,T]$, and define the trilinear operators
\begin{equation}\label{ropi1}
\begin{split}
&  \mathcal{Q}_{j,k}(f, g; C)(x,\gamma, \tau, s):=\int_{\R^3} \int_{\R^3}\mathcal{K}_{j,k}(x-y, v,\tau,s) C(x, y, v, \gamma, \tau, s)\\
&\times f(y-(s-\tau)v+P_1(y,v,\tau,s),\tau)g(y-(s-\gamma)v + P_2(y,v,\gamma,\tau,s), \gamma)\,dy d v, 
\end{split}
\end{equation}
where $\gamma,\tau\in[0,s]$, $(j,k)\in\Z_+\times\Z$. We assume that kernel $\mathcal{K}_{j,k}$ and the functions $P_1,P_2$ satisfy the uniform estimates 
 \be\label{june16eqn4}
 \begin{split}
&\big| \mathcal{K}_{j,k}(y, v,\tau,s)\big|\lesssim 2^{3k}(1+2^k|y|)^{-8}\cdot 2^{-3j}\widetilde{\varphi}_{[j-4,j+4]}(v),\\
&\big|\partial_yP_1(y,v,\tau,s)|+\langle s\rangle^{-1}\big|\partial_vP_1(y,v,\tau,s)|\lesssim\varep_1,\\
&\big|\partial_yP_2(y,v,\gamma,\tau,s)|+\langle s\rangle^{-1}\big|\partial_vP_2(y,v,\gamma,\tau,s)|+\big|\partial_\gamma P_2(y,v,\gamma,\tau,s)|\lesssim\varep_1.
\end{split}
\ee
Notice that we suppress the dependence on the kernels $\mathcal{K}_{j,k}$ and the exact functions $P_1,P_2$ in the notation for the operators $\mathcal{Q}_{j,k}$. In most of our applications $P_1(y,v,\tau,s)=\theta_1\widetilde{Y}(y-sv,v,\tau,s)$ and $P_2=\theta_2\widetilde{Y}(y-sv,v,\gamma,s)$, $\theta_1,\theta_2\in[0,1]$, and the bounds \eqref{june16eqn4} follow from Lemma \ref{derichar2}.

We assume that the coefficient $C$ is differentiable in $\gamma$, and define
\begin{equation}\label{june16eqn3}
\begin{split}
\Lambda(C)(\gamma,\tau, s)&:= \|C(x, y, v,\gamma,\tau,s)\|_{L^\infty_{x,y,v}}  + \|\langle\gamma\rangle\p_\gamma C(x, y, v,\gamma,\tau,s)\|_{L^\infty_{x,y,v}}.
\end{split}
\end{equation}

Assume that $m_2\geq 0$, $t_3,t_4\in [2^{m_2}-1,2^{m_2+1}]\cap [0,s]$, and define the trilinear operators
\begin{equation}\label{june16eqn1}
\mathcal{B}^{ }_{j,k}(f, g; C)(x, \tau, s) = \int_{t_3}^{t_4}\mathcal{Q}_{j,k}(f, g; C)(x, \gamma, \tau, s)\,d\gamma.
\end{equation}
For these we have the following bounds:

\begin{lemma}[Trilinear estimates]\label{keybilinearlemma}
Let  $s\in [0, T]$,  $\tau \in [0, s],$ $2^{m_2-1}\leq s$. With the assumptions of the bootstrap Proposition \ref{MainBootstrapProp}, and the notation and assumptions above,    we have
\begin{align}
\| \mathcal{B}_{j,k}(\mathcal{R}^1\nabla E,  \mathcal{R}^2E; C) (., \tau, s)  \|_{B^0_s}&\lesssim   \varep_1^2 \min\{2^{- 1.1 m_2}, \langle \tau \rangle^{-1.1} \}  \Lambda^\ast(C)(\tau,s),\label{june16eqn31A1}\\
\| \mathcal{B}_{j,k}(\mathcal{R}^1E, \mathcal{R}^2\nabla E; C)  (., \tau, s)\|_{B^0_s}&\lesssim \varep_1^2 \min\{2^{- 1.1 m_2}, \langle \tau \rangle^{-1.1} \}  \Lambda^\ast(C)(\tau,s),\label{june16eqn31A2}\\
\| \mathcal{B}_{j,k}(\mathcal{R}^1E,  \mathcal{R}^2E; C) (., \tau, s)\|_{B^0_s}  &\lesssim \varep_1^2 \min\{2^{- 0.1 m_2}, \langle \tau \rangle^{-0.1} \}  \Lambda^\ast(C)(\tau,s),\label{june16eqn32}\\
\| \mathcal{B}_{j,k}(\mathcal{R}^1\nabla E, \mathcal{R}^2\nabla E; C) (., \tau, s)\|_{B^0_s}&\lesssim \varep_1^2  \min\{2^{- 2.1 m_2}, \langle \tau \rangle^{-2.1} \}  \Lambda^\ast(C)(\tau,s),\label{june16eqn33}\\
\| \mathcal{B}_{j,k}(\mathcal{R}^1\nabla^2 P_{k_1}E,  \mathcal{R}^2E; C) (., \tau, s)\|_{B^0_s}&\lesssim \varep_1^2 \langle \tau \rangle^{-1} \min\{2^{- 1.1 m_2}, \langle \tau \rangle^{-1.1} \}  \Lambda^\ast(C)(\tau,s),\label{april10eqn50}\\
\| \mathcal{B}_{j,k}(\mathcal{R}^1\nabla^2 P_{k_1}E, \mathcal{R}^2\nabla E; C) (., \tau, s)\|_{B^0_s}&\lesssim \varep_1^2 \langle \tau \rangle^{-1} \min\{2^{- 2.1 m_2}, \langle \tau \rangle^{-2.1} \} \Lambda^\ast(C)(\tau,s),\label{april10eqn51}
\end{align}
for any $k_1\in\Z$, where $\Lambda^\ast(C)(\tau,s):=\sup_{\gamma\in [t_3,t_4]}\Lambda(\gamma, \tau, s)$ and $\mathcal{R}^1,\mathcal{R}^2$ are operators defined by H\"{o}rmander-Michlin multipliers satisfying the inequalities
\begin{equation}\label{HormMich}
\sup_{|\alpha|\leq 8}\sup_{\xi\in\R^3}|\xi|^{|\alpha|}|D^\alpha_\xi \mathcal{R}^l(\xi)|\leq 1,\qquad l\in\{1,2\}.
\end{equation}

Moreover, for any $k_1,k_2\in\Z$ and $\gamma,\tau\in[0,s]$,
\begin{equation}\label{change15}
\begin{split}
&\|\mathcal{Q}_{j,k} (P_{k_1}\mathcal{R}^1\rho, P_{k_2}\mathcal{R}^2\rho; C)(., \gamma, \tau, s)\|_{L^1_x}+\langle s\rangle^{3}\|\mathcal{Q}_{j,k} (P_{k_1}\mathcal{R}^1\rho, P_{k_2}\mathcal{R}^2\rho; C)(., \gamma, \tau, s)\|_{L^\infty_x}\\
&\quad\lesssim\frac{\varep_1^2}{\langle \tau\rangle^{1-2\delta}2^{k_1}+\langle \tau\rangle^{-\delta}}\frac{\min\{2^{k_{\min}}, \langle \tau\rangle^{-1} , \langle \gamma\rangle^{-1}\}^3}{\langle \gamma\rangle^{1-2\delta}2^{k_2}+\langle \gamma\rangle^{-\delta}}\|C(.,.,.,\gamma,\tau,s)\|_{L^\infty_{x,y,v}}.
\end{split}
\end{equation}
\end{lemma}

\begin{proof}
For the sake of presentation we postpone the proof to section \ref{trilin}. 
\end{proof}

In the rest of this section we fix $t\in[0,T]$ and an integer $m\geq 0$ such that $t\in[2^m-1,2^{m+1}]$. We also fix an integer $K$ such that
\begin{equation}\label{SizeOfKLinearBoundInEpsilon}
1\le \varepsilon_1 2^K\le 8.
\end{equation}
This is needed in some cases of very high frequencies, to ensure that the overall contribution is nonlinearly small. The formulas \eqref{qwp1}, \eqref{qwp2}, and \eqref{Lan6} show that
\begin{equation}\label{L2Formula}
\begin{split}
L_{2,j,k}(x,v,s,t)&=P_k\big\{E(x,s)M'_j(v)-E(x+\widetilde{Y}_\#(x,v,s,t),s)M'_j(v+\widetilde{W}_\#(x,v,s,t))\big\},\\
\widetilde{Y}_\#(x,v,s,t)&:=\widetilde{Y}(x-sv,v,s,t),\quad   \widetilde{W}_\#(x,v,s,t):=\widetilde{W}(x-sv,v,s,t).
\end{split}
\end{equation}

 \subsection{Large velocity case: $j\geq 19m/20$ or $m\leq \delta^{-4}$}
 
The goal of this section is to prove that the contribution of large velocities is acceptable. 

\begin{proposition}\label{change6}
Assume that the bounds \eqref{YW12} hold, $k\in\Z$, and $2^j\gtrsim 2^{19m/20}$. Then
\begin{equation}\label{april2eqn2}
\begin{split}
\Vert \langle t\rangle^{1-\delta^2} \langle \nabla_x\rangle R^I_{2,j,k}(t)\Vert_{B_t^0}\lesssim \varepsilon_1^{3/2}2^{-\delta^2 j},\\
\Vert \langle t\rangle^{-\delta} T^I_{2,j,k}(t)\Vert_{B_t^0}+\Vert \langle t\rangle^{1-\delta} \nabla_{x,t}T^I_{2,j,k}(t)\Vert_{B_t^0}\lesssim \varepsilon_1^{3/2}2^{-\delta^2 j}.
\end{split}
\end{equation}

\end{proposition}

\begin{proof}

The first inequality follows from Lemma \ref{LargeV1Lem} and Lemma \ref{LargeV2Lem} below. To prove the second estimate in \eqref{april2eqn2}, it thus suffices to show that
\begin{equation}\label{SuffCSProp1}
\begin{split}
\Vert \langle t\rangle^{-\delta} T^I_{2,j,k}\Vert_{B_t^0}+\Vert \langle t\rangle^{1-\delta} \nabla_{x}T^I_{2,j,k}\Vert_{B_t^0}\lesssim \varepsilon_1^{3/2}2^{-\delta^2 j}. 
\end{split}
\end{equation}

From \eqref{deco2} we have
\begin{equation}\label{april2eqn1}
\begin{split}
T^I_{2,j,k}(x,t)&: =-i\int_{0}^te^{is}e^{-(t-s)\vert\nabla\vert}R^I_{2,j,k}(x,s)ds.
\end{split}
\end{equation}
Notice that, 
\begin{equation}\label{BoundsPoissonKernel}
\begin{split}
\text{ for any }p\in[1, \infty]\text{ and }\lambda>0, \qquad \Vert \vert\nabla\vert e^{-\lambda\vert\nabla\vert} f\Vert_{L^p}&\lesssim \lambda^{-1}\Vert f\Vert_{L^p}.  
\end{split} 
\end{equation}
Using Lemma \ref{LargeV1Lem}, Lemma \ref{LargeV2Lem}, and \eqref{BoundsPoissonKernel}, we find that
\begin{equation*}
\begin{split}
\Vert e^{is}e^{-(t-s)\vert\nabla\vert}R^I_{2,j,k}(x,s)\Vert_{B_t^0}&\lesssim \varepsilon_1^{3/2}\langle s\rangle^{\delta^2-1}2^{-\delta^2 j},\\
\Vert e^{is}e^{-(t-s)\vert\nabla\vert}\nabla_xR^I_{2,j,k}(x,s)\Vert_{B_t^0}&\lesssim \varepsilon_1^{3/2}2^{k^-}[1+2^k(t-s)]^{-1}\langle s\rangle^{\delta^2-1}2^{-\delta^2 j}.
\end{split}
\end{equation*}
This gives \eqref{SuffCSProp1}, upon integration.
\end{proof}

It remains to prove Lemmas \ref{LargeV1Lem} and \ref{LargeV2Lem}.  We use the formula \eqref{deco2}, thus
\begin{equation}\label{change7}
R^I_{2,j,k}(x,t)=\int_0^t\int_{\R^3}L_{2,j,k}(x-(t-s)v,v,s,t)\,dvds.
\end{equation}
We start with the contribution of high frequencies.

\begin{lemma}\label{LargeV1Lem}
Assume that $2^j\gtrsim 2^{19m/20}$. Then
\begin{equation*}
\Vert \langle t\rangle\nabla_x R^I_{2,j,k}(t)\Vert_{B_T}\lesssim \varepsilon_1^{3/2}2^{-\delta j}.
\end{equation*}
\end{lemma}

\begin{proof}
We examine the formula \eqref{L2Formula} and decompose, for $p\in\{1,2,3\}$,
\begin{equation}\label{L2HF}
\partial_{x^p}L_{2,j,k}:=   P_k [Err_{j}]+  \sum_{k_1\in \Z}P_k[\mathcal{L}_{j,k_1}],\\
\end{equation}
where, with $E_{k_1}:=P_{k_1}E$, we have
\begin{equation}\label{april5eqn2}
\begin{split}
\mathcal{L}_{j,k_1}(x,&v,s,t)
:=\big[\partial_{x^p}E_{k_1}(x,s)-\partial_{x^p}E_{k_1}(x+\widetilde{Y}_\#(x,v,s,t),s)\big]\cdot M'_j(v),\\
Err_{j}(x,&v,s,t):=\partial_{x^p}E_{ }(x+\widetilde{Y}_\#(x,v,s,t),s)\cdot\big[M'_j(v)-M'_j(v+\widetilde{W}_\#(x,v,s,t))\big]\\
&-\partial_{x^p}\widetilde{Y}^q_\#(x,v,s,t)\partial_{x^q}E_{ }(x+\widetilde{Y}_\#(x,v,s,t),s)\cdot M'_j(v+\widetilde{W}_\#(x,v,s,t))\\
&-E_{ }(x+\widetilde{Y}_\#(x,v,s,t),s)\cdot\big[\partial_{x^p}\widetilde{W}_\#^q(x,v,s,t)\partial_{v^q}M'_j(v+\widetilde{W}_\#(x,v,s,t))\big]. 
\end{split}
\end{equation}

We first estimate the contribution of the error term $Err_{j }$. We expand as in \eqref{Lan7} 
\begin{equation}\label{april2eqn5}
\begin{split}
\widetilde{W}_\#(x,v,s,t)&=-\int_{s}^tE(x-(s-\gamma)v+\widetilde{Y}(x-s v,v,\gamma,t),\gamma)\,d\gamma,\\
\widetilde{Y}_\#(x,v,s,t)&=\int_{s}^t(\gamma-s)E(x-(s-\gamma)v+\widetilde{Y}(x-s v,v,\gamma,t),\gamma)\,d\gamma.
\end{split}
\end{equation}
Then we observe that, with the notation of \eqref{ropi1}, we can write the contribution of $Err_{j }$ as a sum of terms of the form
\begin{equation}\label{ErrHF}
\begin{split}
&\int_{0}^t\int_{\mathbb{R}^3}P_k[Err_{j }](x-(t-s)v,v,s,t) d v ds=\int_{0}^1 \int_{0}^t \int_{s}^t \mathcal{Q}_{j,k}(\nabla E_{ },E,C^1)(x,\gamma, s , t) d \gamma d s d\theta\\
 &\qquad\qquad\qquad  +  \int_{0}^t \int_{s}^t\big[\mathcal{Q}_{j,k}(\nabla E_{ },\nabla E,C^2)+ \mathcal{Q}_{j,k}(E_{ },\nabla E,C^3)\big](x,\gamma, s , t) d \gamma d s,\\
\end{split}
\end{equation}
for suitable kernels $\mathcal{K}_{j,k}^i$ and
\begin{equation*}
\begin{split}
C^1(x,y,v,\gamma,s,t)&=2^{3j}\nabla_vM_j^\prime(v+\theta \widetilde{W}(y-tv,v,s,t)),\\
C^2(x,y,v,\gamma,s,t)&=2^{3j}(\gamma-s)(\delta_p^\ell+\partial_{x^p}\widetilde{Y}^\ell(y-tv,v,\gamma,t))M_j^\prime(v+\widetilde{W}(y-tv,v,s,t)),\\
C^3(x,y,v,\gamma,s,t)&=2^{3j}(\delta_p^\ell+\partial_{x^p}\widetilde{Y}^\ell(y-tv,v,\gamma,t))\nabla_vM_j^\prime(v+\widetilde{W}(y-tv,v,s,t)).
\end{split}
\end{equation*}
Notice that $\Lambda(C^1)(\gamma,s,t)+\Lambda(C^3)(\gamma,s,t)\lesssim 2^{-2j}$ and $\Lambda(C^2)(\gamma,s,t)\lesssim 2^{-2j}\langle \gamma\rangle$. This leads to acceptable contributions using Lemma \ref{keybilinearlemma} and the bound $2^j\gtrsim 2^{19m/20}$.

Now, we move on to the estimate of the main contribution. We split into two cases based on the size of $k_1$ as follows (recall $K$ defined in \eqref{SizeOfKLinearBoundInEpsilon}).

\textbf{Case $\mathbf{1}$}: $k_1\ge -m/2 +K/2$. We first consider the $L^1_x$-estimate.  Using Lemma \ref{Laga10}(i),
\begin{equation}\label{LFEstim1}
\begin{split}
\langle t\rangle\Big\|\int_{0}^t\int_{\mathbb{R}^3}&\mathcal{L}_{j,k_1}(x-(t-s)v,v,s,t)dv  d s \Big\|_{L^1_x}\lesssim\langle t\rangle 2^{-2j}\int_{0}^t\Vert \nabla_xE_{k_1}(s)\Vert_{L^1}ds\\
&\lesssim \varepsilon_1\langle t\rangle2^{-2j-k_1}\int_{0}^t\langle s\rangle^{2\delta-1}ds\lesssim\varepsilon_1 2^{-2j-k_1}\langle t\rangle^{1+2\delta}.
\end{split}
\end{equation}
Thus the contribution for all $k_1$ with  $k_1\ge  -m/2 +K/2$  is acceptable, due to the assumption $2^{2j}\gtrsim 2^{19m/10}$. Now, we consider the $L^\infty_x$-estimate. 
If $t/2\leq s \leq t$, we use that
\begin{equation*}
\begin{split}
\Vert \nabla E_{k_1}(s)\Vert_{L^\infty_x}\lesssim \varepsilon_1\langle s\rangle^{2\delta-4}2^{-k_1}
\end{split}
\end{equation*}
and integrate over $v$, so
\begin{equation*}
\begin{split}
 &\langle t\rangle^4\Big\| \int_{t/2}^t\int_{\mathbb{R}^3}\mathcal{L}_{j,k_1}(x-(t-s)v,v,s,t)dv  d s \Big\|_{L^\infty_x}\\ 
&\qquad \lesssim \langle t\rangle^{4}\int_{t/2}^t\int_{\mathbb{R}^3} \Vert \nabla E_{k_1}(s)\Vert_{L^\infty_x} \vert M'_j(v)\vert dv d s
 \lesssim \varepsilon_1 \langle t\rangle^{2\delta+1}2^{-k_1-2j}
\end{split}
\end{equation*}
which also gives an acceptable contribution.  It remains to consider the case  $0\le s\le t/2$. If $t\lesssim 1$ then the same bounds as above hold. On the other hand, if $t\geq 10$ then we can use the characteristics bounds in Lemma \ref{derichar2} to verify that the Jacobian of the map
\begin{equation*}\label{AddedJacobian}
w\mapsto y:=x-(t-s)w+\widetilde{Y}(x-tw,w,s,t)
\end{equation*}
is of size $\langle t\rangle^3$. Using a simple change of variable, we deduce that, for fixed $x$,
\begin{equation*}
\begin{split}
\int_{\mathbb{R}^3}\vert\partial_{x^p}E_{k_1}(x-(t-s)v,s)M_j^\prime(v)\vert dv&\lesssim 2^{-5j}\langle t\rangle^{-3}\Vert P_{k_1}\rho(s)\Vert_{L^1},\\
\int_{\mathbb{R}^3}\vert\partial_{x^p}E_{k_1}(x-(t-s)v+\widetilde{Y}(x-tv,v,s,t),s)M_j^\prime(v)\vert dv&\lesssim 2^{-5j}\langle t\rangle^{-3}\Vert P_{k_1}\rho(s)\Vert_{L^1}.
\end{split}
\end{equation*}
Therefore
\begin{equation*}
\begin{split}
\langle t\rangle^4\Big\| \int_{0}^{t/2}\int_{\mathbb{R}^3}\mathcal{L}_{j,k_1}(x-(t-s)v,v,s,t)dvds\Big\|_{L^\infty}&\lesssim \varepsilon_1 2^{-k_1-5j}\langle t\rangle^{2\delta+1},
\end{split}
\end{equation*}
which again gives an acceptable contribution.

 \textbf{Case $\mathbf{2}$}: $k_1\leq   -m/2 +K/2$. For this case,  the derivative in $E$ is not problematic. Note that, after taking the difference between the two electric fields, we have 
\be\label{april5eqn1}
\big[\partial_{x^p}E_{k_1}(x ,s)-\partial_{x^p}E_{k_1}(x+\widetilde{Y}_\#(x,v,s,t),s)\big]M'_j(v) =- \int_{0}^1 L^\theta_{j, k_1}(x,v,s,t) d\theta,
\ee 
where
\[
L^\theta_{j, k_1}(x,v,s,t) := \widetilde{Y}^q_\#(x,v,s,t)\partial_{x^p}\partial_{x^q}E_{k_1}(x+\theta\widetilde{Y}_\#(x,v,s,t),s)\cdot M'_j(v). 
\]

We expand the function $\widetilde{Y}_\#$ as in \eqref{april2eqn5}. Therefore, for any fixed $\theta\in [0,1]$, in terms of the trilinear operators defined in \eqref{ropi1}, we have 
\[
\int_{\R^3} P_k L^\theta_{j, k_1}(x-(t-s)v,v,s,t) d v = \int_{s }^t \mathcal{Q}_{j,k}^{ }( \nabla E_{k_1}, E; C^4)(x,   \tau,s, t) d \tau,
\]
with $C^4(x,y,v,\gamma,s,t):=2^{k_1}(\gamma-s)  2^{3j}M'_j(v)$, $\Lambda(C^4)(\gamma,s,t)\lesssim 2^{-2j}2^{k_1}\langle \gamma\rangle$. This leads to an acceptable contribution using Lemma \ref{keybilinearlemma} and recalling the assumption $2^{2j}\gtrsim 2^{1.9m}$.
\end{proof}

The following lemma deals with the contributions of low frequencies, when $\langle \nabla \rangle \approx 1$.

\begin{lemma}\label{LargeV2Lem}
If $2^j\gtrsim 2^{19m/20}$ and $k\leq 0$ then
\begin{equation}\label{BoundLargeV2}
\Vert R^I_{2,j,k}\Vert_{B^0_t}\lesssim \varepsilon_1^2\langle t\rangle^{0.9}2^{-2j}.
\end{equation}
\end{lemma}

\begin{proof}

We decompose
\begin{equation}\label{L2}
L_{2,j,k}(x,v,s,t)=L_{2,j,k}^a+L_{2,k,j}^b,
\end{equation}
where
\begin{equation*}
\begin{split}
L_{2,j,k}^a(x,v,s,t)&:=P_k\big\{E(x,s)-E(x+\widetilde{Y}_\#(x,v,s,t),s)\big\}\cdot M'_j(v),\\
L_{2,j,k}^b(x,v,s,t)&:=P_k\big\{E(x+\widetilde{Y}_\#(x,v,s,t),s)\cdot\big[M'_j(v)-M'_j(v+\widetilde{W}_\#(x,v,s,t))\big]\big\}.
\end{split}
\end{equation*}
We can rewrite
\begin{equation*}
\begin{split}
E(x,s)-E(x+\widetilde{Y}_\#(x,v,s,t),s)&=-\int_{0}^1\widetilde{Y}_\#(x,v,s,t)\cdot\nabla_x E(x+\theta\widetilde{Y}_\#(x,v,s,t),s)d\theta,\\
M'_j(v)-M'_j(v+\widetilde{W}_\#(x,v,s,t))&=-\int_{0}^1\widetilde{W}_\#(x,v,s,t)\cdot\nabla_{v}M'_j(v+\theta\widetilde{W}_\#(x,v,s,t))d\theta,\\
\end{split}
\end{equation*}
We use the formulas \eqref{april2eqn5}, so for any $\ast\in\{a,b\}$ we have
\begin{equation}\label{april5eqn21}
\begin{split}
L^{ \ast}_{2,j,k}(x,v,s,t)&=\int_{0}^1 \int_{s}^t  \mathcal{L}^{ \ast,\theta}_{2,j,k}(x,v,\gamma,s,t)\,d \gamma d\theta, \\
 \mathcal{L}^{ a,\theta}_{2,j,k}(x,v,\gamma,s,t) &:=P_k\big[-(\gamma-s)E^p(x-(s-\gamma)v+\widetilde{Y}(x-sv,v,\gamma,t),\gamma)\\
 &\qquad\qquad\times\partial_{x^p}E(x+\theta\widetilde{Y}_\#(x,v,s,t),s)\cdot M'_j(v) \big],\\
 \mathcal{L}^{b,\theta}_{2,j,k}(x,v,\gamma, s,t) &:=P_k\big[  E^p(x-(s-\gamma)v+\widetilde{Y}(x-sv,v,\gamma,t),\gamma)\\
 &\qquad\qquad\times E(x+\widetilde{Y}_\#(x,v,s,t),s)\cdot \partial_{v^p}M'_j(v+\theta\widetilde{W}_\#(x,v,s,t))\big]. 
\end{split}
\end{equation}

Therefore, for any $\theta\in [0,1]$, in terms of the trilinear operator defined in \eqref{ropi1}, we can write
\begin{equation}\label{april5eqn22}
\begin{split}
  \int_{\R^3}  \mathcal{L}^{a,\theta}_{2,j,k}(x-(t-s)v,v,\gamma,s,t)  d v = \mathcal{Q}_{j,k}(\nabla E, E; C^a)(x, \gamma,s, t), \\ 
    \int_{\R^3}  \mathcal{L}^{b,\theta}_{2,j,k}(x-(t-s)v,v,\gamma,s,t)  d v = \mathcal{Q}_{j,k}(E, \nabla E; C^b)(x, \gamma, s, t), 
\end{split}
\end{equation}
where
\[
C^a(x,y,v,\gamma,s,t) =(\gamma-s)2^{3j}M'_j(v),\qquad C^b(x,y,v,\gamma,s,t)=2^{3j}\nabla_vM'_j(v+\theta\widetilde{W}_\#(y,v,s,t)).
\]
Clearly $\Lambda(C^a)(\gamma,s,t)+\Lambda(C^b)(\gamma,s,t)\lesssim 2^{-2j}\langle \gamma\rangle$. Therefore, using Lemma \ref{keybilinearlemma},
\begin{equation*}
\begin{split}
  \Big\|\int_s^t\int_{\R^3}  \mathcal{L}^{\ast,\theta}_{2,j,k}(x-(t-s)v,v,\gamma,s,t) \,d vd\gamma\Big\|_{B_t^0}\lesssim \varep_1^2\langle s\rangle^{-0.1}2^{-2j},
\end{split}
\end{equation*}
for $\ast\in\{a,b\}$. We integrate in $s\in[0,t]$ and $\theta\in[0,1]$ to prove the desired bounds \eqref{BoundLargeV2}.
\end{proof}

 \subsection{Small velocity case: $j\leq 19m/20$ and $m\geq\delta^{-4}$} 

In this subsection we estimate the terms $R^I_{2,j,k}$ and $T^I_{2,j,k}$ in the case where $j\leq 19m/20$, $m\geq\delta^{-4}$, and $j+k\geq - \delta m /3$, and the terms $R^{II}_{2,j,k}$ when $j\leq 19m/20$, $m\geq\delta^{-4}$, and $j+k\leq  - \delta m /3$. 

Since in the expression for $R^{II}_{2,j,k}$ we gain a factor of $\mathcal{D}^2$ over $R^I_{2,j,k}$, we can treat both these terms in the same way.
For any $\ast\in \{I, II\}$, we localize the size of $t-s$ and decompose
\begin{equation*}
\begin{split}
R^{\ast}_{2,j,k}&=\sum_{0\le n\le m+2}R^{\ast}_{2,j,k,n},\qquad R^{\ast}_{2,j,k,n}(x,t):=\int_{0}^t\int_{\mathbb{R}^3}\widetilde{\varphi}_n(t-s)L^{\ast}_{2,j,k}(x-(t-s)v,v,s,t)dsdv,
\end{split}
\end{equation*}
where
\be\label{april7eqn1}
\begin{split}
&L^{I}_{2,j,k}(x,v,s,t):= L_{2,j,k}(x ,v,s,t),\\
&L^{II}_{2,j,k}(x,v,s,t):=\int_{\R^3}\mathcal{K}_{j,k}^{II}(y,v ) L_{2,j}(x-y ,v,s,t) d y, 
\end{split}
\ee
and the kernels $\mathcal{K}^{II}_{j,k}(y,v )$ are defined when $j\leq 19m/20$, $m\geq\delta^{-4}$, and $j+k\leq  - \delta m /3$ by
\be\label{april10eqn11}
\begin{split}
&\mathcal{K}^{II}_{ j,k }(y,v ) : = \frac{1}{(2\pi)^3}\widetilde{\varphi}_{[j-2,j+2]}(v)\int_{\R^3} e^{i y \cdot \xi } \frac{\mathcal{D}(\xi,v)^2}{1+\mathcal{D}(\xi,v)^2 } \varphi_{k}(\xi) \,d \xi, \\ 
 &  |\mathcal{K}^{II}_{j,k}(y,v ) |\lesssim 2^{3k+2(k+j)} (1+2^k|y|)^{-100}. 
\end{split}
\ee 
Therefore, we can think that the functions $L^{I}_{2,j,k}$ and $L^{II}_{2,j,k}$ differ essentially by a factor of $2^{2k+2j}$, $L^{II}_{2,j,k}\sim 2^{2k+2j}L^I_{2,j,k}$.

Our main result in this section is the following:

\begin{proposition}\label{MainPropositionI}
Assume that the bootstrap bounds \eqref{YW12} hold, and $m\geq\delta^{-4}$, $j\in[0,19m/20]$, $k\in\Z$, $n\in[0, m+2]$. Then
\begin{equation}\label{ElementaryBoundTermI}
\begin{split}
&\text{ if }\,\,j+k \ge-\delta m/3\,\,\text{ then }\,\,\Vert \langle t\rangle \langle\nabla_x\rangle R^I_{2,j,k,n}(t)\Vert_{B^0_t}\lesssim \varepsilon_1^{3/2} 2^{-(j+k^-)}2^{\delta m/2}, \\ 
&\text{ if }\,\,j+k \le -\delta m/3,\text{ then }\,\,\Vert \langle t\rangle  R^{II}_{2,j,k,n}(t)\Vert_{B^0_t}\lesssim \varepsilon_1^{3/2}2^{ \delta m/2}.
\end{split}
\end{equation}
In particular, for any $k,m\in\Z$, $m\geq \delta^{-4}$,
\begin{equation}\label{CCLITerms}
\begin{split}
&\sum_{j\in [0,19m/20],\,j\ge-k-\delta m/3} \Vert \langle t\rangle^{1-\delta}\langle\nabla_x\rangle R^I_{2,j,k}(t)\Vert_{B_t^0}\lesssim\varepsilon_1^{3/2},\\
&\sum_{j\in [0,19m/20],\,j\le-k-\delta m/3} \Vert \langle t\rangle^{1-\delta}\langle\nabla_x\rangle R^{II}_{2,j,k}(t)\Vert_{B_t^0}\lesssim\varepsilon_1^{3/2},\\
&\sum_{j\in [0,19m/20],\,j\ge-k-\delta m/3}\Vert \langle t\rangle^{-\delta}T^I_{2,j,k}(t)\Vert_{B_t^0}+\Vert \langle t\rangle^{1-\delta}\nabla_{x,t} T^I_{2,j,k}(t)\Vert_{B_t^0}\lesssim\varepsilon_1^{3/2}.\\
\end{split}
\end{equation}
\end{proposition}

\begin{proof}[Proof] The basic bounds \eqref{ElementaryBoundTermI} follow from Lemma \ref{RILateTimeHighFrequencies}, Lemma \ref{LemReductionR12}, and Lemma \ref{maintypeI} below. Here we show that they imply the bounds \eqref{CCLITerms}.

The first two bounds in \eqref{CCLITerms} follow by summation with respect to $j$ and $n$ (notice that $k\leq 0$ in the bounds concerning $R^{II}_{2,j,k}$). For the last bound, using \eqref{april2eqn1}, we have
\begin{equation*}
\begin{split}
\Vert e^{is}e^{-(t-s)\vert\nabla\vert} R^I_{2,j,k}(s)\Vert_{B_t^0}&\lesssim \Vert R^I_{2,j,k}(s)\Vert_{B^0_s},\\
(t-s)\Vert e^{is}\vert\nabla\vert e^{-(t-s)\vert\nabla\vert} R^I_{2,j,k}(s)\Vert_{B^0_t}&\lesssim \Vert R^I_{2,j,k}(s)\Vert_{B^0_s},
\end{split}
\end{equation*}
for $s\in[0,t]$. Therefore, using \eqref{april2eqn1},
\begin{equation*}
\begin{split}
\Vert T^I_{2,j,k}(t)\Vert_{B^0_t}&\lesssim \varepsilon_1^{3/2} 2^{3\delta m/5}2^{-(j+k^-)},\\
\Vert \vert\nabla\vert T^I_{2,j,k}(t)\Vert_{B^0_t}&\lesssim \int_{0}^{t-1}\Vert R^I_{2,j,k}(s)\Vert_{B^0_s}\frac{ds}{t-s}+2^k\int_{t-1}^t \Vert R^I_{2,j,k}(s)\Vert_{B^0_s}ds\\
& \lesssim\varepsilon_1^{3/2} 2^{-(j+k^-)}2^{\delta m/2}\langle t\rangle^{-1+\delta m/10}.
\end{split}
\end{equation*}
Finally, we notice that $\left(\partial_t+\vert\nabla\vert\right)T^I_{2,j,k}(t)=-ie^{it}R^I_{2,j,k}(t)$, and the bounds in the last line of \eqref{CCLITerms} follow by summation over $j$.
\end{proof}

We bound first the contributions of the small indices $n\in[0,10]$.

\begin{lemma}\label{RILateTimeHighFrequencies}
With the hypothesis of Proposition \ref{MainPropositionI}, if $n\in[0,10]$ then
\begin{equation}\label{change12}
\begin{split}
&\text{ if }\,\,\,j+k \ge -\delta m/3\,\,\,\text{ then }\,\,\,2^j\Vert \langle t\rangle \nabla_xR^I_{2,j,k,n}(t)\Vert_{B^0_t}\lesssim\varepsilon_1^{3/2},\\
&\text{ if }\,\,j+k \le -\delta m/3\,\,\,\text{ then }\,\,\,2^j\Vert \langle t\rangle  \nabla_x R^{II}_{2,j,k,n}(t)\Vert_{B^0_t}\lesssim \varepsilon_1^{3/2} 2^{j+k}.
\end{split}
\end{equation}
\end{lemma}

\begin{proof} To prove the bounds on $R^I_{2,j,k,n}$ we use again the decomposition \eqref{L2HF}. The error terms can be written as in \eqref{ErrHF}. The integrals over $s$ and $\gamma$ are over intervals of length $\lesssim 1$, so these contributions are easily acceptable using the bounds \eqref{change15}. 

To bound the main term, we consider again two cases. For  $k_1\geq m/4+K/2$, with $K$ defined in \eqref{SizeOfKLinearBoundInEpsilon}, for $\theta\in \{0,1\}$ we find that
\begin{equation*}
\begin{split}
\Big\Vert \int_{\R^3} \partial_{x^p} E_{k_1}(x-(t-s)v+\theta \widetilde{Y}(x-tv,v,s,t), s)\cdot M'_j(v) d v\Big\Vert_{B^0_s}&\lesssim \varepsilon_1 2^{-k_1-2j}\langle s\rangle^{2\delta-1}
\end{split}
\end{equation*} 
and  each term in $\mathcal{L}_{j,k_1}$ gives an acceptable contribution separately. 

 If $k_1\le m/4+K/2$,  we work with the  difference of the two electric fields as in   \eqref{april5eqn1}.  As $t\geq s\geq t -2^{11}$, from  rough bilinear estimates, we have 
\begin{equation*}
\begin{split}
\Big\Vert \int_{\R^3} L^\theta_{j, k_1 }(x-(t-s)v,v,s,t) d v \Big\Vert_{B_t^0}&\lesssim 2^{-2j}\Vert \widetilde{Y}_\#(s,t)\Vert_{L^\infty}\Vert \nabla^2E_{k_1}(s)\Vert_{B^0_s}\lesssim\varepsilon_1^22^{  -2j} 2^{\delta k_1^-} \langle t \rangle^{-11/10}.
\end{split}
\end{equation*}
The desired bounds follow by summation over $k_1\le m/4+K/2$ and integration over $s$.

The bounds on $R^{II}_{2,k,j,n}$ are similar, using the identities \eqref{april7eqn1}--\eqref{april10eqn11} (recall the heuristical equivalence $L^{II}_{2,j,k}\sim 2^{2k+2j}L^I_{2,j,k}$) and the assumption $j+k\leq -\delta m/3\leq 0$.
\end{proof}

From now on we assume that $n\ge 10$. We first isolate the main contribution.

\begin{lemma}\label{LemReductionR12}
For any $\star\in \{I, II\}$, we have
\begin{equation*}
\begin{split}
\int_{0}^t\varphi_n(t-s)\int_{\mathbb{R}^3}L^{\star}_{2,j,k}(x-(t-s)v,v,s,t)dsdv&=I_{n,j,k}^{\star}(x,t)+Rem_{n,j,k}^{\star}(x,t)
\end{split}
\end{equation*}
where $Rem_{n,j,k}^{\star}$ denote acceptable terms,
\begin{equation}\label{change17}
\begin{split}
\text{ if }\,\,\,j+k\ge -\delta m/3\,\,\,\text{ then }\,\,\,\Vert \langle t\rangle Rem_{n,j,k}^{I}\Vert_{B_t^0}&\lesssim\varepsilon_1^2 2^{-(k+j)+ \delta m/10}, \\   
\text{ if }\,\,\,j+k\le -\delta m/3\,\,\,\text{ then }\,\,\,\Vert \langle t\rangle Rem_{n,j,k}^{II}\Vert_{B_t^0}&\lesssim\varepsilon_1^22^{  \delta m/10}. 
\end{split}
\end{equation}
The main terms $I^{\star}_{n,j,k}, \star\in \{I,II\}$, can be written as a linear superposition of terms of the form
\begin{equation}\label{ModelTerms}
\begin{split}
& I^{\star}_{n,j,k} =\sum_{k_{1}, k_{2}\in\Z} 2^{-5j-k-n}\iint_{\{0\le s\le \tau\le t\}}\varphi_n(t-s)\int_{\mathbb{R}^3}\mathfrak{I}^{\star}_{kk_1k_2}[P_{k_1}\rho,P_{k_2}\rho]\, dvdsd\tau,\\
&\mathfrak{I}^{I}_{kk_1k_2}[f,g](x,v, \tau,  s, t) = \mathfrak{I}_{kk_1k_2}[f,g](x,v, \tau,  s, t), \\ 
 &\mathfrak{I}^{II}_{kk_1k_2}[f,g](x,v, \tau,  s, t) = \int_{\R^3} \mathcal{K}^{II}_{j,k}(x-y, v) \mathfrak{I}_{kk_1k_2}[f,g](y,v, \tau,  s, t) d y, \\ 
&\mathfrak{I}_{kk_1k_2}[f,g](x,v, \tau,  s, t) =\omega_{k_1}(\tau,s )\omega_{k_2}(\tau,s)\cdot b(v,s,t)\widetilde{\varphi}_{[j-2,j+2]}(v)\\
 &\qquad\qquad\qquad\qquad\times P_{[k-2,k+2]}\mathcal{R}^1\left[(\mathcal{R}^2f)(x-(t-s)v,s)\cdot(\mathcal{R}^3g)(x-(t-\tau)v,\tau)\right],
\end{split}
\end{equation}
where $\mathcal{R}^a$ denote normalized Calderon-Zygmund operators (as defined in \eqref{HormMich}) and
\begin{equation}\label{ModelTermsC}
\begin{split}
\omega_{k_1}(\tau, s)\in\{\tau-s,2^{-k_1}\},\qquad\omega_{k_2}(\tau,s)\in\{\tau-s,2^{-k_2}\},\qquad \Vert b\Vert_{L_{v,s,t}^{\infty}}+2^n\Vert \partial_sb\Vert_{L_{v,s,t}^{\infty}}\lesssim 1.
\end{split}
\end{equation}
\end{lemma}

\begin{proof} Recall \eqref{april7eqn1}. Notice that the only difference between $L^{I}_{2,j,k}(x,v,s,t)$ and $L^{II}_{2,j,k}(x,v,s,t)$ lies in the kernels, which play a minor role, so it would be sufficient to consider  $L^{I}_{2,j,k}(x,v,s,t)$ in detail here.  We observe that for any function $F$,
\begin{equation*}
\begin{split}
 \int_{\mathbb{R}^3}P_k F(x-(t-s)v,v,s,t)dv&=\frac{1}{2^k(t-s)}\int_{\mathbb{R}^3}\left[(t-s)\partial_{x^r}+\partial_{v^r}\right] P^{r}_kF(x-(t-s)v,v,s,t)\,dv\\
&=\frac{1}{2^k(t-s)}\int_{\mathbb{R}^3}P^{r}_k(\partial_{v^r}F)(x-(t-s)v,v,s,t)\,dv,\\
\end{split}
\end{equation*}
where $P_k^r:=-2^k\partial_{x^r}|\nabla|^{-2}P_k$ has the same properties as $P_k$. Using \eqref{qwp1} we calculate
\begin{equation}\label{dL2dv}
\begin{split}
\partial_{v^r}&L_{2,j}(x,v,s,t)=\big\{E(x,s)-E(x+\widetilde{Y}(x-sv,v,s,t),s)\big\}\cdot \partial_{v^r}M'_j(v)\\
&\quad+E(x+\widetilde{Y}(x-sv,v,s,t),s)\cdot\big[\partial_{v^r}M'_j(v)-\partial_{v^r}M'_j(v+\widetilde{W}(x-sv,v,s,t))\big]\\
&\quad+\nu^q_{r}(x,v,s,t)\partial_{x^q}E(x+\widetilde{Y}(x-sv,v,s,t),s)\cdot M'_j(v+\widetilde{W}(x-sv,v,s,t))\\
&\quad+\mu_r^q(x,v,s,t)E(x+\widetilde{Y}(x-sv,v,s,t),s)\cdot\partial_{v^q}M'_j(v+\widetilde{W}(x-sv,v,s,t)),
\end{split}
\end{equation}
where
\begin{equation*}
\begin{split}
&\nu_r^q(x,v,s,t):=\big[(s\partial_{x^r}-\partial_{v^r})\widetilde{Y}^q\big](x-sv,v,s,t),\\
&\mu_r^q(x,v,s,t):=\big[\left(s\partial_{x^r}-\partial_{v^r}\right)\widetilde{W}^q\big](x-sv,v,s,t). 
\end{split}
\end{equation*}
As a result of direct computations using \eqref{Lan7}, we have 
\begin{align}
 &\nu^q_r(x,v,s,t) = -
 \int_s^t(\tau-s)^2 \p_{x^r} E^q(x -(s-\tau) v+\widetilde{Y}(x-sv ,v,\tau,t),\tau)\,d\tau\\
&\quad+\int_s^t(\tau-s) (s  \p_{x^r}-\p_{v^r})  \widetilde{Y}^{q'}(x-sv ,v,\tau,t ) \p_{x^{q'}} E^q(x -(s-\tau)  v+\widetilde{Y}(x -sv,v,\tau,t),\tau) d\tau,\\
&\mu^q_r(x,v,s,t)= \int_s^t(\tau-s)   \p_{x^r} E^q(x  -(s-\tau) v+\widetilde{Y}(x-sv ,v,\tau,t),\tau)\,d\tau\\
&\quad - \int_s^t(s\p_{x^r}-\p_{v^r})  \widetilde{Y}^{q'}(x-sv ,v,\tau,t ) \p_{x^{q'}} E^q(x  -(s-\tau) v+\widetilde{Y}(x -sv,v,\tau,t),\tau) d\tau.  
\end{align}
After localizing the frequencies of the two electric fields, we have
\begin{equation*}
\begin{split}
\partial_{v^r}L_{2,j,k}(x,v,s,t)&= \sum_{k_1, k_2 \in \Z }\int_{s}^t P_k \left[\mathcal{M}_{k_1,k_2} \right](x,v,\tau,s,t) d\tau, 
\end{split}
\end{equation*}
where 
\begin{equation}\label{april5eqn45}
 \mathcal{M}_{k_1,k_2}=\mathcal{M}^0_{k_1,k_2}+\sum_{i=1}^{13}\mathcal{M}_{k_1,k_2}^{i},
\end{equation}
and, with $E_l:=P_lE$,
\begin{equation*}
\begin{split}
\mathcal{M}_{k_1, k_2}^0(x,v,\tau,s,t)&:=-(\tau-s)\partial_{x^q}E_{k_1}(x,s)E_{k_2}^q(x-(s-\tau)v,\tau)\cdot\partial_{v^r}M'_j(v)\\
&+E_{k_1}(x,s)E_{k_2}^p(x-(s-\tau)v,\tau)\cdot\partial_{v^r}\partial_{v^p}M'_j(v)\\
&-(\tau-s)^2\partial_{x^q}E_{k_1}(x,s)\partial_{x^r}E_{k_2}^q(x-(s-\tau)v,\tau)\cdot M'_j(v)\\
&+(\tau-s)E_{k_1}(x,s)\partial_{x^r}E_{k_2}^q(x-(s-\tau)v,\tau)\cdot\partial_{v^q}M'_j(v),
\end{split}
\end{equation*}
\begin{equation*}
\begin{split}
&\mathcal{M}_{k_1, k_2}^1(x,v,\tau,s,t):= -(\tau-s) \int_0^1\big[\p_{x^q} E_{k_1}(x+ \theta \widetilde{Y}(x-sv,v,s,t) ,s)- \p_{x^q} E_{k_1}(x,s)\big]\\
&\qquad\times E^q_{k_2} (x-(s-\tau)v+\widetilde{Y}(x-sv,v,\tau,t),\tau)\cdot\partial_{v^r}M'_j(v)\,d \theta,
  \end{split}
  \end{equation*}
\begin{equation*}
\begin{split}
&\mathcal{M}_{k_1, k_2}^2(x,v,\tau,s,t):= -(\tau-s)\p_{x^q} E_{k_1}(x,s)\\
&\qquad\times \big[E^q_{k_2} (x-(s-\tau)v+\widetilde{Y}(x-sv,v,\tau,t),\tau)-E^q_{k_2} (x-(s-\tau)v,\tau)\big]\cdot \partial_{v^r}M'_j(v),
  \end{split}
  \end{equation*}
\begin{equation*}
\begin{split}
&\mathcal{M}_{k_1, k_2}^3(x,v,\tau,s,t):=\int_0^1 \big[\p_{v^q} \partial_{v^r}M'_j(v+ \theta  \widetilde{W}(x-sv,v,s,t))-\partial_{v^r}\partial_{v^q}M'_j(v)\big]\\
&\qquad\cdot E_{k_1}(x+\widetilde{Y}(x-sv,v,s,t),s)E_{k_2}^q  (x-(s-\tau)v+\widetilde{Y}(x-sv,v,\tau,t),\tau)\,d\theta,
\end{split}
\end{equation*}
\begin{equation*}
\begin{split}
&\mathcal{M}_{k_1, k_2}^4(x,v,\tau,s,t):=\int_0^1\p_{v^q} \partial_{v^r}M'_j(v)\cdot \big[E_{k_1}(x+\widetilde{Y}(x-sv,v,s,t),s)-E_{k_1}(x,s)\big]\\
&\qquad\times E_{k_2}^q  (x-(s-\tau)v+\widetilde{Y}(x-sv,v,\tau,t),\tau)\,d \theta,
\end{split}
\end{equation*}
\begin{equation*}
\begin{split}
&\mathcal{M}_{k_1, k_2}^5(x,v,\tau,s,t):=\int_0^1 \p_{v^q} \partial_{v^r}M'_j(v)\cdot E_{k_1}(x,s) \\
&\qquad\times \big[E_{k_2}^q  (x-(s-\tau)v+\widetilde{Y}(x-sv,v,\tau,t),\tau)-E_{k_2}^q(x-(s-\tau)v,\tau)\big]\,d \theta,
\end{split}
\end{equation*}
\begin{equation*}
\begin{split}
&\mathcal{M}_{k_1, k_2}^6(x,v,\tau,s,t):=-(\tau-s)^2M'_j(v+\widetilde{W}(x-sv,v,s,t))\cdot\partial_{x^q}E_{k_1}(x+\widetilde{Y}(x-sv,v,s,t),s)\\
&\qquad\times\big[\partial_{x^r}E_{k_2}^q(x-(s-\tau)v+\widetilde{Y}(x-sv,v,\tau,t),\tau)-\partial_{x^r}E_{k_2}^q(x-(s-\tau)v,\tau)\big],
\end{split}
\end{equation*}
\begin{equation*}
\begin{split}
&\mathcal{M}_{k_1, k_2}^7(x,v,\tau,s,t):=-(\tau-s)^2 M'_j(v+\widetilde{W}(x-sv,v,s,t))\\
&\qquad\cdot\big[\partial_{x^q}E_{k_1}(x+\widetilde{Y}(x-sv,v,s,t),s)-\partial_{x^q}E_{k_1}(x,s)\big]\partial_{x^r}E_{k_2}^q(x-(s-\tau)v,\tau),
\end{split}
\end{equation*}
\begin{equation*}
\begin{split}
&\mathcal{M}_{k_1, k_2}^8(x,v,\tau,s,t):=-(\tau-s)^2 \big[M'_j(v+\widetilde{W}(x-sv,v,s,t))-M'_j(v)\big]\\
&\qquad\cdot\partial_{x^q}E_{k_1}(x,s)\partial_{x^r}E_{k_2}^q(x-(s-\tau)v,\tau),
\end{split}
\end{equation*}
\begin{equation*}
\begin{split}
&\mathcal{M}_{k_1, k_2}^9(x,v,\tau,s,t):=(\tau-s)\big[\partial_{v^q}M'_j(v+\widetilde{W}(x-sv,v,s,t))-\partial_{v^q}M'_j(v)\big]\\
&\qquad\cdot E_{k_1}(x+\widetilde{Y}(x-sv,v,s,t),s)\p_{x^r} E^q_{k_2}(x  -(s-\tau) v+\widetilde{Y}(x-sv ,v,\tau,t),\tau),
\end{split}
\end{equation*}
\begin{equation*}
\begin{split}
&\mathcal{M}_{k_1, k_2}^{10}(x,v,\tau,s,t):=(\tau-s)\partial_{v^q}M'_j(v)\cdot \big[E_{k_1}(x+\widetilde{Y}(x-sv,v,s,t),s)-E_{k_1}(x,s)\big]\\
&\qquad\times\p_{x^r} E^q_{k_2}(x  -(s-\tau) v+\widetilde{Y}(x-sv ,v,\tau,t),\tau),
\end{split}
\end{equation*}
\begin{equation*}
\begin{split}
&\mathcal{M}_{k_1, k_2}^{11}(x,v,\tau,s,t):=(\tau-s)\big[\partial_{v^q}M'_j(v)\cdot E_{k_1}(x,s)\\
&\qquad\times\big[\p_{x^r} E^q_{k_2}(x  -(s-\tau) v+\widetilde{Y}(x-sv ,v,\tau,t),\tau)-\partial_{x^r}E_{k_2}^q(x-(s-\tau)v,\tau)\big],
\end{split}
\end{equation*}
\begin{equation*}
\begin{split}
&\mathcal{M}_{k_1, k_2}^{12}(x,v,\tau,s,t):=(\tau-s) (s  \p_{x^r}-\p_{v^r})  \widetilde{Y}^{q'}(x-sv ,v,\tau,t )\partial_{x^q}E_{k_1}(x+\widetilde{Y}(x-sv,v,s,t),s)\\
&\qquad\cdot\p_{x^{q'} } E^q_{k_2}(x -(s-\tau)  v+\widetilde{Y}(x-sv,v,\tau,t), \tau)M'_j(v+\widetilde{W}(x-sv,v,s,t),
\end{split}
\end{equation*}
\begin{equation*}
\begin{split}
&\mathcal{M}_{k_1, k_2}^{13}(x,v,\tau,s,t):=-(s\p_{x^r}-\p_{v^r})  \widetilde{Y}^{q'}(x-sv ,v,\tau,t ) E_{k_1}(x+\widetilde{Y}(x-sv,v,s,t),s)\\
&\qquad\cdot \p_{x^{q'}} E^q_{k_2}(x  -(s-\tau) v+\widetilde{Y}(x -sv,v,\tau,t),\tau)\partial_{v^q}M'_j(v+\widetilde{W}(x-sv,v,s,t)).
\end{split}
\end{equation*}

Direct inspection shows that all terms in $\mathcal{M}^0_{k_1, k_2}$ lead to terms of the form \eqref{ModelTerms} as claimed, whereas for the error terms $\mathcal{M}_{k_1, k_2}^i$, $i\in\{1,\ldots,13\}$, we can apply Lemma \ref{PerturbativeTermsR12} to conclude.
\end{proof}

We focus now on the terms in \eqref{ModelTerms} and prove the following lemma.

\begin{lemma}\label{maintypeI}
With the hypothesis of Lemma \ref{LemReductionR12} and the notations inside its proof, if $m\geq\delta^{-4}$, $j\in[0,19m/20]$, $k\in\Z$, $n\in[10,m+2]$ then
\begin{equation}\label{april13eqn21}
\begin{split}
\text{ if }\,\,\,j+k\ge -\delta m/3\,\,\,\text{ then }\,\,\,\Vert \langle t\rangle I^{I}_{n,j,k}\Vert_{B^0_t}   &\lesssim \varepsilon_1^22^{-(k+j)+\delta m /2} ,\\ 
\text{ if }\,\,\,j+k\le -\delta m/3\,\,\,\text{ then }\,\,\,\Vert \langle t\rangle I^{II}_{n,j,k}\Vert_{B^0_t}   &\lesssim \varepsilon_1^22^{\delta m /2}. 
\end{split}
\end{equation}

\end{lemma}
\begin{proof} We focus on the estimate of $I^{I}_{n,j,k}$. With minor modifications, the estimate of $I^{II}_{n,j,k}$ can be obtained as a byproduct.  

Without loss of generality, we  assume that $k_1\le k_2$ and apply the decomposition  \eqref{sug31} for the first input (in the case $k_2\leq k_1$ we apply the  decomposition in \eqref{sug31} for the second input, use the same strategy with the only difference that we do integration by parts in $\tau$ instead of $s$ for the oscillatory part, which gives us better estimates since $\tau\geq s$).

Note that $k\leq k_2+10.$ Moreover, we use the following partition to localize the size of $\tau$, 
\be\label{april13eqn66}
\mathbf{1}_{[0, t]}(\tau) = \sum_{0\le \tilde{m}\le  m+2} \widetilde{\varphi}_{\tilde{m}}(\tau)  \mathbf{1}_{[0, t]}(\tau). 
\ee 
For simplicity of notation, let $\rho_l:=P_l\rho$, $\rho_l^{stat}:=P_l(\rho^{stat})$, $\rho_l^{osc}:=P_l(\rho^{osc})$.

{\bf{Step $1$: the static component.}} We will show that
\begin{equation}\label{change50}
\begin{split}
&\sum_{k_1,k_2\in\Z,\,k_2\geq k-10}\sum_{\tilde{m}\in[0,m+2]}2^{-4j-n}\int_{0}^t \int_s^t\varphi_n(t-s)\widetilde{\varphi}_{\tilde{m}}(\tau)\big[\big\|\mathfrak{J}_{kk_1k_2}[\rho_{k_1}^{stat},\rho_{k_2}](\tau,s,t)\big\|_{L^1_xL^1_v}\\
&\qquad\qquad+2^{3m}\big\|\mathfrak{J}_{kk_1k_2}[\rho_{k_1}^{stat},\rho_{k_2}](\tau,s,t)\big\|_{L^\infty_xL^1_v}\big]\, d \tau  d s\lesssim \varep_1^22^{-m+\delta m/2}.
\end{split}
\end{equation}

A simple estimate using \eqref{Laga2}-\eqref{Laga3} and the fact that $k_2\ge k-10$ shows that
\begin{equation}\label{april13eqn31}
\begin{split}
&2^{-4j-n}\int_{0}^t \int_s^t  \varphi_n(t-s)   \widetilde{\varphi}_{\tilde{m}}(\tau)  \big\|\mathfrak{J}_{kk_1k_2}[\rho_{k_1}^{stat},\rho_{k_2}](\tau,s,t)\big\|_{L^1_{x,v}} d \tau  d s\\
&\lesssim 2^{-j-n} \int_{0}^t \int_s^t \varphi_n(t-s) \widetilde{\varphi}_{\tilde{m}}(\tau) 2^{2\max\{-k_{1}, \tilde{m}\}}\\
&\qquad\qquad\times\min\{\Vert \rho^{stat}_{k_1}(s)\Vert_{L^\infty}\Vert \rho_{k_2}(\tau)\Vert_{L^1},\Vert \rho^{stat}_{k_1}(s)\Vert_{L^1}\Vert\rho_{k_2}(\tau)\Vert_{L^\infty}\} \,d \tau ds\\
&\lesssim  2^{-j-k_2-n}\int_{0}^t \int_s^t \varepsilon_1^2  \varphi_n(t-s) \widetilde{\varphi}_{\tilde{m}}(\tau)  2^{2\max\{-k_{1}, \tilde{m}\}} 2^{3\min\{k_{1}, -\tilde{m}\} }\langle s\rangle^{2\delta-1}\langle \tau\rangle^{2\delta-1} d \tau d s \\
&\lesssim \varepsilon_1^2 2^{-j-k_2}2^{\min\{k_{1}, -\tilde{m}\}}2^{4\delta \tilde{m}}2^{n-2m},
\end{split}
\end{equation}
where the last inequality follows by considering the two possible cases $(n\leq m-10\text{ and }\widetilde{m}\geq m-10)$ and $n\geq m-10$. One can sum over $k_1,k_2,\tilde{m}$ and recall the assumption $j+k\geq -\delta m/3$ to see that this gives an acceptable contribution to \eqref{change50}.

To bound $L^\infty$ norms we use the general estimate
\begin{equation}\label{change34}
\begin{split}
\int_{\R^3}|f(x-\lambda_1v)|&|g(x-\lambda_2v)|\varphi(2^{-j}v)\,dv\\
&\lesssim \min\big\{2^{3j}\|f\|_{L^\infty}\|g\|_{L^\infty},|\lambda_1|^{-3}\|f\|_{L^1}\|g\|_{L^\infty},|\lambda_2|^{-3}\|f\|_{L^\infty}\|g\|_{L^1}\big\},
\end{split}
\end{equation}
for any functions $f,g\in (L^1\cap L^\infty)(\R^3)$, and any $x\in\R^3$, $\lambda_1,\lambda_2\in\R$, $j\in\Z_+$. 

Using this estimate we bound first the contribution of large times $s$,
\begin{equation}\label{april12eqn1}
\begin{split}
&2^{-4j-n}\int_{t/4}^t \int_s^t  \varphi_n(t-s)   \widetilde{\varphi}_{\tilde{m}}(\tau)  2^{3m}\big\Vert \mathfrak{J}_{kk_1k_2} [\rho_{k_1}^{stat},\rho_{k_2}](\tau,s,t)\big\Vert_{L^\infty_xL^1_v} d \tau  d s\\
&\lesssim 2^{-j-n+3m} \int_{t/4}^t \int_s^t \varphi_n(t-s)   \widetilde{\varphi}_{\tilde{m}}(\tau) 
(2^{-k_1}+\tau)^2\Vert\rho_{k_1}^{stat}(s)\Vert_{L^\infty}\Vert\rho_{k_2}(\tau)\Vert_{L^\infty}\,d \tau  d s\\
&\lesssim \varepsilon_1^22^{-j-k_2}2^{\min(k_1,-m)}2^{n-2m+4\delta m}. 
\end{split}
\end{equation}
Similarly, letting $\bar{s}:=\max\{s,  \min\{2^{-k_1}, t/2\}\}$, we use again \eqref{change34} and \eqref{Laga2}-\eqref{Laga3} to estimate
\begin{equation}\label{april12eqn2}
\begin{split}
&2^{-4j-n}\int_0^{t/4}  \int_{\bar{s} }^t  \varphi_n(t-s)   \widetilde{\varphi}_{\tilde{m}}(\tau)  2^{3m}\Vert \mathfrak{J}_{kk_1k_2} [ \rho^{stat}_{k_1},\rho_{k_2}](\tau,s,t)\Vert_{L^\infty_xL^1_v} d \tau  d s\\
&\lesssim  2^{-j+2m}\int_0^{t/4}   \int_{ \bar{s} }^t  \widetilde{\varphi}_{\tilde{m}}(\tau) (2^{-k_1}+\tau)^2\min(2^{3k_1},2^{-3m})\Vert\rho^{stat}_{k_1}(s)\Vert_{L^1}\Vert\rho_{k_2}(\tau)\Vert_{L^\infty} d \tau d s\\
&\lesssim \varepsilon_1^22^{-j-k_2}2^{\min(k_1,-\tilde{m})}2^{-m+4\delta \tilde{m}}.
\end{split}
\end{equation}
and
\begin{equation}\label{april12eqn3}
\begin{split}
&2^{-4j-n}\int_0^{t/4}  \int_s^{\bar{s} }\varphi_n(t-s)   \widetilde{\varphi}_{\tilde{m}}(\tau)  2^{3m}\Vert \mathfrak{J}_{kk_1k_2} [ \rho^{stat}_{k_1},\rho_{k_2}](\tau,s,t)\Vert_{L^\infty_xL^1_v} d \tau  d s\\
&\lesssim  2^{-4j+2m}\int_0^{t/4}   \int_s^{ \bar{s} }  \widetilde{\varphi}_{\tilde{m}}(\tau) (2^{-k_1}+\tau)^2\Vert\rho^{stat}_{k_1}(s)\Vert_{L^\infty}2^{-3m}\Vert\rho_{k_2}(\tau)\Vert_{L^1} d \tau d s\\
&\lesssim \varepsilon_1^22^{-j-k_2}2^{\min(k_1,-\tilde{m})}2^{-m+4\delta \tilde{m}}.
\end{split}
\end{equation}
The desired bounds \eqref{change50} follow from \eqref{april13eqn31} and \eqref{april12eqn1}, \eqref{april12eqn2}, \eqref{april12eqn3}.

{\bf{Step 2: the oscillatory component.}} We now consider the contribution of the oscillatory field. We estimate first as in \eqref{april13eqn31}, \eqref{april12eqn1}, \eqref{april12eqn2}, \eqref{april12eqn3}, with additional factors of $\min(2^{-k_1},\langle\tau\rangle)$ coming from \eqref{Laga2}--\eqref{Laga3}. The result is
\begin{equation}\label{change21}
\begin{split}
&2^{-4j-n}\int_{0}^t \int_s^t\varphi_n(t-s)\widetilde{\varphi}_{\tilde{m}}(\tau)\big[\big\|\mathfrak{J}_{kk_1k_2}[e^{-i\iota s}\rho_{k_1}^{osc,\iota},\rho_{k_2}](\tau,s,t)\big\|_{L^1_xL^1_v}\\
&\qquad\qquad+2^{3m}\big\|\mathfrak{J}_{kk_1k_2}[e^{-i\iota s}\rho_{k_1}^{osc,\iota},\rho_{k_2}](\tau,s,t)\big\|_{L^\infty_xL^1_v}\big]\, d \tau  d s\\
&\lesssim \varepsilon_1^2 2^{-j-k_2}2^{\min\{k_{1}, -\tilde{m}\}}2^{-m+4\delta \tilde{m}}2^{\min\{\tilde{m},-k_1\}},
\end{split}
\end{equation}
where $\iota\in\{-1,1\}$, $\rho_{l}^{osc,+}:=\rho_l^{osc}$, $\rho_{l}^{osc,-}:=\overline{\rho_l^{osc}}$.
This suffices to bound the contributions corresponding to either $|k_1+\tilde{m}|\geq 4\delta\tilde{m}$, or $k_2\geq k+4\delta\tilde{m}$, or $\tilde{m}\leq m/40$. After these reductions it remains to prove that for $q\in\{1,\infty\}$ we have
\begin{equation}\label{change22}
\begin{split}
&2^{-4j-n}2^{3m(1-1/q)}\Big\|\int_{0}^t \int_s^t\int_{\R^3}\varphi_n(t-s)\varphi_{\tilde{m}}(\tau)\\
&\qquad\qquad\times\mathfrak{J}_{kk_1k_2}[e^{-\iota is}\rho_{k_1}^{osc,\iota},\rho_{k_2}](x,v,\tau,s,t)\, dvd \tau  d s\Big\|_{L^q_x}\lesssim 2^{-m+2\delta m/5},
\end{split}
\end{equation}
provided that $m\geq\delta^{-4}$, $j\in[0,19m/20]$, $n\in[10,m+2]$, $j+k\geq-\delta m/3$, and
\begin{equation}\label{change23}
\tilde{m}\in[m/40,m+2],\qquad k_1\in[-\tilde{m}-4\delta\tilde{m},-\tilde{m}+4\delta\tilde{m}],\qquad k_2\in[k-10,k+4\delta\tilde{m}].
\end{equation}

By taking complex conjugates, in proving \eqref{change22} we may assume that $\iota=+1$. Let 
\be\label{april12eqn31}
\rho_{l,\leq p}^{osc}:=\Pi_{-1,l,\leq p}\rho^{osc},\qquad\rho_{l,>p}^{osc}:=\Pi_{-1,l,>p}\rho^{osc},\qquad \Gamma_{l,>p}:=\Pi_{-1,l,>p}\omega^{-1}_{-1,0}\rho^{osc},
\ee
for $l\in\Z$ and $p\leq -4$, where the operators $\Pi_{\pm 1,l,\leq p}$, $\Pi_{\pm1,l,>p}$, $\omega^{-1}_{\pm 1,0}$ are defined in \eqref{omeio0}--\eqref{omeio1}.

To bound the contribution of $\rho_{k_1,>p_0}^{osc}$ we would like to integrate by parts in $s$ (the method of normal forms). We start from the identity
\begin{equation}\label{change24}
\begin{split}
&\int_{0}^\tau\int_{\R^3}\varphi_n(t-s)\mathfrak{J}_{kk_1k_2}[e^{-is}\rho_{k_1,>p_0}^{osc},\rho_{k_2}](x,v,\tau,s,t)\, dvd s\\
&=\frac{1}{(2\pi)^3}\int_{\R^3}\int_{\R^3}\int_{\R^3}\int_{0}^\tau \varphi_n(t-s)\omega_{k_1}(\tau,s)\omega_{k_2}(\tau,s)b(v,s,t)\widetilde{\varphi}_{[j-2,j+2]}(v)A_k(x-y)\\
&\qquad\times e^{-is(1-v\cdot\xi)}\widehat{\mathcal{R}^2\rho_{k_1}^{osc}}(\xi,s)\varphi_{>p_0}(1-v\cdot\xi)e^{i(y-tv)\cdot\xi}(\mathcal{R}^3\rho_{k_2})(y-(t-\tau)v,\tau)\,d s d\xi dvdy,
\end{split}
\end{equation}
where the kernel $A_k$ is defined by $\widehat{A_k}(\eta)=\varphi_{[k-2,k+2]}(\eta)\mathcal{R}^1(\eta)$. Then we use the identity
\begin{equation}\label{IBPTau}
\begin{split}
&\int_{0}^\tau \varphi_n(t-s)\omega_{k_1}(\tau,s)\omega_{k_2}(\tau,s)b(v,s,t)e^{-is(1-v\cdot\xi)}\widehat{\mathcal{R}^2\rho_{k_1}^{osc}}(\xi,s)\varphi_{>p_0}(1-v\cdot\xi)\,d s\\
&=\int_{0}^\tau\frac{e^{-is(1-v\cdot\xi)}\varphi_{>p_0}(1-v\cdot\xi)}{i(1-v\cdot\xi)}\big[B(v,\tau,s,t)(\partial_s\widehat{\mathcal{R}^2\rho_{k_1}^{osc}})(\xi,s)+(\partial_sB)(v,\tau,s,t)\widehat{\mathcal{R}^2\rho_{k_1}^{osc}}(\xi,s)\big]\\
&+\frac{\varphi_{>p_0}(1-v\cdot\xi)}{i(1-v\cdot\xi)}B(v,\tau,0,t)\widehat{\mathcal{R}^2\rho_{k_1}^{osc}}(\xi,0)-\frac{e^{-i\tau(1-v\cdot\xi)}\varphi_{>p_0}(1-v\cdot\xi)}{i(1-v\cdot\xi)}B(v,\tau,\tau,t)\widehat{\mathcal{R}^2\rho_{k_1}^{osc}}(\xi,\tau),
\end{split}
\end{equation}
where
\begin{equation}\label{AddedBoundsBF23}
B(v,\tau,s,t)=B_{n,k_1,k_2}(v,\tau,s,t):=\varphi_n(t-s)\omega_{k_1}(\tau,s)\omega_{k_2}(\tau,s)b(v,s,t).
\end{equation}
We combine the formulas \eqref{change24}--\eqref{IBPTau} and examine the $\xi$-integral. Notice that
\begin{equation*}
\begin{split}
\frac{1}{(2\pi)^3}\int_{\R^3}&\frac{e^{isv\cdot\xi}\varphi_{>p_0}(1-v\cdot\xi)}{i(1-v\cdot\xi)}e^{i(y-tv)\cdot\xi}\widehat{F_{k_1}}(\xi,s)\,d\xi=-i\big(\Pi_{-1,k_1,>p_0}\omega^{-1}_{-1,0}F\big)(y-(t-s)v,v,s),
\end{split}
\end{equation*}
according to the definitions \eqref{omeio0}--\eqref{omeio1}, where $F\in\{\mathcal{R}^2\rho^{osc},\partial_s\mathcal{R}^2\rho^{osc}\}$. Thus
\begin{equation}\label{change26}
\begin{split}
&\int_{0}^\tau\int_{\R^3}\varphi_n(t-s)\mathfrak{J}_{kk_1k_2}[e^{-is}\rho_{k_1,>p_0}^{osc},\rho_{k_2}](x,v,\tau,s,t)\, dvd s=\int_{0}^\tau\int_{\R^3}\int_{\R^3} (-i)e^{-is}\\
&\qquad\times A_k(x-y)\widetilde{\varphi}_{[j-2,j+2]}(v)\cdot(\mathcal{R}^3\rho_{k_2})(y-(t-\tau)v,\tau)X_{n,k_1,k_2}(y,v,\tau,s,t)\,dvdyds,
\end{split}
\end{equation}
where the functions $\Gamma_{k_1,>p_0}$ are defined as in \eqref{april12eqn31}, and
\begin{equation}\label{change27}
\begin{split}
X_{n,k_1,k_2}(y,v,\tau,s,t):&=B_{n,k_1,k_2}(v,\tau,s,t)(\mathcal{R}^2\partial_s\Gamma_{k_1,>p_0})(y-(t-s)v,v,s)\\
&+B'_{n,k_1,k_2}(v,\tau,s,t)(\mathcal{R}^2\Gamma_{k_1,>p_0})(y-(t-s)v,v,s),\\
B'_{n,k_1,k_2}(v,\tau,s,t):&=(\partial_sB_{n,k_1,k_2})(v,\tau,s,t)+B_{n,k_1,k_2}(v,\tau,s,t)[\delta_0(s)-\delta_0(s-\tau)].
\end{split}
\end{equation}

{\bf{Case 1: the $L^1$ bounds.}} We prove first the bounds \eqref{change22} when $q=1$.  Let $p_0:=-100\delta \tilde{m}$. We use Lemma \ref{omegaioLem} (i) (so $\rho_{k_1,\leq p_0}^{osc}\equiv 0$ unless $j+k_1\geq -4$) and estimate as in \eqref{april13eqn31},
\begin{equation}\label{april13eqn32}
\begin{split}
&2^{-4j-n} \int_{0}^t \int_s^t \varphi_n(t-s)\varphi_{\tilde{m}}(\tau)  \big\Vert \mathfrak{J}_{kk_1k_2} [e^{-is}\rho_{k_1,\leq p_0}^{osc},\rho_{k_2}] \big\Vert_{L^1_xL^1_v}\,d \tau  d s\\
& \lesssim 2^{-j-n}\int_{0}^t \int_s^t  \varphi_n(t-s)\varphi_{\tilde{m}}(\tau)(2^{-k_1}+\tau)^2\Vert\rho_{k_1,\leq p_0}^{osc}(s)\Vert_{L^\infty_vL^\infty_x}\Vert \rho_{k_2}(\tau)\Vert_{L^1_x}\,d \tau ds \\
&\lesssim   \varepsilon_1^2 2^{-j-n-k_2}\int_{0}^t \int_s^t \varphi_n(t-s)\varphi_{\tilde{m}}(\tau)  (2^{-2k_1}+2^{2\tilde{m}})2^{3k_1+p_0}2^{-k_1} \langle s\rangle^{2\delta-1}\langle \tau\rangle^{2\delta-1} d \tau  d s \\
&\lesssim \, \varepsilon_1^22^{-j-k-m},
\end{split}
\end{equation}
using also the assumptions \eqref{change23}.

To bound the contribution of $\rho_{k_1,>p_0}^{osc}$ we use the formulas \eqref{change26}--\eqref{change27}. In view of \eqref{Laga3}--\eqref{Laga4} and Lemma \ref{omegaioLem} (ii) we estimate
\begin{equation}\label{change28}
\big\|\mathcal{R}^2\Gamma_{k_1,>p_0}(s)\big\|_{L^\infty_vL^1_y}+\langle s\rangle\big\|\mathcal{R}^2\partial_s\Gamma_{k_1,>p_0}(s)\big\|_{L^\infty_vL^1_y}\lesssim \varep_12^{-p_0}\langle s\rangle^{2\delta}.
\end{equation}
Therefore, using also the assumptions \eqref{change23} we estimate
\begin{equation*}
\begin{split}
\int_{0}^\tau\big\|\widetilde{\varphi}_{[j-2,j+2]}(v)&\cdot(\mathcal{R}^3\rho_{k_2})(y-(t-\tau)v,\tau)X_{n,k_1,k_2}(y,v,\tau,s,t)\big\|_{L^1_{y,v}}\,ds\\
&\lesssim \|\rho_{k_2}(\tau)\|_{L^\infty}2^{3j}\|X_{n,k_1,k_2}(y,v,\tau,s,t)\|_{L^\infty_vL^1_sL^1_y}\\
&\lesssim \varep_1^22^{3j}2^{-k_2}\langle\tau\rangle^{-4+2\delta}\cdot 2^{-p_0}\langle\tau\rangle^{2+12\delta}.
\end{split}
\end{equation*}
Therefore, using the formula \eqref{change26} and the definition \eqref{AddedBoundsBF23} we estimate
\begin{equation}\label{april13eqn33}
\begin{split}
&2^{-4j-n}\Big\|\int_{0}^t \int_s^t\int_{\R^3}\varphi_n(t-s)\varphi_{\tilde{m}}(\tau)\mathfrak{J}_{kk_1k_2}[e^{-is}\rho_{k_1,>p_0}^{osc},\rho_{k_2}](x,v,\tau,s,t)\, dvd \tau  d s\Big\|_{L^1_x}\\
&\lesssim 2^{-4j-n}\int_{0}^t \varphi_{\tilde{m}}(\tau)\varphi_{\leq n+2}(t-\tau)\\
&\qquad\qquad\times\int_0^\tau\big\|\widetilde{\varphi}_{[j-2,j+2]}(v)(\mathcal{R}^3\rho_{k_2})(y-(t-\tau)v,\tau)X_{n,k_1,k_2}(y,v,\tau,s,t)\big\|_{L^1_{y,v}} ds d\tau\\
&\lesssim \varep_1^22^{-j-k_2-n}2^{-p_0}2^{-2\tilde{m}+20\delta\tilde{m}}\min\{2^{\tilde{m}},2^n\}.
\end{split}
\end{equation}
The desired bounds \eqref{change22} follow in the case $q=1$, using also \eqref{april13eqn32}.

{\bf{Case 2: the $L^\infty$ bounds.}} We prove now the bounds \eqref{change22} when $q=\infty$. As before we  let $p_0=-100\delta \tilde{m}$,  and recall the definition \eqref{april12eqn31}. 

We  first  bound the contribution of the function $\rho_{k_1,\leq p_0}^{osc}$, which is nontrivial only if $j+k_1\geq -4$. For $s\in[0,t]$ let $s':=\max\{s,t/4\}$. Using Lemma \ref{omegaioLem2} (i) and the assumptions \eqref{change23} (in particular the integrals below are nontrivial only if $\tilde{m}\geq m-10$ and $|k_1+m|\leq 4\delta m$),
\begin{equation}\label{april12eqn7}
\begin{split}
&2^{-4j-n}2^{3m} \int_{0}^t \int_{s'}^t \varphi_n(t-s)\varphi_{\tilde{m}}(\tau)  \big\Vert \mathfrak{J}_{kk_1k_2} [e^{-is}\rho_{k_1,\leq p_0}^{osc},\rho_{k_2}] \big\Vert_{L^\infty_xL^1_v}\,d \tau  d s\\
&\lesssim  2^{-j-n}2^{3m}\int_{0}^t \int_{s'}^t \varphi_n(t-s)\varphi_{\tilde{m}}(\tau) \big[2^{-k_1}+\tau\big]^2 \Vert \rho_{k_1,\leq p_0}^{osc}(s)\Vert_{L^\infty_vL^\infty_x}\Vert \rho_{k_2}(\tau)\Vert_{L^\infty_x}d \tau  d s\\
&\lesssim  2^{-j-n-k_2}\int_{0}^t \int_{s'}^t \varphi_n(t-s) \varphi_{\tilde{m}}(\tau) \big[2^{-k_1}+\tau\big]^2 2^{3k_1+p_0}2^{-k_1} \langle s\rangle^{2\delta-1}\langle \tau\rangle^{2\delta-1}\,d \tau  d s\\
&\lesssim \, \varepsilon_1^22^{-j-k-m-10\delta \tilde{m}}. 
\end{split}
\end{equation}

This suffices to prove the desired bounds if $n\leq m-4$. In the remaining cases we may assume that $n\geq m-4$ and $s\leq t/4$, and we can estimate the remaining integral using \eqref{change34}, 
\begin{equation}\label{change29}
\begin{split}
&2^{-4j-n}2^{3m} \int_{0}^{t/4} \int_{s}^{t/4} \varphi_n(t-s)\varphi_{\tilde{m}}(\tau)  \big\Vert \mathfrak{J}_{kk_1k_2} [e^{-is}\rho_{k_1,\leq p_0}^{osc},\rho_{k_2}] \big\Vert_{L^\infty_xL^1_v}\,d \tau  d s\\
&\lesssim  2^{-4j-m}2^{3m}\int_{0}^{t/4} \int_s^{t/4}\varphi_{\tilde{m}}(\tau) \big[2^{-k_1}+\tau\big]^2 \Vert \rho_{k_1,\leq p_0}^{osc}(s)\Vert_{L^\infty_vL^\infty_x}2^{-3m}\Vert \rho_{k_2}(\tau)\Vert_{L^1_x}d \tau  d s\\
&\lesssim  2^{-4j-m-k_2}\int_{0}^{t/4} \int_{s}^{t/4} \varphi_{\tilde{m}}(\tau) \big[2^{-k_1}+\tau\big]^2 2^{3k_1+p_0}2^{-k_1} \langle s\rangle^{2\delta-1}\langle \tau\rangle^{2\delta-1}\,d \tau  d s\\
&\lesssim \, \varepsilon_1^22^{-j-k-m-10\delta \tilde{m}}. 
\end{split}
\end{equation} 

Finally, we consider the contribution from the oscillatory part $\rho_{k_1,>p_0}^{osc}$. As in the $L^1$-estimate, we integrate by parts in $s$ once. In view of the formula \eqref{change26}, for \eqref{change22} it suffices to prove that
\begin{equation}\label{change22.5}
\begin{split}
2^{-4j-n}&2^{3m}\Big\|\int_{0}^t \int_{0}^\tau\int_{\R^3}\varphi_{\tilde{m}}(\tau)\varphi_{[n-2,n+2]}(t-s)\big|\widetilde{\varphi}_{[j-2,j+2]}(v)(\mathcal{R}^3\rho_{k_2})(y-(t-\tau)v,\tau)\\
&\qquad\qquad\qquad\times X_{n,k_1,k_2}(y,v,\tau,s,t)\big|\,dvdsd\tau\Big\|_{L^\infty_y}\lesssim 2^{-m+2\delta m/5}.
\end{split}
\end{equation}

We consider two cases. If $n\leq m-6$ then we may assume that $\widetilde{m}\geq m-6$ and $t/2\leq s\leq\tau\leq t$.  We examine the identities \eqref{AddedBoundsBF23} and \eqref{change27}, and use Lemma \ref{omegaioLem2} (ii) to estimate
\begin{equation*}
\begin{split}
\|X_{n,k_1,k_2}(y,v,\tau,s,&t)\|_{L^\infty_{y,v}}\lesssim (2^{-k_1}+2^m)^22^{-p_0}\\
&\big[\|\partial_s\rho^{osc}_{k_1}(s)\|_{L^\infty}+2^{-n}\|\rho^{osc}_{k_1}(s)\|_{L^\infty}+\|\rho^{osc}_{k_1}(\tau)\|_{L^\infty}\delta_0(s-\tau)\big].
\end{split}
\end{equation*}
Thus, using \eqref{Laga2}--\eqref{Laga4} and \eqref{change23}, the left-hand side of \eqref{change22.5} is bounded by
\begin{equation*}
\begin{split}
C\varep_1^22^{-j-n}&2^{3m}\int_{0}^t\int_{s}^t\varphi_{\leq n+2}(t-s)2^{-k_2}\langle\tau\rangle^{-4+2\delta}2^{-p_0}(2^{-k_1}+2^m)^3\langle s\rangle^{-4+2\delta}(2^{-n}+\delta_0(s-\tau))d\tau ds\\
&\lesssim \varep_1^22^{-j-n-k}2^{-1.9m},
\end{split}
\end{equation*}
which clearly gives the desired bounds \eqref{change22.5} in this case.

On the other hand, if $n\geq m-6$ then we may assume that $t-s\geq 2^{m-6}$. We use first the dispersive estimates \eqref{omegaio2y}, so
\begin{equation*}
\begin{split}
\int_{\R^3}\varphi_{\leq j+4}(v)&\big[\langle s\rangle\big|(\mathcal{R}^2\partial_s\Gamma_{k_1,>p_0})(y-(t-s)v,v,s)\big|\\
&+\big|(\mathcal{R}^2\Gamma_{k_1,>p_0})(y-(t-s)v,v,s)\big|\big]\,dv\lesssim \varep_12^{-p_0}2^{-3m}\langle s\rangle^\delta.
\end{split}
\end{equation*}
Therefore, using the identities \eqref{AddedBoundsBF23} and \eqref{change27}, if $s\leq\tau\leq 2^{\tilde{m}+4}$,
\begin{equation*}
\|X_{n,k_1,k_2}(y,v,\tau,s,t)\|_{L^\infty_yL^1_v}\lesssim (2^{-k_1}+2^{\tilde{m}})^2\varep_12^{-p_0}2^{-3m}\big[\langle s\rangle^{2\delta-1}+\delta_0(s)+\langle \tau\rangle^{2\delta}\delta_0(s-\tau)\big].
\end{equation*}
Thus, using \eqref{Laga2}--\eqref{Laga3} and \eqref{change23}, the left-hand side of \eqref{change22.5} is bounded by
\begin{equation*}
\begin{split}
C\varep_1^22^{-4j-m}&2^{3m}\int_{0}^t\int_{0}^\tau \varphi_{\tilde{m}}(\tau)2^{-k_2}\langle\tau\rangle^{-4+2\delta}\\
&\times(2^{-k_1}+2^{\tilde{m}})^22^{-p_0}2^{-3m}\big[\langle s\rangle^{2\delta-1}+\delta_0(s)+\langle \tau\rangle^{2\delta}\delta_0(s-\tau)\big]\, dsd\tau\\
&\lesssim \varep_1^22^{-4j-k}2^{-m-\tilde{m}/2},
\end{split}
\end{equation*}
which clearly gives the desired bounds \eqref{change22.5} in this case.
\end{proof}

For the error type terms $\mathcal{M}_{k_1, k_2}^{l}$ appearing in \eqref{april5eqn45} we have: 

\begin{lemma}\label{PerturbativeTermsR12}
With the hypothesis of Lemma \ref{LemReductionR12} and the notations inside its proof, if $m\geq\delta^{-4}$, $k\in\Z$, $j\in[0,19m/20]$, $n\in[10,m+2]$, $q\in\{1,\infty\}$, and $l\in\{1,\ldots,13\}$ then
\begin{equation}\label{april10eqn66}
\begin{split}
&2^{j-n}2^{3m(1-1/q)}\sum_{k_1,k_2\in\Z}\Big\Vert\int_0^t\int_s^t\varphi_n(t-s)\\
&\qquad\qquad\qquad\times\int_{\mathbb{R}^3}P_k\mathcal{M}_{k_1, k_2}^{l}(x-(t-s)v,v,\tau,s,t)\,dvd\tau ds\Big\Vert_{L^q_x}\lesssim \varepsilon_1^2 2^{-m+\delta m/10}.
\end{split}
\end{equation}
\end{lemma}

\begin{proof} Notice that the bounds \eqref{april10eqn66} suffice to prove both estimates in \eqref{change17}: the estimates in the first line follow in the case $k+j+\delta m/3\geq 0$, while the estimates in the second line follow in the case $k+j+\delta m/3\geq 0$, using also the bounds \eqref{april10eqn11}.

Recall the definitions of the functions $\mathcal{M}_{k_1, k_2}^{l}$ following \eqref{april5eqn45}. Notice that factors $E_{k_1}$ and $E_{k_2}$ appear with either 0 or 1 $x$-derivatives, so the contributions decay sufficiently fast when one of the indices $k_1$ or $k_2$ is either too small, or too large. It suffices to estimate the core contributions, coming from $k_1,k_2\in[-10m,10m]$.  

Most of the terms can be estimated using Lemma \ref{keybilinearlemma}. Indeed, with the notation \eqref{ropi1}--\eqref{june16eqn4},
\begin{equation*}
\begin{split}
&\int_{\mathbb{R}^3}P_k\mathcal{M}_{k_1, k_2}^2(x-(t-s)v,v,\tau,s,t)\,dv\sim \mathcal{Q}_{j,k}(\nabla E_{k_1},\nabla E_{k_2};C^2)(x,\tau,s,t),\\
&\qquad C^2(x,y,v,\tau,s,t):=2^{3j}(\tau-s)\widetilde{Y}(y-tv,v,\tau,t)\nabla_v M'_j(v),
\end{split}
\end{equation*}
\begin{equation*}
\begin{split}
&\int_{\mathbb{R}^3}P_k\mathcal{M}_{k_1, k_2}^8(x-(t-s)v,v,\tau,s,t)\,dv\sim \mathcal{Q}_{j,k}(\nabla E_{k_1},\nabla E_{k_2};C^8)(x,\tau,s,t),\\
&\qquad C^8(x,y,v,\tau,s,t):=2^{3j}(\tau-s)^2 \big[M'_j(v+\widetilde{W}(y-tv,v,s,t))-M'_j(v)\big],
\end{split}
\end{equation*}
\begin{equation*}
\begin{split}
&\int_{\mathbb{R}^3}P_k\mathcal{M}_{k_1, k_2}^{10}(x-(t-s)v,v,\tau,s,t)\,dv\sim \mathcal{Q}_{j,k}(\nabla E_{k_1},\nabla E_{k_2};C^{10})(x,\tau,s,t),\\
&\qquad C^{10}(x,y,v,\tau,s,t):=2^{3j}(\tau-s)\widetilde{Y}(y-tv,v,s,t)\nabla_v M'_j(v),
\end{split}
\end{equation*}
\begin{equation*}
\begin{split}
&\int_{\mathbb{R}^3}P_k\mathcal{M}_{k_1, k_2}^{12}(x-(t-s)v,v,\tau,s,t)\,dv\sim \mathcal{Q}_{j,k}(\nabla E_{k_1},\nabla E_{k_2};C^{12})(x,\tau,s,t),\\
&\qquad C^{12}(x,y,v,\tau,s,t):=2^{3j}(\tau-s) (s\nabla_y-\nabla_v)  \widetilde{Y}(y-tv,v,\tau,t )M'_j(v+\widetilde{W}(y-tv,v,s,t)),
\end{split}
\end{equation*}
where the notation $\sim$ means that the function in the left-hand side can be written as a linear combination of functions in the right-hand side (various partial derivatives, various components, and in some cases, integrals from $0$ to $1$). Using \eqref{cui6}--\eqref{cui7.5}, we have $\Lambda(C^l)(\tau,s,t)\lesssim 2^{-j}\langle\tau\rangle^2\langle s\rangle^{-2+3\delta}$ for $l\in\{2,8,10,12\}$. Therefore, using \eqref{june16eqn33},
\begin{equation}\label{change61}
2^{3m(1-1/q)}\Big\Vert\int_s^t\int_{\mathbb{R}^3}P_k\mathcal{M}_{k_1, k_2}^{l}(x-(t-s)v,v,\tau,s,t)\,dvd\tau \Big\Vert_{L^q_x}\lesssim\varep_1^2\langle s\rangle^{-2}2^{-j}
\end{equation}
for any $s\in[0,t]$, $q\in\{1,\infty\}$, and $l\in\{2,8,10,12\}$. The desired bounds \eqref{april10eqn66} follow by integration in $s$ and summation over $k_1,k_2\in[-10m,10m]$.

Similarly,
\begin{equation*}
\begin{split}
&\int_{\mathbb{R}^3}P_k\mathcal{M}_{k_1, k_2}^4(x-(t-s)v,v,\tau,s,t)\,dv\sim \mathcal{Q}_{j,k}(\nabla E_{k_1},E_{k_2};C^4)(x,\tau,s,t),\\
&\qquad C^4(x,y,v,\tau,s,t):=2^{3j}\widetilde{Y}(y-tv,v,s,t)\nabla^2_v M'_j(v),
\end{split}
\end{equation*}
\begin{equation*}
\begin{split}
&\int_{\mathbb{R}^3}P_k\mathcal{M}_{k_1, k_2}^5(x-(t-s)v,v,\tau,s,t)\,dv\sim \mathcal{Q}_{j,k}(E_{k_1},\nabla E_{k_2};C^5)(x,\tau,s,t),\\
&\qquad C^5(x,y,v,\tau,s,t):=2^{3j}\widetilde{Y}(y-tv,v,\tau,t)\nabla^2_v M'_j(v),
\end{split}
\end{equation*}
\begin{equation*}
\begin{split}
&\int_{\mathbb{R}^3}P_k\mathcal{M}_{k_1, k_2}^{9}(x-(t-s)v,v,\tau,s,t)\,dv\sim \mathcal{Q}_{j,k}(E_{k_1},\nabla E_{k_2};C^{9})(x,\tau,s,t),\\
&\qquad C^{9}(x,y,v,\tau,s,t):=2^{3j}(\tau-s) \big[\nabla_v M'_j(v+\widetilde{W}(y-tv,v,s,t))-\nabla_v M'_j(v)\big],
\end{split}
\end{equation*}
\begin{equation*}
\begin{split}
&\int_{\mathbb{R}^3}P_k\mathcal{M}_{k_1, k_2}^{13}(x-(t-s)v,v,\tau,s,t)\,dv\sim \mathcal{Q}_{j,k}(E_{k_1},\nabla E_{k_2};C^{13})(x,\tau,s,t),\\
&\qquad C^{13}(x,y,v,\tau,s,t):=2^{3j}(s\nabla_y-\nabla_v)  \widetilde{Y}(y-tv,v,\tau,t )\nabla M'_j(v+\widetilde{W}(y-tv,v,s,t)).
\end{split}
\end{equation*}
Using \eqref{cui6}--\eqref{cui7.5} we have $\Lambda(C^l)(\tau,s,t)\lesssim 2^{-j}\langle\tau\rangle\langle s\rangle^{-2+3\delta}$ for $l\in\{4,5,9,13\}$. Thus
\begin{equation*}
2^{3m(1-1/q)}\Big\Vert\int_s^t\int_{\mathbb{R}^3}P_k\mathcal{M}_{k_1, k_2}^{l}(x-(t-s)v,v,\tau,s,t)\,dvd\tau \Big\Vert_{L^q_x}\lesssim\varep_1^2\langle s\rangle^{-2}2^{-j},
\end{equation*}
using \eqref{june16eqn31A1}--\eqref{june16eqn31A2}, for any $s\in[0,t]$, $q\in\{1,\infty\}$, and $l\in\{4,5,9,13\}$. The desired bounds \eqref{april10eqn66} follow by integration in $s$ and summation over $k_1,k_2\in[-10m,10m]$, as before.

Similarly,
\begin{equation*}
\begin{split}
&\int_{\mathbb{R}^3}P_k\mathcal{M}_{k_1, k_2}^3(x-(t-s)v,v,\tau,s,t)\,dv\sim \mathcal{Q}_{j,k}(E_{k_1},E_{k_2};C^3)(x,\tau,s,t),\\
&\qquad C^3(x,y,v,\tau,s,t):=2^{3j}\big[\nabla_v^2 M'_j(v+\theta\widetilde{W}(y-tv,v,s,t))-\nabla_v^2 M'_j(v)\big],
\end{split}
\end{equation*}
\begin{equation*}
\begin{split}
&\int_{\mathbb{R}^3}P_k\mathcal{M}_{k_1, k_2}^1(x-(t-s)v,v,\tau,s,t)\,dv\sim \mathcal{Q}_{j,k}(\nabla^2 E_{k_1},E_{k_2};C^1)(x,\tau,s,t),\\
&\qquad C^1(x,y,v,\tau,s,t):=2^{3j}(\tau-s)\widetilde{Y}(y-tv,v,s,t )\nabla M'_j(v),
\end{split}
\end{equation*}
\begin{equation*}
\begin{split}
&\int_{\mathbb{R}^3}P_k\mathcal{M}_{k_1, k_2}^7(x-(t-s)v,v,\tau,s,t)\,dv\sim \mathcal{Q}_{j,k}(\nabla^2E_{k_1},\nabla E_{k_2};C^7)(x,\tau,s,t),\\
&\qquad C^7(x,y,v,\tau,s,t):=2^{3j}(\tau-s)^2\widetilde{Y}(y-tv,v,s,t )M'_j(v+\widetilde{W}(y-tv,v,s,t)).
\end{split}
\end{equation*}
Clearly, $\Lambda(C^3)(\tau,s,t)\lesssim 2^{-j}\langle s\rangle^{-2+3\delta}$, $\Lambda(C^1)(\tau,s,t)\lesssim 2^{-j}\langle\tau\rangle\langle s\rangle^{-1+3\delta}$, and $\Lambda(C^7)(\tau,s,t)\lesssim 2^{-j}\langle\tau\rangle^2\langle s\rangle^{-1+3\delta}$, using \eqref{cui6}. Bounds similar to \eqref{change61} follow for $l\in\{3,1,7\}$. 

In the remaining cases $l=6$ and $l=11$ we need to reverse the roles of the variables $\tau$ and $s$. Indeed,
\begin{equation*}
\begin{split}
&\int_{\mathbb{R}^3}P_k\mathcal{M}_{k_1, k_2}^6(x-(t-s)v,v,\tau,s,t)\,dv\sim \mathcal{Q}_{j,k}(\nabla^2E_{k_2},\nabla E_{k_1};C^6)(x,s,\tau,t),\\
&\qquad C^6(x,y,v,s,\tau,t):=2^{3j}(\tau-s)^2\widetilde{Y}(y-tv,v,\tau,t)M'_j(v+\widetilde{W}(y-tv,v,s,t)),
\end{split}
\end{equation*}
\begin{equation*}
\begin{split}
&\int_{\mathbb{R}^3}P_k\mathcal{M}_{k_1, k_2}^{11}(x-(t-s)v,v,\tau,s,t)\,dv\sim \mathcal{Q}_{j,k}(\nabla^2 E_{k_2},E_{k_2};C^{11})(x,s,\tau,t),\\
&\qquad C^{11}(x,y,v,s,\tau,t):=2^{3j}(\tau-s)\widetilde{Y}(y-tv,v,\tau,t)\nabla_vM'_j(v).
\end{split}
\end{equation*}
Using \eqref{cui6} we have $\Lambda(C^6)(s,\tau,t)\lesssim 2^{-j}\langle\tau\rangle^{1+2\delta}$ and $\Lambda(C^{11})(s,\tau,t)\lesssim 2^{-j}\langle\tau\rangle^{2\delta}$. Thus
\begin{equation*}
2^{3m(1-1/q)}\Big\Vert\int_0^\tau\int_{\mathbb{R}^3}P_k\mathcal{M}_{k_1, k_2}^{l}(x-(t-s)v,v,\tau,s,t)\,dvds \Big\Vert_{L^q_x}\lesssim\varep_1^2\langle \tau\rangle^{-2}2^{-j},
\end{equation*}
using \eqref{april10eqn50}--\eqref{april10eqn51}, for any $\tau\in[0,t]$, $q\in\{1,\infty\}$, and $l\in\{6,11\}$. The desired bounds \eqref{april10eqn66} follow by integration in $\tau$ and summation over $k_1,k_2\in[-10m,10m]$.
\end{proof}

 \section{Bounds on the type-II reaction term  }\label{StatCont3}

In this section we prove the following:

\begin{proposition}\label{closeboot2}
With the notation in \eqref{bvn2}, under the assumptions \eqref{YW12} of Proposition \ref{MainBootstrapProp} we have
\begin{equation}\label{copi1}
\Vert \rho^{osc}_{2,II}\Vert_{Osc_\delta}\lesssim\varep_1^2.
\end{equation}
\end{proposition}

\begin{proof} In view of the definitions, for \eqref{copi1} it suffices to prove that, for any $k,m\in\Z$, $m\geq \delta^{-4}$,
\begin{equation}\label{copi8}
\Big\|\sum_{j\in [0, 19m/20]\cap [0, -k-\delta m/3]}\widetilde{\varphi}_m(t)T^{II}_{2,j,k}\Big\|_{Osc_\delta}\lesssim\varep_1^2 .
\end{equation}

We recall \eqref{sug23.4}, and start from the definition \eqref{qwp1} to calculate $(\partial_tL_{2,j})(x,v,s,t)$. For this we notice that by definition \eqref{Lan6} we have
\begin{equation}\label{copi2}
\begin{split}
&X(x+(t-s)v,v,s,t)=\widetilde{Y}(x-sv,v,s,t)+x,\\
&V(x+(t-s)v,v,s,t)=\widetilde{W}(x-sv,v,s,t)+v.
\end{split}
\end{equation}
Therefore
\begin{equation*}
\begin{split}
&(\partial_tL_{2,j})(x,v,s,t)=\widetilde{\varphi}_j(v)\\
&\times[-(\partial_iE^l)(x+\widetilde{Y}(x-sv,v,s,t),s) (\partial_t\widetilde{Y}^i)(x-sv,v,s,t) \p_l M_0(v+\widetilde{W}(x-sv,v,s,t))\\
&\quad\,\,-E^l(x+\widetilde{Y}(x-sv,v,s,t),s)  (\partial_t\widetilde{W}^i)(x-sv,v,s,t) \partial_i \p_{l}M_0(v+\widetilde{W}(x-sv,v,s,t))\big].
\end{split}
\end{equation*}
Using the formulas \eqref{Lan7} we have
\begin{equation}\label{copi3}
\begin{split}
&(\partial_t\widetilde{Y})(x-sv,v,s,t)=(t-s) E( x+(t-s)v,t)+Err^1(x+(t-s)v,v,s,t),\\
&Err^1_i(x,v,s,t)\!:=\!\int_{s}^t (\tau-s) (\nabla E_{i})(x-tv+\tau v +\widetilde{Y}(x-tv,v,\tau,t),\tau )\cdot \p_t \widetilde{Y}(x-tv,v,\tau,t)d \tau,
\end{split}
\end{equation}
and
\begin{equation}\label{copi4}
\begin{split}
&(\partial_t\widetilde{W})(x-sv,v,s,t)=-E( x+(t-s)v,t)-Err^2(x+(t-s)v,v,s,t),\\
&Err^2_{i}(x,v,s,t)\!:=\!\int_{s}^t (\nabla E_{i})(x-tv+\tau v +\widetilde{Y}(x-tv,v,\tau,t),\tau )\cdot \p_t \widetilde{Y}(x-tv,v,\tau,t)d \tau.
\end{split}
\end{equation}
Therefore we can decompose
\begin{equation}\label{copi5}
(\partial_tL_{2,j})(x,v,s,t)=\big[N_{1 }(x+(t-s)v,v,s,t)+N_{2 }(x+(t-s)v,v,s,t)\big]\widetilde{\varphi}_j(v),
\end{equation}
where
\begin{equation}\label{copi6}
\begin{split}
N_{1}&(y,v,s,t):=-(\partial_iE^l)(y-(t-s)v+\widetilde{Y}(y-tv,v,s,t),s)\\
&\qquad\qquad\qquad\cdot (t-s)E^i(y,t) \p_{l}M_0(v+\widetilde{W}(y-tv,v,s,t))\\
&+ E^l(y-(t-s)v+\widetilde{Y}(y-tv,v,s,t),s)\cdot E^i(y,t)  \partial_i \p_{l}M_0 (v+\widetilde{W}(y-tv,v,s,t)),
\end{split}
\end{equation}
and
\begin{align}
N_{2 }(y,v,s,t)&=N^1_{2 }(y,v,s,t)+N^2_{2 }(y,v,s,t),\\
N^1_{2 }(y,v,s,t)&:=-(\partial_iE^l)(y-(t-s)v+\widetilde{Y}(y-tv,v,s,t),s)\cdot Err^1_i(y,v,s,t)\label{copi7}\\
&\qquad\times \p_{l}M_0(v+\widetilde{W}(y-tv,v,s,t)),\\
N^2_{2 }(y,v,s,t)&:=E^l(y-(t-s)v+\widetilde{Y}(y-tv,v,s,t),s)\cdot Err^2_i(y,v,s,t)\\
&\qquad\times   \partial_i \p_{l}M_0 (v+\widetilde{W}(y-tv,v,s,t)).
\end{align}

We can use the formulas \eqref{copi5}--\eqref{copi7} and \eqref{sug23.4} to decompose
\begin{equation}\label{copi9}
\begin{aligned}
T^{II}_{2,j,k}&= O^1_{j,k}+O^2_{j,k},\\
\widehat{O^a_{j,k}}(\xi,t)&:=\int_0^t\int_0^t \int_{\R^3}\mathbf{1}_+(s-\tau)e^{-(t-s)|\xi|}\frac{e^{is}}{1-i\mathcal{D}}\varphi_k(\xi) \widetilde{\varphi}_j(v) \widehat{N_{a }}(\xi,v,\tau,s)\,dvd\tau ds
\end{aligned}
\end{equation}
for $a\in\{1,2\}$, where as before $\mathcal{D}=|\xi|-iv\cdot\xi$. For future reference we denote the spatial kernel of this expression by
\begin{equation}\label{copi12}
\mathcal{L}_{j,k}(z,v,s,t):= \widetilde{\varphi}_j(v)  \frac{1}{(2\pi )^3}\int_{\R^3}e^{iz\cdot\xi}e^{-(t-s)|\xi|}\frac{e^{is}}{1-i\mathcal{D}}\varphi_k(\xi)  \,d\xi.
\end{equation}
Notice that for $(j,k,m)\in A^{II}$ we have $k+j\leq -10$, so we can integrate by parts several times in $\xi$ in the definition above to see that
\begin{equation}\label{copi12.5}
|\mathcal{L}_{j,k}(z,v,s,t)|\lesssim \widetilde{\varphi}_{[j-2,j+2]}(v)(1+|t-s|2^k)^{-8}2^{3k}(1+2^k|z|)^{-10}.
\end{equation}

Lemmas \ref{copi10lem}--\ref{sil0lem} below, together with the further decompositions \eqref{miu2} and \eqref{sil1}, yield the desired bounds \eqref{copi8}.
\end{proof}

We estimate first the remainder terms $O^2_{j,k}$.

\begin{lemma}\label{copi10lem}
Under the assumptions of Proposition \ref{closeboot2}, for any $(j,k,m)\in A^{II}$ we have
\begin{equation}\label{copi11}
\Vert \widetilde{\varphi}_m(t){O}^2_{j,k}\Vert_{ Osc_\delta }\lesssim\varep_1^22^{-\delta j}.
\end{equation}
\end{lemma}

\begin{proof}  We decompose
\begin{equation}\label{copi13}
\begin{split}
&{O}^2_{j,k}={O}^{2,1}_{j,k}+{O}^{2,2}_{j,k},\\
&{O}^{2,l}_{j,k}(x,t):=\int_0^t\int_0^t \int_{\R^3}\int_{\R^3}\mathbf{1}_+(s-\tau)\mathcal{L}_{j,k}(x-y,v,s,t)N^l_{2 }(y,v,\tau,s)\,dydvd\tau ds.
\end{split}
\end{equation}

{\bf{Step 1.}} We consider first the case $l=1$, and examine the formula for $N^1_{2 }$ in \eqref{copi7} and the formula for the functions $Err^1$ in \eqref{copi3}. We have
\begin{equation}\label{newlabel1}
\begin{split}
{O}^{2,1}_{j,k}(x,t):=\int_0^t&\int_0^s\int_{\tau}^s \int_{\R^3}\int_{\R^3}\mathcal{L}_{j,k}(x-y,v,s,t)(-\partial_iE^l)(y-(s-\tau)v+\widetilde{Y}(y-sv,v,\tau,s),\tau)\\
&\times (\gamma-\tau)(\partial_aE^i)(y-(s-\gamma)v +\widetilde{Y}(y-sv,v,\gamma,s),\gamma)\\
&\times(\p_s \widetilde{Y}^a)(y-sv,v,\gamma,s)\p_l M_0(v+\widetilde{W}(y-sv,v,\tau,s))\,dydv d\gamma d\tau ds.
\end{split}
\end{equation}
Notice that, with the notation in \eqref{ropi1}, for a suitable kernel $\mathcal{K}_{j,k}$ satisfying \eqref{june16eqn4} we have
\begin{equation}\label{copi14.5}
\begin{split}
&{O}^{2,1}_{j,k}(x,t)=\int_0^t\int_0^s\int_{\tau}^s \mathcal{Q}_{j,k}(\partial_iE^{l},\partial_aE^{i};C_{al}^1)(x,\gamma,\tau,s;t)\,d\gamma d\tau ds,\\
&C^1_{al}(x,y,v,\gamma,\tau,s;t):=(\gamma-\tau)(1+|t-s|2^k)^{-8}(\p_s \widetilde{Y}^a)(y-sv,v,\gamma,s)\\
&\qquad\qquad\qquad\quad\times 2^{3j} \p_l M_0 (v+\widetilde{W}(y-sv,v,\tau,s))\widetilde{\varphi}_{[j-4,j+4]}(v).
\end{split}
\end{equation}
Clearly, using \eqref{nov28eqn2}--\eqref{dec5eqn51} and Lemma \ref{Laga10} (ii) we have
\begin{equation}\label{copi15}
|(\p_s \widetilde{Y}^a)(y-sv,v,\gamma,s)|+\langle\gamma\rangle |(\p_\gamma\p_s \widetilde{Y}^a)(y-sv,v,\gamma,s)|\lesssim \varep_1\langle s\rangle^{-1+\delta}.
\end{equation}
Since $\|\p_l M_0(v)\widetilde{\varphi}_{[j-4,j+4]}(v)  \|_{L^\infty}\lesssim 2^{-5j}$, with the notation in \eqref{june16eqn3} we have
\begin{equation}\label{copi15.5}
\Lambda(C^1_{al})(\gamma,\tau,s)\lesssim \langle\gamma\rangle (1+|t-s|2^k)^{-8}\langle s\rangle^{-1+\delta}2^{-2j}.
\end{equation}

We decompose ${O}^{2,1}_{j,k}={O}^{2,1,1}_{j,k}+{O}^{2,1,2}_{j,k}$ where
\begin{equation*}
\begin{split}
&{O}^{2,1,1}_{j,k}(x,t):=\int_0^{t/2}\int_0^s\int_{\tau}^s \mathcal{Q}_{j,k}(\partial_iE^l,\partial_aE^i;C_{al}^{1})(x,\gamma,\tau,s;t)\,d\gamma d\tau ds,\\
&{O}^{2,1,2}_{j,k}(x,t):=\int_{t/2}^t\int_0^s\int_{\tau}^s \mathcal{Q}_{j,k}(\partial_iE^l,\partial_aE^i;C_{al}^{1})(x,\gamma,\tau,s;t)\,d\gamma d\tau ds.
\end{split}
\end{equation*}
Using now \eqref{june16eqn33} and recalling that $k\leq 0$ it follows that
\begin{equation}\label{copi17}
\begin{split}
\|{O}^{2,1,1}_{j,k}(.,t)\|_{L^1_x}&\lesssim \int_0^{t/2}\int_0^s (1+|t-s|2^k)^{-8}\langle s\rangle^{-1+\delta}\varep_1^2\langle\tau\rangle^{-1.1}2^{-2j}\,d\tau ds\\
&\lesssim \varep_1^22^{-2j}\langle t\rangle^{\delta}(1+2^k\langle t\rangle)^{-8}.
\end{split}
\end{equation}
Therefore, using Sobolev embedding, 
\begin{equation}\label{copi18}
\langle t\rangle^3\|{O}^{2,1,1}_{j,k}(.,t)\|_{L^\infty_x}\lesssim \varep_1^22^{-2j}2^{3k}\langle t\rangle^3\langle t\rangle^{\delta}(1+2^k\langle t\rangle)^{-8}.
\end{equation}
Moreover, using again \eqref{june16eqn33}, we can also estimate 
\begin{equation}\label{copi19}
\begin{split}
\|{O}^{2,1,2}_{j,k}(.,t)\|_{L^1_x}&+\langle t\rangle^3\|{O}^{2,1,2}_{j,k}(.,t)\|_{L^\infty_x}\\
&\lesssim \int_{t/2}^t\int_0^s (1+|t-s|2^k)^{-8}\langle s\rangle^{-1+\delta}\varep_1^2\langle\tau\rangle^{-1.1}2^{-2j}\,d\tau ds\\
&\lesssim \varep_1^22^{-2j}\langle t\rangle^{-1+\delta}\min\{2^{-k},\langle t\rangle\}.
\end{split}
\end{equation}
These three bounds show that
\begin{equation}\label{copi20}
\|\widetilde{\varphi}_m(t)\langle t\rangle^{-\delta}{O}^{2,1}_{j,k}\|_{ B^0_t }+ \|\widetilde{\varphi}_m(t)\langle t\rangle^{1-2\delta}\nabla_x{O}^{2,1}_{j,k}\|_{ B^0_t }\lesssim \varep_1^2 2^{-2j}.
\end{equation}

We show now that
\begin{equation}\label{copi21}
\|\widetilde{\varphi}_m(t)\langle t\rangle^{1-2\delta}(\partial_t+|\nabla_x|){O}^{2,1}_{j,k}\|_{ B^0_t }\lesssim \varep_1^2 2^{-2j}.
\end{equation}
Indeed, the identities \eqref{newlabel1} show that
\begin{equation*}
(\partial_t+|\nabla_{x}|){O}^{2,1}_{j,k}(x,t)=\int_0^t\int_{\tau}^t \mathcal{Q}_{j,k} (\partial_iE^l,\partial_aE^i;C_{al}^1)(x,\gamma,\tau,t;t)\,d\gamma d\tau.
\end{equation*}
It follows from \eqref{june16eqn33} and \eqref{copi15.5} that
\begin{equation*}
\|(\partial_t+|\nabla_{x}|){O}^{2,1}_{j,k}(.,t)\|_{B_t^0}\lesssim\int_0^t \varep_1^2\langle t\rangle^{-1+\delta}2^{-2j}{\langle}\tau{\rangle}^{-1.1}\,d\tau\lesssim \varep_1^2\langle t\rangle^{-1+\delta}2^{-2j}.
\end{equation*}
The desired bounds \eqref{copi21} follow, which completes the analysis of the term $O^{2,1}_{j,k}$.

{\bf{Step 2.}} We consider now the case $l=2$, and examine the formula for $N^2_{2}$ in \eqref{copi7} and the formula for the functions $Err^2$ in \eqref{copi4}. As before, with the notation in \eqref{ropi1} and using \eqref{copi12.5}, for a suitable kernel $\mathcal{K}_{j,k}$ satisfying \eqref{june16eqn4} we have
\begin{equation}\label{copi42}
\begin{split}
&{O}^{2,2}_{j,k}(x,t):=\int_{0}^t\int_0^s\int_{\tau}^s \mathcal{Q}_{j,k}(E^l,\partial_aE^i;C_{ali}^2)(x,\gamma,\tau,s;t)\,d\gamma d\tau ds,\\
&C^2_{ali}(x,y,v,\gamma,\tau,s;t):=(1+|t-s|2^k)^{-8}(\p_s \widetilde{Y}^a)(y-sv,v,\gamma,s)\\
&\qquad\qquad\qquad\quad\times 2^{3j} (\partial_i\p_l M_0)(v+\widetilde{W}(y-sv,v,\tau,s)) \widetilde{\varphi}_{[j-4,j+4]}(v) .
\end{split}
\end{equation}
Similarly to \eqref{copi15.5} we have
\begin{equation*}
\Lambda(C^2_{ali})(\gamma,\tau,s)\lesssim  (1+|t-s|2^k)^{-8}\langle s\rangle^{-1+\delta}2^{-2j}.
\end{equation*}
Then we estimate, using \eqref{june16eqn31A2} and Sobolev embedding
\begin{equation*}
\|{O}^{2,2}_{j,k}(.,t)\|_{L^1_x}+\langle t\rangle^3\|{O}^{2,2}_{j,k}(.,t)\|_{L^\infty_x}\lesssim \varep_1^22^{-2j}\langle t\rangle^{-1+\delta}\min\{2^{-k},\langle t\rangle\}
\end{equation*}
as in \eqref{copi17}, \eqref{copi18}, and \eqref{copi19}. Then we notice that
\begin{equation*}
(\partial_t+|\nabla_{x}|){O}^{2,2}_{j,k}(x,t)=\int_0^t\int_{\tau}^t \mathcal{Q}_{j,k}(E^l,\partial_aE^i,C_{ali}^2)(x,\gamma,\tau,t;t)\,d\gamma d\tau,
\end{equation*}
so we can estimate using \eqref{june16eqn31A2}
\begin{equation*}
\|(\partial_t+|\nabla_{x}|){O}^{2,2}_{j,k}(.,t)\|_{B_t^0}\lesssim\int_0^t \varep_1^2\langle t\rangle^{-1+\delta}2^{-2j}{\langle}\tau{\rangle}^{-1.1}\,d\tau\lesssim \varep_1^2\langle t\rangle^{-1+\delta}2^{-2j}.
\end{equation*}
These bounds show that
\begin{equation*}
\|\widetilde{\varphi}_m(t)\langle t\rangle^{-\delta}{O}^{2,2}_{j,k}\|_{B^0_t}+ \|\widetilde{\varphi}_m(t)\langle t\rangle^{1-2\delta}\nabla_x{O}^{2,2}_{j,k}\|_{B^0_t}+\|\widetilde{\varphi}_m(t)\langle t\rangle^{1-2\delta}\partial_t {O}^{2,2}_{j,k}\|_{B^0_t}\lesssim \varep_1^2 2^{-2j}.
\end{equation*}
The lemma follows using also \eqref{copi20}--\eqref{copi21}.
\end{proof}

We estimate now the main terms ${O}^1_{j,k}$ defined in \eqref{copi9}. Using the identities \eqref{copi6}, we rewrite
\begin{equation}\label{miu2}
\begin{split}
N_{1}&=N_{1 }^{ 1}+N_{1 }^{  2}+N_{1 }^{  3},\\
N_{1 }^{  1}(y,v,\tau,s)&:=\frac{d}{dv^i}\Big\{E^l(y-(s-\tau)v+\widetilde{Y}(y-sv,v,\tau,s),\tau)E^i(y,s)\\
&\qquad\qquad\qquad\times \p_l M_0(v+\widetilde{W}(y-sv,v,\tau,s))\Big\},\\
N_{1 }^{ 2}(y,v,\tau,s)&:=-(\partial_aE^l)(y-(s-\tau)v+\widetilde{Y}(y-sv,v,\tau,s),\tau) \widetilde{Y}'_{i,a} (y,v,\tau,s)\\
&\qquad\qquad\qquad\times E^i(y,s) \p_l M_0 (v+\widetilde{W}(y-sv,v,\tau,s)),\\
N_{1 }^{  3}(y,v,\tau,s)&:=-E^l(y-(s-\tau)v+\widetilde{Y}(y-sv,v,\tau,s),\tau)E^i(y,s)\\
&\qquad\qquad\qquad\times   \widetilde{W}'_{i,a} (y,v,\tau,s)(\partial_a \p_l M_0 )(v+\widetilde{W}(y-sv,v,\tau,s)),
\end{split}
\end{equation}
where
\begin{equation}\label{miu3}
\begin{split}
\widetilde{Y}'_{i,a}(y,v,\tau,s)&:=[(\partial_{v^i}-s\partial_{x^i})\widetilde{Y}^a](y-sv,v,\tau,s),\\
\widetilde{W}'_{i,a}(y,v,\tau,s)&:=[(\partial_{v^i}-s\partial_{x^i})\widetilde{W}^a](y-sv,v,\tau,s).
\end{split}
\end{equation}

From the above decomposition, with the notation \eqref{copi12}, we have
\begin{equation}\label{sil1}
\begin{split}
{O}^1_{j,k}(x,t)&=\sum_{l=1,2,3} O^{1 ,l}_{j,k}(x,t), \\
{O}^{1 ,l}_{j,k}(x,t)&:=\int_0^t\int_0^t\int_{\R^3}\int_{\R^3}\mathbf{1}_+(s-\tau)\mathcal{L}_{j,k}(x-y,v,s,t)N^{ l}_{1 }(y,v,\tau,s)\,dydvd\tau ds.
\end{split}
\end{equation}
We estimate first the main term:

\begin{lemma}\label{copi40lem}
For any fixed $k,m\in \Z$, $m\geq \delta^{-4}$, we have
\begin{equation}\label{copi41}
\Big\Vert \sum_{j\in [0, 19m/20]\cap [0, -k-\delta m/3]} \widetilde{\varphi}_m(t){O}^{1, 1}_{j,k}\Big\Vert_{Osc_\delta}\lesssim\varep_1^2 .
\end{equation}
\end{lemma}

\begin{proof}

Recall \eqref{sil1} and  \eqref{miu2}. After integrating  by parts in $v$, we have 
\begin{equation}\label{sil2}
\begin{split}
{O}^{1, 1}_{j,k}&(x,t)=-\int_0^t\int_0^s\int_{\R^3}\int_{\R^3} (\partial_{v^i}\mathcal{L}_{j,k})(x-y,v,s,t)E^i(y,s)\\
&\times E^l(y-(s-\tau)v+\widetilde{Y}(y-sv,v,\tau,s),\tau) \p_l M_0 (v+\widetilde{W}(y-sv,v,\tau,s))\,dydvd\tau  ds.
\end{split}
\end{equation}

Let $J_0:=\min\{19m/20,-k-\delta m/3\}$ and consider the sum over $j$ of these terms. We write
\begin{equation}\label{2022april25eqn33}
\begin{split}
&\sum_{j\in [0, J_0]}\partial_{v^i}\mathcal{L}_{j,k}(z,v,s,t) =\sum_{j\in[0,J_0+2]}\widetilde{\mathcal{L}}^i_{j,k}(z,v,s,t),\\
&\widetilde{\mathcal{L}}^i_{j,k}(z,v,s,t):=\widetilde{\varphi}_j(v)\frac{d}{dv^i}\Big\{\frac{\widetilde{\varphi}_{\leq J_0}(v)}{(2\pi )^3}\int_{\R^3}e^{iz\cdot\xi}e^{-(t-s)|\xi|}\frac{e^{is}}{1-i\mathcal{D}}\varphi_k(\xi)  \,d\xi\Big\}.
\end{split}
\end{equation}
We integrate by parts a few times in $\xi$ using Lemma \ref{tech5} (notice that $J_0+k\leq -10$, so the denominator $1-i\mathcal{D}$ in the integral is not singular) to see that
\begin{equation}\label{2022april25eqn34}
 \big|\widetilde{\mathcal{L}}^i_{j,k}(z,v,s,t)\big|\lesssim  \widetilde{\varphi}_{[j-2, j+2]}(v)(1+|t-s|2^k)^{-8}2^{3k }(1+2^k|z|)^{-10}\cdot (2^k+2^{-J_0}\mathfrak{1}_{[J_0-2,J_0+2]}(j)).
\end{equation}

We decompose $\sum_{j\in [0, J_0]} \widetilde{\varphi}_m(t)O^{1, 1}_{j,k}$ via \eqref{2022april25eqn33}. Let
\begin{equation}\label{change39.1}
\begin{split}
X_{j,k}&(x,s;t):=\int_0^s\int_{\R^3}\int_{\R^3}\widetilde{\mathcal{L}}^i_{j,k}(x-y,v,s,t)E^i(y,s)\\
&\times E^l(y-(s-\tau)v+\widetilde{Y}(y-sv,v,\tau,s),\tau) \p_l M_0 (v+\widetilde{W}(y-sv,v,\tau,s))\,dydvd\tau,
\end{split}
\end{equation}
and notice that, due to \eqref{sil2},
\begin{equation}\label{change39.2}
\begin{split}
&\sum_{j\in [0,J_0]}  {O}^{1, 1}_{j,k} (x,t)=\sum_{j\in [0,J_0+2]}\int_0^t X_{j,k}(x,s;t)\,ds,\\
&\sum_{j\in [0,J_0]}  \nabla_x{O}^{1, 1}_{j,k} (x,t)=\sum_{j\in [0,J_0+2]}\int_0^t \nabla_xX_{j,k}(x,s;t)\,ds,\\
&\sum_{j\in [0,J_0]}  (\partial_t+|\nabla_x|){O}^{1, 1}_{j,k} (x,t)=\sum_{j\in [0,J_0+2]}X_{j,k}(x,t;t).
\end{split}
\end{equation}

We prove now that for $r\in\{1,\infty\}$
\begin{equation}\label{change39.3}
\begin{split}
\langle s\rangle^{3-3/r}\|X_{j,k}(.,s;t)\|_{L^r_x}&\lesssim (1+|t-s|2^k)^{-8}(2^k+2^{-J_0}\mathfrak{1}_{[J_0-2,J_0+2]}(j))2^{-j},\\
\langle s\rangle^{3-3/r}\|\nabla_xX_{j,k}(.,s;t)\|_{L^r_x}&\lesssim (1+|t-s|2^k)^{-8}(2^k+2^{-J_0}\mathfrak{1}_{[J_0-2,J_0+2]}(j))2^{-j}\langle s\rangle^{-1+\delta},
\end{split}
\end{equation}
for any $k,j,s,t$. Indeed, we use \eqref{2022april25eqn34} to express the functions $X_{j,k}$ this in terms of trilinear operators as defined in \eqref{ropi1}, so that 
\begin{equation}\label{change41}
\begin{split}
&X_{j,k}(x,s;t)=\int_{0}^s  \mathcal{Q}_{j,k}(E, E;C^1)(x, \tau,s,s;t) \,d \tau d s,\\
&\quad C^1(x,y,v,\tau,s,s;t):=(1+|t-s|2^k)^{-8}(2^k+2^{-J_0}\mathfrak{1}_{[J_0-2,J_0+2]}(j))\\
&\qquad\qquad\qquad\times2^{3j}\p_l M_0 (v+\widetilde{W}(y-sv,v,\tau,s))\widetilde{\varphi}_{[j-2, j+2]}(v).
\end{split}
\end{equation}
Moreover, if we distribute one $\nabla_x$ derivative in \eqref{change39.1} we can write
\begin{equation}\label{change42}
\begin{split}
&\nabla_x X_{j,k}=X^1_{j,k}+X^2_{j,k},\\
&X^1_{j,k}(x,s;t):=\int_{0}^s  \big[\mathcal{Q}_{j,k}(E, \nabla E;C^1)(x, \tau,s,s;t)+\mathcal{Q}_{j,k}(\nabla E, E;C^2)(x, \tau,s,s;t)\big] \,d \tau,\\
&X^2_{j,k}(x,s;t):=\int_{0}^s  \mathcal{Q}_{j,k}(E, E;C^3)(x, \tau,s,s;t) \,d \tau,
\end{split}
\end{equation}
where
\begin{equation}\label{change43}
\begin{split}
&C^2(x,y,v,\tau,s,s;t):=(1+|t-s|2^k)^{-8}(2^k+2^{-J_0}\mathfrak{1}_{[J_0-2,J_0+2]}(j))\\
&\qquad\qquad\times [\delta_a^b+\partial_{y^a}\widetilde{Y}^b(y-sv,v,\tau,s)]2^{3j}\p_l M_0 (v+\widetilde{W}(y-sv,v,\tau,s))\widetilde{\varphi}_{[j-2, j+2]}(v).\\
&C^3(x,y,v,\tau,s,s;t):=(1+|t-s|2^k)^{-8}(2^k+2^{-J_0}\mathfrak{1}_{[J_0-2,J_0+2]}(j))\\
&\qquad\qquad\times2^{3j}\p_{y^a}\widetilde{W}^b(y-sv,v,\tau,s)\p_b\p_l M_0 (v+\widetilde{W}(y-sv,v,\tau,s))\widetilde{\varphi}_{[j-2, j+2]}(v).
\end{split}
\end{equation}
The formulas \eqref{change41} and \eqref{change43} and the bounds \eqref{cui7} show that
\begin{equation}\label{change44}
\begin{split}
\Lambda(C^1)(\tau,s,s)+\Lambda(C^2)(\tau,s,s)&\lesssim (1+|t-s|2^k)^{-8}(2^k+2^{-J_0}\mathfrak{1}_{[J_0-2,J_0+2]}(j))2^{-j},\\
\Lambda(C^3)(\tau,s,s)&\lesssim (1+|t-s|2^k)^{-8}(2^k+2^{-J_0}\mathfrak{1}_{[J_0-2,J_0+2]}(j))2^{-j}\langle\tau\rangle^{-2}.
\end{split}
\end{equation}
The bounds in the first line of \eqref{change39.3} follow from \eqref{june16eqn32}. The bounds in the second line follow from \eqref{june16eqn31A1}--\eqref{june16eqn31A2} for the term $X^{j,k}$ and from \eqref{change15} (by summation over $k_1,k_2\in\Z$) for $X^2_{j,k}$.

We can now prove the bounds \eqref{copi41}. Notice that 
\begin{equation}\label{change44.5}
\sum_{j\in[0,J_0+2]}(2^k+2^{-J_0}\mathfrak{1}_{[J_0-2,J_0+2]}(j))2^{-j}\lesssim 2^k+2^{-2J_0}\lesssim 2^{k+\delta m/3}+2^{-3m/2}.
\end{equation}
We can thus use the formulas \eqref{change39.2} and the bounds \eqref{change39.3} to estimate
\begin{equation*}
\begin{split}
\Big\|\widetilde{\varphi}_m(t)\sum_{j\in [0,J_0]} {O}^{1, 1}_{j,k} (t)\Big\|_{L^1_x}&\lesssim\varep_1^2\int_0^t(1+|t-s|2^k)^{-8}(2^{k+\delta m/3}+2^{-3m/2})\,d s\lesssim \varep_1^22^{\delta m/3},
\end{split}
\end{equation*}
\begin{equation*}
\begin{split}
\Big\|\widetilde{\varphi}_m(t)\sum_{j\in [0,J_0]} \nabla_x{O}^{1, 1}_{j,k} (t)\Big\|_{L^1_x}&\lesssim\varep_1^2\int_0^t(1+|t-s|2^k)^{-8}(2^{k+\delta m/3}+2^{-3m/2})\langle s\rangle^{-1+\delta}\,d s\\
&\lesssim \varep_1^22^{-m+4\delta m/3}.
\end{split}
\end{equation*}
Also, using the Sobolev embedding when $s\leq t/2$, we estimate
\begin{equation*}
\begin{split}
\langle t\rangle^3\Big\|\widetilde{\varphi}_m(t)&\sum_{j\in [0,J_0]} {O}^{1, 1}_{j,k} (t)\Big\|_{L^\infty_x}\lesssim\varep_1^2\int_{t/2}^t(1+|t-s|2^k)^{-8}(2^{k+\delta m/3}+2^{-3m/2})\,ds\\
&+\varep_1^2\int_0^{t/2}\langle t\rangle^32^{3k}(1+|t-s|2^k)^{-8}(2^{k+\delta m/3}+2^{-3m/2})\,ds\lesssim \varep_1^22^{\delta m/3}.
\end{split}
\end{equation*}
\begin{equation*}
\begin{split}
\langle t\rangle^3\Big\|\widetilde{\varphi}_m(t)&\sum_{j\in [0,J_0]} \nabla_x{O}^{1, 1}_{j,k} (t)\Big\|_{L^\infty_x}\lesssim\varep_1^2\int_{t/2}^t(1+|t-s|2^k)^{-8}(2^{k+\delta m/3}+2^{-3m/2})\langle s\rangle^{-1+\delta}\,ds\\
&+\varep_1^2\int_0^{t/2}\langle t\rangle^32^{3k}(1+|t-s|2^k)^{-8}(2^{k+\delta m/3}+2^{-3m/2})\langle s\rangle^{-1+\delta}\,ds\lesssim \varep_1^22^{-m+4\delta m/3}.
\end{split}
\end{equation*}
Finally, using the identity in the last line of \eqref{change39.2} and the bounds \eqref{change39.3},
\begin{equation*}
\langle t\rangle^{3-3/r}\Big\|\widetilde{\varphi}_m(t)\sum_{j\in [0,J_0]} (\p_t + |\nabla_x|){O}^{1, 1}_{j,k}(t)\Big\|_{L^r_x}\lesssim\varep_1^22^{\delta m/3},
\end{equation*}
for $r\in\{1,\infty\}$. The desired estimates \eqref{copi41} follow from these last five bounds.
\end{proof}

Finally we estimate the remainders arising in the decomposition \eqref{miu2}.

\begin{lemma}\label{sil0lem}
With $O^{1, l}_{j,k}$ defined as in \eqref{sil1}, for any $(j,k,m)\in A^{II}$ and $l\in\{2,3\}$ we have 
\begin{equation}\label{sil3}
\Vert \widetilde{\varphi}_m(t)O^{1, l}_{j,k}\Vert_{Osc_\delta}\lesssim\varep_1^22^{ -\delta j }.
\end{equation}
\end{lemma}

\begin{proof} Recalling the definitions \eqref{miu3} and the formulas \eqref{Lan7}, we compute
\begin{equation*}
\begin{split}
&\widetilde{W}'_{i,a}(y,v,\tau,s) =
 -\int_\tau^s \big[(\gamma-s) \p_{y^i} E^{a}(y-(s-\gamma) v+\widetilde{Y}(y-sv ,v,\gamma,s),\gamma)\\
&\quad-(\p_{v^i} - s  \p_{y^i})  \widetilde{Y}^{b}(y-sv ,v,\gamma,s ) \p_{y^b} E^{a}(y-(s-\gamma) v+\widetilde{Y}(y-sv ,v,\gamma,s),\gamma)\big]\,d\gamma
\end{split}
\end{equation*}
and
\begin{equation*}
\begin{split}
&\widetilde{Y}'_{i,a}(y,v,\tau,s) =\int_\tau^s \big[(\gamma-\tau)(\gamma-s) \p_{y^i} E^{a}(y-(s-\gamma) v+\widetilde{Y}(y-sv ,v,\gamma,s),\gamma)\\
&\quad+(\gamma-\tau)(\p_{v^i} - s  \p_{y^i})  \widetilde{Y}^{b}(y-sv ,v,\gamma,s ) \p_{y^b} E^{a}(y-(s-\gamma) v+\widetilde{Y}(y-sv ,v,\gamma,s),\gamma)\big]\,d\gamma.
\end{split}
\end{equation*}

Therefore, using \eqref{copi12.5}, in terms of the trilinear operators defined in \eqref{ropi1} we can write
\begin{equation}\label{change70}
\begin{split}
&O^{1,2}_{j,k}(x,t)=\int_0^t\int_0^s\int_\tau^s\mathcal{Q}_{j,k}(\partial_{b}E^a,\partial_aE^l;C^1_{bl})(x,\gamma,\tau,s;t)\,d\gamma d\tau ds,\\
&C^1_{bl}(x,y,v,\gamma,\tau,s;t):=(1+|t-s|2^k)^{-8}(\gamma-\tau)E^i(y,s)\\
&\qquad\times\big[(\gamma-s)\delta_i^b+(\partial_{v^i}-s\partial_{y^i})\widetilde{Y}^b(y-sv,v,\gamma,s)\big]2^{3j}\partial_lM_0(v+\widetilde{W}(y-sv,v,\tau,s)),
\end{split}
\end{equation}
and
\begin{equation}\label{change71}
\begin{split}
&O^{1,3}_{j,k}(x,t)=\int_0^t\int_0^s\int_\tau^s\mathcal{Q}_{j,k}(\partial_{b}E^a,E^l;C^2_{abl})(x,\gamma,\tau,s;t)\,d\gamma d\tau ds,\\
&C^2_{abl}(x,y,v,\gamma,\tau,s;t):=(1+|t-s|2^k)^{-8}E^i(y,s)\\
&\qquad\times\big[(\gamma-s)\delta_i^b+(\partial_{v^i}-s\partial_{y^i})\widetilde{Y}^b(y-sv,v,\gamma,s)\big]2^{3j}\partial_a\partial_lM_0(v+\widetilde{W}(y-sv,v,\tau,s)).
\end{split}
\end{equation}
The bounds in Lemma \ref{Laga10} (ii) show that
\begin{equation}\label{change72}
\begin{split}
\Lambda(C^1_{bl})(\gamma,\tau,s)&\lesssim (1+|t-s|2^k)^{-8}2^{-j}\langle s\rangle^{-1+\delta}\langle\gamma\rangle,\\
\Lambda(C^2_{abl})(\gamma,\tau,s)&\lesssim (1+|t-s|2^k)^{-8}2^{-j}\langle s\rangle^{-1+\delta}.
\end{split}
\end{equation}
Therefore, using \eqref{june16eqn31A1} and \eqref{june16eqn33}, for $l\in\{2,3\}$,
\begin{equation}\label{change73}
\begin{split}
\big\|\widetilde{\varphi}_m(t){O}^{1,l}_{j,k} (t)\big\|_{L^1_x}&\lesssim\varep_1^2\int_0^t\int_0^s(1+|t-s|2^k)^{-8}2^{-j}\langle s\rangle^{-1+\delta}\langle\tau\rangle^{-1.1}\,d\tau d s\\
&\lesssim\varep_1^22^{-j}2^{\delta m}\min\{1,2^{-k-m}\}.
\end{split}
\end{equation}
Moreover, using Sobolev embedding when $s\leq t/2$ as before,
\begin{equation}\label{change74}
\begin{split}
\langle t\rangle^3&\big\|\widetilde{\varphi}_m(t){O}^{1,l}_{j,k} (t)\big\|_{L^\infty_x}\lesssim\varep_1^2\int_0^{t/2}\int_0^s\langle t\rangle^32^{3k}(1+|t-s|2^k)^{-8}2^{-j}\langle s\rangle^{-1+\delta}\langle\tau\rangle^{-1.1}\,d\tau d s\\
&+\varep_1^2\int_{t/2}^t\int_0^s(1+|t-s|2^k)^{-8}2^{-j}\langle s\rangle^{-1+\delta}\langle\tau\rangle^{-1.1}\,d\tau d s\lesssim\varep_1^22^{-j}2^{\delta m}\min\{1,2^{-k-m}\}.
\end{split}
\end{equation}

Finally, we notice that
\begin{equation*}
\begin{split}
&(\partial_t+|\nabla_x|)O^{1,2}_{j,k}(x,t)=\int_0^t\int_\tau^t\mathcal{Q}_{j,k}(\partial_{b}E^a,\partial_aE^l;C^1_{bl})(x,\gamma,\tau,t;t)\,d\gamma d\tau,\\
&(\partial_t+|\nabla_x|)O^{1,3}_{j,k}(x,t)=\int_0^t\int_\tau^t\mathcal{Q}_{j,k}(\partial_{b}E^a,E^l;C^2_{abl})(x,\gamma,\tau,t;t)\,d\gamma d\tau.
\end{split}
\end{equation*}
and one can use \eqref{june16eqn31A1} and \eqref{june16eqn33} again to see that
\begin{equation}\label{change75}
\langle t\rangle^{3-3/r}\big\|\widetilde{\varphi}_m(t)(\partial_t+|\nabla_x|){O}^{1,l}_{j,k} (t)\big\|_{L^r_x}\lesssim\varep_1^22^{-j}2^{\delta m}
\end{equation}
for $r\in\{1,\infty\}$. The desired bounds \eqref{sil3} follow from \eqref{change73}--\eqref{change74} and \eqref{change75}.
\end{proof}

\section{Proof of   Lemma \ref{keybilinearlemma}}\label{trilin}

In this section we prove Lemma \ref{keybilinearlemma}. The main idea is to decompose the electric field corresponding to the variable $\gamma$ into its static and oscillatory parts, and then integrate by parts in $\gamma$ in certain cases. Since $\tau,s$ are fixed in the proof we may assume that $\Lambda^\ast(\tau,s)=1$, thus
\begin{equation}\label{change1}
|C(x, y, v,\gamma,\tau,s)|+|\langle\gamma\rangle\p_\gamma C(x, y, v,\gamma,\tau,s)|\lesssim 1\quad\text{ for any }x,y,v,\gamma,\tau,s.
\end{equation}

\subsection{Proof of \eqref{june16eqn31A1}}\label{RopiSec} We may assume that $C$ and $\mathcal{K}_{j,k}$ are real-valued, decompose $E(\gamma)=E^{stat}(\gamma)+\Re\{e^{-i\gamma}E^{osc}(\gamma)\}$ as in \eqref{Laga5}, and then decompose dyadically in frequency, so 
\be\label{july16eqn5}
\mathcal{B}_{j,k} (\mathcal{R}^1\nabla E, \mathcal{R}^2E; C)=\sum_{k_1, k_2\in \Z} \sum_{\ast\in \{stat, osc\}} \Re\{\mathcal{B}_{j,k} (P_{k_1}\mathcal{R}^1(\nabla E), c_{\ast}(\gamma) P_{k_2}\mathcal{R}^2E^{\ast}; C)\}, 
\ee
where $c_{stat}=1$ and $c_{osc}=e^{-i\gamma}$. In the following, we set $k_{\min}:=\min\{k_1,k_2\}$.

{\bf{Step 1.}} We bound first the contributions of the static part. Clearly
\begin{equation}\label{ropi2.5}
\int_{\R^3}\big|h(y-(s-\mu)v +P(y,v))\big|\,dy\lesssim \|h\|_{L^1},
\end{equation}
if $h\in L^1(\R^3)$, $v\in\R^3$, and $P$ satisfies $|\nabla_yP(y,v)|\lesssim\varep_1$. Moreover
\begin{equation}\label{ropi2.6}
\int_{\R^3}\big|h(y-(s-\mu)v +P(y,v))\big|\,dv\lesssim \langle s\rangle^{-3}\|h\|_{L^1},
\end{equation}
provided that $h\in L^1(\R^3)$, $y\in\R^3$, $\mu\leq s\in[0,T]$ satisfy $s-\mu\geq \langle s\rangle/8$, and $|\nabla_vP(y,v)|\lesssim\varep_1\langle s\rangle$.

As before, let $\rho_l:=P_l\rho$, $\rho_l^{stat}:=P_l\rho^{stat}$, and $\rho_l^{osc}:=P_l\rho^{osc}$. We use Lemma \ref{Laga10} (i) and the bounds \eqref{change1} to see that
\begin{equation}\label{ropi2.7}
\begin{split}
&\|\mathcal{Q}_{j,k} (\mathcal{R}^1\rho_{k_1}, \mathcal{R}^2\rho_{k_2}^{stat}; C)(., \gamma, \tau, s)\|_{L^1_x}\lesssim   \|\mathcal{K}_{j,k}(.,.,\tau,s)\|_{L^1_{x,v}}\\
&\qquad\qquad\times \min\{\|\rho_{k_1}(\tau)\|_{L^1}\|\rho_{k_2}^{stat}(\gamma)\|_{L^\infty},\|\rho_{k_1}(\tau)\|_{L^\infty}\|\rho_{k_2}^{stat}(\gamma)\|_{L^1}\}\\
&\lesssim \frac{\varep_1^2}{\langle \tau\rangle^{1-2\delta}2^{k_1}+\langle \tau\rangle^{-\delta}}\cdot 2^{-k_2^+}\langle\gamma\rangle^{-1+2\delta}\cdot \min\{2^{3k_{\min}},\langle \tau\rangle^{-3},\langle\gamma\rangle^{-3}\},
\end{split}
\end{equation}
where the operators $\mathcal{Q}_{j,k} $ are defined as in \eqref{ropi1}. Similarly, using \eqref{ropi2.5} and \eqref{ropi2.6} if $s-\gamma\ge\langle s\rangle/8$ or $s-\tau\ge\langle s\rangle/8$, together with the observation that $\|\mathcal{K}(.,.,\tau,s)\|_{L^1_{x,v}}+ \|\mathcal{K}(.,.,\tau,s)\|_{L^1_xL^\infty_v}\lesssim 1$,
\begin{equation*}
\begin{split}
&\|\mathcal{Q}_{j,k} (\mathcal{R}^1\rho_{k_1}, \mathcal{R}^2\rho_{k_2}^{stat}; C)(., \gamma, \tau, s)\|_{L^\infty_x}\lesssim\,   \min\big\{\|\rho_{k_1}(\tau)\|_{L^\infty}\|\rho_{k_2}^{stat}(\gamma)\|_{L^\infty},\\ 
&\quad\qquad\langle s\rangle^{-3}\|\rho_{k_1}(\tau)\|_{L^\infty}\|\rho_{k_2}^{stat}(\gamma)\|_{L^1},
\langle s\rangle^{-3}\|\rho_{k_1}(\tau)\|_{L^1}\|\rho_{k_2}^{stat}(\gamma)\|_{L^\infty}\big\},
\end{split}
\end{equation*}
and using Lemma \ref{Laga10} (i), we see that
\begin{equation}\label{ropi4}
\begin{split}
&\langle s\rangle^3\|\mathcal{Q}_{j,k} (\mathcal{R}^1\rho_{k_1}, \mathcal{R}^2\rho_{k_2}^{stat}; C)(., \gamma, \tau, s)\|_{L^\infty_x}\\
&\quad\lesssim \frac{\varep_1^2}{\langle \tau\rangle^{1-2\delta}2^{k_1}+\langle \tau\rangle^{-\delta}} 2^{-k_2^+}\langle\gamma\rangle^{-1+2\delta}\min\{2^{k_{\min}},\langle \tau\rangle^{-1},\langle \gamma\rangle^{-1}\}^3.
\end{split}
\end{equation}
Therefore, using \eqref{june16eqn1}, \eqref{ropi2.7} and \eqref{ropi4}, we have
\begin{equation}\label{ropi3}
\begin{split}
&\sum_{k_1, k_2\in \Z}\Big\{\|\mathcal{B}_{j,k} (P_{k_1}\mathcal{R}^1(\nabla E),P_{k_2}\mathcal{R}^2E^{stat}; C)(., \tau, s)\|_{L^1_x}\\
&\quad+\langle s\rangle^{3}\|\mathcal{B}_{j,k} (P_{k_1}\mathcal{R}^1(\nabla E),P_{k_2}\mathcal{R}^2E^{stat}; C)(., \tau, s)\|_{L^\infty_x}\Big\}\lesssim \varep_1^2 2^{ 3\delta m_2}\langle \tau \rangle^{2\delta}\min\{2^{-2 m_2}, \langle \tau \rangle^{-2} \}.
\end{split}
\end{equation}

{\bf{Step 2.}} We bound now the contributions of the oscillatory parts, and we will show that
\begin{equation}\label{ropi8}
\sum_{k_1, k_2\in \Z}2^{-k_2}\|\mathcal{B}_{j,k} (\mathcal{R}^1\rho_{k_1}, e^{-i\gamma}\mathcal{R}^2\rho^{osc}_{k_2}; C)(., \tau, s)\|_{B^0_s}\lesssim  \varep_1^2 \min\{2^{- 1.1 m_2}, \langle \tau \rangle^{-1.1} \}.
\end{equation}
Indeed, we estimate first as in \eqref{ropi2.7}--\eqref{ropi4} to see that for any $k_1,k_2\in\Z$ and $r\in\{1,\infty\}$
\begin{equation}\label{ropi8.1}
\begin{split}
&\langle s\rangle^{3-3/r}\|\mathcal{Q}_{j,k} (\mathcal{R}^1\rho_{k_1}, e^{-i\gamma}\mathcal{R}^2\rho_{k_2}^{osc}; C)(., \gamma, \tau, s)\|_{L^r_x}\\
&\quad\lesssim\frac{\varep_1^2}{\langle \tau\rangle^{1-2\delta}2^{k_1}+\langle \tau\rangle^{-\delta}}\frac{1}{\langle \gamma\rangle^{1-2\delta}2^{k_2}+\langle \gamma\rangle^{-\delta}}\min\{2^{k_{\min}}, \langle \tau\rangle^{-1} , \langle \gamma\rangle^{-1}\}^3,
\end{split}
\end{equation}
where we use \eqref{Laga3} instead of \eqref{Laga2}. The bounds \eqref{ropi8} follow if $m_2\in [0, 400]$ or if $\langle\tau\rangle\geq 2^{9m_2/8}$. On the other hand, if $m_2\geq 400$ and $\langle\tau\rangle\leq 2^{9m_2/8}$ then we have
\begin{equation*}
\begin{split}
\sum_{\{k_1, k_2\in \Z:\,|k_2+m_2|\geq  m_2/8\text{ or }k_1\leq -9m_2/8\}}2^{-k_2}\|\mathcal{B}_{j,k} (\mathcal{R}^1\rho_{k_1}, e^{-i\gamma}\mathcal{R}^2\rho_{k_2}^{osc}; C)(., \tau, s)\|_{B^0_s}\\
\lesssim \varep_1^2 \min\{2^{- 1.1 m_2}, \langle \tau \rangle^{-1.1} \}.
\end{split}
\end{equation*}

It remains to bound the contributions coming from indices $k_2\in[-9m_2/8,-7m_2/8]$ and $k_1\geq -9m_2/8$ when $m_2\geq 400$ and $\langle\tau\rangle\leq 2^{9m_2/8}$. For this we integrate by parts in $\gamma$. We start by writing
\begin{equation}\label{ropi11}
\begin{split}
&e^{-i\gamma}\mathcal{R}^2\rho_{k_2}^{osc}(y-(s-\gamma)v +P_2(y,v,\gamma,\tau,s),s), \gamma)\\
&=\frac{1}{(2\pi)^3}\int_{\R^3}e^{-i\gamma}\varphi_{k_2}(\xi)\mathcal{R}^2(\xi)\widehat{\rho^{osc}}(\xi,\gamma)e^{i\xi \cdot (y-(s-\gamma)v + P_2(y,v,\gamma,\tau,s))}\,d\xi.
\end{split}
\end{equation}
Therefore, using the definitions
\begin{equation}\label{ropi12}
\begin{split}
\mathcal{B}_{j,k} (\mathcal{R}^1\rho_{k_1}, &e^{-i\gamma}\mathcal{R}^2\rho_{k_2}^{osc}; C)(x, \tau, s)=\frac{1}{(2\pi)^3}\int_{\R^3}\int_{\R^3} \int_{\R^3}\int_{t_3}^{t_4}\mathcal{K}_{j,k}(x-y, v,\tau,s)\\
&\times C(x, y, v, \gamma, \tau, s)(\mathcal{R}^1\rho_{k_1})(y-(s-\tau)v +P_1(y,v,\tau,s),\tau)\\
&\times \widehat{\rho^{osc}}(\xi,\gamma)e^{i\gamma(-1+\xi\cdot v)}\varphi_{k_2}(\xi)\mathcal{R}^2(\xi)e^{i\xi \cdot (y-s v + P_2(y,v,\gamma,\tau,s))}\, d\gamma dy d v d\xi. 
\end{split}
\end{equation}
We insert cutoff functions of the form $\varphi_{>p}(1-\xi\cdot v)$ and $\varphi_{\leq p}(1-\xi\cdot v)$ and use the integration by parts formula
\begin{equation}\label{ropi14}
\begin{split}
&\int_{t_3}^{t_4}e^{i\gamma(-1+\xi\cdot v)}C(x, y, v, \gamma, \tau, s)\widehat{\rho^{osc}}(\xi,\gamma)\varphi_{>p}(1-\xi\cdot v)e^{i\xi \cdot (y-s v + P_2(y,v,\gamma,\tau,s))}\, d\gamma\\
&=\int_{t_3}^{t_4}(\partial_\gamma C)(x, y, v, \gamma, \tau, s)I^1_{>p}(\xi, y, v, \gamma, s)\, d\gamma+\int_{t_3}^{t_4}C(x, y, v, \gamma, \tau, s)I^2_{>p}(\xi, y, v, \gamma, s)\, d\gamma\\
&\qquad+\sum_{j\in\{3,4\}}(-1)^{j+1}C(x, y, v, t_j, \tau, s)I^1_{>p}(\xi, y, v, t_j, s),
\end{split}
\end{equation}
where
\begin{equation}\label{ropi16}
\begin{split}
I^1_{>p}(\xi, y, v, \gamma, s)&:=\frac{-e^{i\gamma(-1+\xi\cdot v)}}{i(-1+\xi\cdot v)}\widehat{\rho^{osc}}(\xi,\gamma)\varphi_{>p}(1-\xi\cdot v)e^{i\xi \cdot (y-s v + P_2(y,v,\gamma,\tau,s))},\\
I^2_{>p}(\xi, y, v, \gamma, s)&:=\frac{-e^{i\gamma(-1+\xi\cdot v)}}{i(-1+\xi\cdot v)}\varphi_{>p}(1-\xi\cdot v)e^{i\xi \cdot (y-s v + P_2(y,v,\gamma,\tau,s))}\big\{(\partial_\gamma\widehat{\rho^{osc}})(\xi,\gamma)\\
&\qquad+\widehat{\rho^{osc}}(\xi,\gamma)i\xi \cdot \partial_\gamma P_2(y,v,\gamma,\tau,s)\big\}.
\end{split}
\end{equation}

For $p\leq -4$ and $\ell\in\{1,2\}$ let
\begin{equation}\label{ropi17}
H^l_{>p,k_2}(y, v, \gamma, s):=\frac{1}{(2\pi)^3}\int_{\R^3}I^l_{>p}(\xi, y, v, \gamma, s)\varphi_{k_2}(\xi)\mathcal{R}^2(\xi)\,d\xi,
\end{equation}
\begin{equation}\label{ropi18}
\begin{split}
G_{\leq p,k_2}(y, v, \gamma, s):=\frac{e^{-i\gamma}}{(2\pi)^3}\int_{\R^3}&\varphi_{\leq p}(1-\xi\cdot v)\widehat{\rho^{osc}}(\xi,\gamma)\varphi_{k_2}(\xi)\\
&\times \mathcal{R}^2(\xi)e^{i\xi \cdot (y-(s-\gamma) v + P_2(y,v,\gamma,\tau,s))}\,d\xi.
\end{split}
\end{equation}
In view of these identities, for any $p_0\leq -4$ we can thus decompose
\begin{equation}\label{ropi19}
\mathcal{B}_{j,k} (\mathcal{R}^1\rho_{k_1}, e^{-i\gamma}\mathcal{R}^2\rho_{k_2}^{osc}; C)=X_{\leq p_0, k_1,k_2}+(Y^1_{>p_0, k_1,k_2}+Y^2_{>p_0,k_1,k_2})
\end{equation}
where, with $C':=(\partial_\gamma C)\cdot\mathbf{1}_{[t_3,t_4]}(\gamma)+C\cdot (\delta_0(\gamma-t_3)-\delta_0(\gamma-t_4))$, 
\begin{equation}\label{ropi20}
\begin{split}
&X_{\leq p_0, k_1,k_2}(x, \tau, s):=\int_{t_3}^{t_4}\int_{\R^3}\int_{\R^3}\mathcal{K}_{j,k}(x-y, v,\tau,s)C(x, y, v, \gamma, \tau, s)\\
&\qquad\times (\mathcal{R}^1\rho_{k_1})(y-(s-\tau)v +P_1(y,v,\tau,s),\tau)G_{\leq p_0,k_2}(y, v, \gamma, s)\,  dy d v d\gamma,
\end{split}
\end{equation}
\begin{equation}\label{ropi21}
\begin{split}
&Y^1_{>p_0, k_1,k_2}(x, \tau, s):=\int_{\R}\int_{\R^3}\int_{\R^3}\mathcal{K}_{j,k}(x-y, v,\tau,s)C'(x, y, v, \gamma, \tau, s)\\
&\qquad\times (\mathcal{R}^1\rho_{k_1})(y-(s-\tau)v +P_1(y,v,\tau,s),\tau)H^1_{>p_0,k_2}(y, v, \gamma, s)\,  dy d v d\gamma,
\end{split}
\end{equation}
\begin{equation}\label{ropi22}
\begin{split}
&Y^2_{>p_0, k_1,k_2}(x, \tau, s):=\int_{t_3}^{t_4}\int_{\R^3}\int_{\R^3}\mathcal{K}_{j,k}(x-y, v,\tau,s)C(x, y, v, \gamma, \tau, s)\\
&\qquad\times (\mathcal{R}^1\rho_{k_1})(y-(s-\tau)v +P_1(y,v,\tau,s),\tau)H^2_{>p_0,k_2}(y, v, \gamma, s)\,  dy d v d\gamma.
\end{split}
\end{equation}

{\bf{Step 3.}} We will now prove bounds on the functions $H^l_{>p_0,k_2}$ and $G_{\leq p_0,k_2}$. We claim that
\begin{equation}\label{ropi25}
\begin{split}
&\|H^1_{>p_0,k_2}(.,v,\gamma,s)\|_{L^1_y}+\langle\gamma\rangle^3\|H^1_{>p_0,k_2}(.,v,\gamma,s)\|_{L^\infty_y}\lesssim \varep_12^{-p_0}\langle\gamma\rangle^{\delta},\\
&\|H^2_{>p_0,k_2}(.,v,\gamma,s)\|_{L^1_y}+\langle\gamma\rangle^3\|H^2_{>p_0,k_2}(.,v,\gamma,s)\|_{L^\infty_y}\lesssim \varep_12^{-p_0}\langle\gamma\rangle^{-1+2\delta},
\end{split}
\end{equation}
and
\begin{equation}\label{ropi26}
\begin{split}
\|G_{\leq p_0,k_2}(.,v,\gamma,s)\|_{L^\infty_y}&\lesssim \varep_12^{p_0}\langle v\rangle^{-1}\langle\gamma\rangle^{-1+2\delta}2^{k_2},\\
\|G_{\leq p_0,k_2}(.,v,\gamma,s)\|_{L^1_y}+\langle\gamma\rangle^3\|G_{\leq p_0,k_2}(.,v,\gamma,s)\|_{L^\infty_{y}}&\lesssim \varep_1\langle\gamma\rangle^{\delta},
\end{split}
\end{equation}
for any $v\in\R^3$. Indeed, we examine the definitions \eqref{ropi16}--\eqref{ropi17} and rewrite
\begin{equation*}
\begin{split}
H^1_{>p_0,k_2}(y,v,\gamma,s)&=-e^{-i\gamma}\widetilde{H}^1_{>p_0,k_2}(y-(s-\gamma)v +P_2(y,v,\gamma,\tau,s)),\\
\widetilde{H}^1_{>p_0,k_2}(z,v,\gamma)&:=\frac{1}{(2\pi)^3}\int_{\R^3}\varphi_{k_2}(\xi)\mathcal{R}^2(\xi)\frac{\varphi_{>p_0}(1-\xi\cdot v)}{i(-1+\xi\cdot v)}\widehat{\rho^{osc}}(\xi,\gamma)e^{i\xi \cdot z}\,d\xi.
\end{split}
\end{equation*}
It follows from Lemma \ref{omegaioLem} (ii) that
\begin{equation*}
\|\widetilde{H}^1_{>p_0,k_2}(.,v,\gamma)\|_{L^q_z}\lesssim 2^{-p_0}\|\rho^{osc}(.,\gamma)\|_{L^q},\qquad q\in[1,\infty],
\end{equation*}
and the desired bounds in the first line of \eqref{ropi25} follow using \eqref{Laga3} and \eqref{ropi2.5}. The bounds on $H^2_{>p_0,k_2}$ in the second line of \eqref{ropi25} follow by a similar argument.

To prove the bounds \eqref{ropi26} we examine the formula \eqref{ropi18} and notice that if $p_0,k_2\leq 4$ then the functions $G_{\leq p_0,k_2}$ are nontrivial only if $|v|\geq 1$. Using the bounds \eqref{omegaio4}
\begin{equation*}
|G_{\leq p_0,k_2}(y, v, \gamma, s)|\lesssim 2^{2k_2}2^{p_0}\langle v\rangle^{-1}\|P_{k_2}\rho^{osc}(.,\gamma)\|_{L^1}\lesssim \varep_12^{p_0}\langle v\rangle^{-1}\langle\gamma\rangle^{-1+2\delta}2^{k_2},
\end{equation*}
for any $y,v\in\R^3$, as claimed. The remaining bounds in \eqref{ropi26} follow from the corresponding bounds on $\rho(.,\gamma)$ and Lemma \ref{omegaioLem} (i).

{\bf{Step 4.}} We return now to the proof of the main bounds \eqref{ropi8}. In view of the discussion in {\it{Step 2}}, it suffices to show that if $m_2\geq 400$ and $\langle\tau\rangle\leq 2^{9m_2/8}$ then
\begin{equation}\label{ropi30}
\begin{split}
\sum_{\{k_1, k_2\in \Z:\,|k_2+m_2|\leq   m_2/8\text{ and }k_1\geq -9m_2/8\}}2^{-k_2}\|\mathcal{B}_{j,k} (\mathcal{R}^1\rho_{k_1}, e^{-i\gamma}\mathcal{R}^2\rho_{k_2}^{osc}; C)(., \tau, s)\|_{B^0_s}\\
\lesssim   \varep_1^2 \min\{2^{- 1.1 m_2}, \langle \tau \rangle^{-1.1} \}. 
\end{split}
\end{equation}

Recall that $\|\mathcal{K}(.,.,\tau,s)\|_{L^1_{x,v}}+ \|\mathcal{K}(.,.,\tau,s)\|_{L^1_xL^\infty_v}\lesssim 1$ (see \eqref{june16eqn4}). We use the decomposition \eqref{ropi19} with $p_0:=-m_2/2 $. The function $X_{\leq p_0, k_1,k_2}$ defined in \eqref{ropi20} is nontrivial only if $k_2+j\geq -20$. Using \eqref{change1} and \eqref{ropi26}, we have 
\begin{equation}\label{ropi35}
\begin{split}
\|X_{\leq p_0, k_1,k_2}(., \tau, s)\|_{L^1_x} &\lesssim  \int_{t_3}^{t_4}\|\rho_{k_1}(\tau)\|_{L^1}\|G_{\leq p_0,k_2}(., ., \gamma, s)\|_{ L^\infty_{y,v}  }\,d\gamma\\
&\lesssim \frac{\varep_1^2 2^{2\delta m_2}}{\langle\tau\rangle^{1-2\delta}2^{k_1}+\langle\tau\rangle^{-\delta}}2^{p_0}2^{-j}2^{k_2}.
\end{split}
\end{equation}
Using the trivial estimate if $s-\tau\leq\langle s\rangle/8$ and \eqref{ropi2.6} if $s-\tau\geq\langle s\rangle/8$, together with \eqref{ropi26}, 
\begin{equation}\label{ropi36}
\begin{split}
\|X_{\leq p_0, k_1,k_2}(., \tau, s)\|_{L^\infty_x}&\lesssim \int_{t_3}^{t_4}\|G_{\leq p_0,k_2}(., ., \gamma, s)\|_{L^\infty_{y,v}}\cdot  \langle s\rangle^{-3}\|\rho_{k_1}(\tau)\|_{B^0_\tau}\,\,d\gamma\\
&\lesssim \varepsilon_1^22^{2\delta m_2}2^{p_0-j+k_2}\cdot \langle s\rangle^{-3}\frac{1}{\langle\tau\rangle^{1-2\delta}2^{k_1}+\langle\tau\rangle^{-\delta}}.
\end{split}
\end{equation}
Therefore, recalling that   $-j\leq 20+k_2$ and $\langle\tau\rangle\leq 2^{9m_2/8}$, we have 
\begin{equation}\label{ropi31}
\sum_{\{k_1, k_2\in \Z:\,   |k_2+m_2|\leq   m_2/8 \text{ and }k_1\geq -9m_2/8\}}2^{-k_2}\|X_{\leq p_0, k_1,k_2}(., \tau, s)\|_{B^0_s}\lesssim\varep_1^22^{-5m_2/4}.\end{equation}

Now, we proceed to estimate the contributions of the function $Y^1_{>p_0, k_1,k_2}(x, \tau, s)$ defined in \eqref{ropi21}. As before, we use \eqref{change1} and \eqref{ropi25} to estimate, for $p_0=-m_2/2$,
\begin{equation}\label{ropi37}
\begin{split}
\|Y^1_{>p_0, k_1,k_2}(., \tau, s)\|_{L^1_x}&\lesssim \|\rho_{k_1}(\tau)\|_{L^1} \|H^1_{>p_0,k_2}(., ., \gamma, s)\|_{L^\infty_{y,v}}\\
&\lesssim\frac{\varep_1^2}{\langle\tau\rangle^{1-2\delta}2^{k_1}+\langle\tau\rangle^{-\delta}}2^{-p_0}  2^{2\delta m_2}2^{-3m_2}.
\end{split}
\end{equation}
Moreover, using also \eqref{ropi2.6} if $s-\tau\geq\langle s\rangle/8$ and a direct $L^\infty$ bound if $s-\tau\le\langle s\rangle/8$, we have
\begin{equation}\label{ropi38}
\begin{split}
\|Y^1_{>p_0, k_1,k_2}(., \tau, s)\|_{L^\infty_x}&\lesssim \|H^1_{>p_0,k_2}(., ., \gamma, s)\|_{L^\infty_{y,v}}\cdot \langle s\rangle^{-3}\Vert\rho_{k_1}(\tau)\Vert_{B^0_\tau}\\
&\lesssim  \frac{ \varep_1^2\langle s\rangle^{-3}}{\langle\tau\rangle^{1-2\delta}2^{k_1}+\langle\tau\rangle^{-\delta}}2^{-p_0}  2^{\delta m_2} 2^{-3m_2}. 
\end{split}
\end{equation}
Therefore, after combining the obtained estimates \eqref{ropi37} and \eqref{ropi38}, we have 
\begin{equation*}
\sum_{\{k_1, k_2\in \Z:\, |k_2+m_2|\leq   m_2/8 \text{ and }k_1\geq -9m_2/8\}}2^{-k_2}\|Y^1_{>p_0, k_1,k_2}(., \tau, s)\|_{B^0_s}\lesssim      \varep_1^2 2^{- 5m_2/4}.
\end{equation*}

The contributions of the functions $Y^2_{>p_0, k_1,k_2}$ can be estimated in a similar way, as there is an additional factor of $2^{m_2}$ (since $C'$ is replaced with $C$) but also a factor of $2^{-m_2}$ in \eqref{ropi25}. The bounds \eqref{ropi30} follow using also \eqref{ropi31} and recalling that $\langle\tau\rangle^{-1}\geq 2^{-9m_2/8}$. This completes the proof of the main bounds \eqref{june16eqn31A1}.

\subsection{Proof of \eqref{june16eqn31A2}} {\bf{Step 1.}} As in subsection \ref{RopiSec} we decompose
\be\label{ropi40}
\mathcal{B}_{j,k} (E, \nabla E; C)=\sum_{k_1, k_2\in \Z} \sum_{\ast\in \{stat, osc\}} \mathcal{B}_{j,k} (P_{k_1}E, c_{\ast}(\gamma) P_{k_2}(\nabla E^{\ast}); C), 
\ee
where $c_{stat}=1$ and $c_{osc}=e^{-i\gamma}$ as before. As in \eqref{ropi2.7}--\eqref{ropi4}, for $r\in\{1,\infty\}$ we have 
\begin{equation}\label{ropi41}
\begin{split}
\langle s\rangle^{3-3/r}&\|\mathcal{B}_{j,k} (P_{k_1}E, P_{k_2}(\nabla E^{stat}); C)(.,\tau, s)\|_{L^r_x}\\
&\lesssim \frac{\varep_1^22^{-k_1}}{\langle \tau\rangle^{1-2\delta}2^{k_1}+\langle \tau\rangle^{-\delta}} 2^{-k_2^+}2^{2\delta m_2} \min\{2^{k_{\min}}, \langle \tau\rangle^{-1},2^{-m_2}\}^3
\end{split}
\end{equation}
and, as in \eqref{ropi8.1},
\begin{equation}\label{ropi42}
\begin{split}
 \langle s\rangle^{3-3/r}&\|\mathcal{B}_{j,k} (P_{k_1}E, e^{-i\gamma}P_{k_2}(\nabla E^{osc}); C)(.,\tau, s)\|_{L^r_x} \\
 &\lesssim \frac{\varep_1^22^{-k_1}}{\langle \tau\rangle^{1-2\delta}2^{k_1}+\langle \tau\rangle^{-\delta}}\frac{2^{m_2}}{2^{(1-2\delta)m_2}2^{k_2}+2^{-\delta m_2}}\min\{2^{k_{\min}},\langle \tau\rangle^{-1},2^{-m_2}\}^3\\
 &\lesssim \varep_1^2\langle\tau\rangle^{2\delta}2^{2\delta m_2}\min\{\langle \tau\rangle^{-1},2^{-m_2}\}\frac{\min\{1,2^{m_2+k_2},2^{m_2}\langle\tau\rangle^{-1},2^{m_2+k_1}\}}{(1+2^{k_2+m_2})(1+\langle\tau\rangle 2^{k_1})}.
\end{split}
\end{equation}
These bounds are sufficient to estimate the contributions of the static component $E^{stat}$ and of the oscillatory component $E^{osc}$ when $2^{9m_2/8}\lesssim \langle\tau\rangle$. They also suffice to estimate the contributions corresponding to frequencies $\{k_1, k_2\in \Z:\,  |k_2+m_2|\geq  m_2/8 \text{ or }k_1\leq -9m_2/8\text{ or }k_1\leq k_2-m/8\}$ in the oscillatory component $E^{osc}$.

{\bf{Step 2.}} As before, it remains to prove that if $m_2\geq 400$ and $2^{9m_2/8}\geq\langle\tau\rangle$ then
\begin{equation}\label{ropi43}
\begin{split}
&\sum_{|k_2+m_2|\leq m_2/8,\,k_1\geq -m_2/8+\max(k_2,-m_2)}2^{-k_1}\|\mathcal{B}_{j,k} (\mathcal{R}^1\rho_{k_1}, e^{-i\gamma}\mathcal{R}^2\rho_{k_2}^{osc}; C)(., \tau, s)\|_{B^0_s}\\
&\qquad \lesssim \varep_1^2  \min\{2^{- 1.1 m_2}, \langle \tau \rangle^{-1.1} \},
\end{split}
\end{equation}
where $\mathcal{R}^1,\mathcal{R}^2$ are operators defined by multipliers satisfying the differential inequalities \eqref{HormMich}. 

We use the identities \eqref{ropi11}--\eqref{ropi12}, integrate by parts as in \eqref{ropi14}--\eqref{ropi16}, and decompose $\mathcal{B}_{j,k}(\mathcal{R}^1\rho_{k_1}, e^{-i\gamma}\mathcal{R}^2\rho_{k_2}^{osc}; C)$ as in \eqref{ropi17}--\eqref{ropi22} with $p_0=-m_2/2$. Then we use the bounds \eqref{ropi35}--\eqref{ropi36} and recall that the functions $X_{\leq p_0, k_1,k_2}$ are nontrivial only if $-j\leq 20+k_2$, so
\begin{equation}\label{ropi44}
\begin{split}
&\sum_{|k_2+m_2|\leq m_2/8,\,k_1\geq -m_2/8+\max(k_2,-m_2)} 2^{-k_1}\|X_{\leq p_0, k_1,k_2}(., \tau, s)\|_{B^0_s}\\
&\qquad\lesssim  \sum_{|k_2+m_2|\leq m_2/8,\,k_1\geq -m_2/8+\max(k_2,-m_2)}2^{-k_1}\frac{\varep_1^2 2^{2\delta m_2}}{\langle\tau\rangle^{1-2\delta}2^{k_1}+\langle\tau\rangle^{-\delta}}2^{p_0}2^{2k_2}\\
&\qquad\lesssim \varep_1^2  2^{3\delta m_2}\langle\tau\rangle^{3\delta}2^{-5m_2/4}.
\end{split}
\end{equation}
Similarly, using \eqref{ropi37}--\eqref{ropi38}, for $l\in\{1,2\}$ we estimate
\begin{equation}\label{ropi45}
\begin{split}
& \sum_{|k_2+m_2|\leq m_2/8,\,k_1\geq -m_2/8+\max(k_2,-m_2)}2^{-k_1}\|Y^l_{>p_0, k_1,k_2}(., \tau, s)\|_{B^0_s}\\
&\qquad\lesssim   \sum_{|k_2+m_2|\leq m_2/8,\,k_1\geq -m_2/8+\max(k_2,-m_2)}2^{-k_1}\frac{\varep_1^22^{-p_0}  2^{-m_2(3-2\delta)}}{\langle\tau\rangle^{1-2\delta}2^{k_1}+\langle\tau\rangle^{-\delta}}\\
&\qquad\lesssim\varep_1^2  2^{3\delta m_2}\langle\tau\rangle^{3\delta}2^{-5m_2/4}.
\end{split}
\end{equation}
The desired bounds \eqref{ropi43} follow since $2^{9m_2/8}\geq\langle\tau\rangle$.

\subsection{Proofs of \eqref{june16eqn32}--\eqref{april10eqn51}} These estimates can be proved in the same way. We first decompose the electric field in the second position into its static and oscillatory components, and decompose  dyadically in frequency. Then we estimate the contributions of the static components using bounds similar to \eqref{ropi41} (with different factors of $2^{-k_1}$ or $2^{-k_2}$) and the contributions of most of the oscillatory components using bounds similar to \eqref{ropi42}.

After these reductions we are left with the contributions of the oscillatory components when $m_2\geq 400$ and $2^{9m_2/8}\geq\langle\tau\rangle$, coming from frequencies $\{k_1, k_2\in \Z:\,|k_2+m_2|\leq m_2/8\text{ and }k_1\geq -m_2/8+\max\{-m_2,k_2\}\}$. Here we integrate by parts in $\gamma$, as before, and use the bounds \eqref{ropi35}--\eqref{ropi36} and \eqref{ropi37}--\eqref{ropi38} to control the corresponding contributions. To get bounds similar to \eqref{ropi44}--\eqref{ropi45} we need to use different cutoff values of $p_0$: we set $p_0=-3m_2/8$ for the bounds \eqref{june16eqn32} and $p_0=-m/2$ for the bounds \eqref{june16eqn33}--\eqref{april10eqn51} (the frequency $k_1$ is fixed in the last two bounds, in order to avoid the divergent sum corresponding to $k_1\geq 0$, and the sum in the second step is simply taken over frequencies $k_2\in[-9m_2/8,-7m_2/8]$).

\section{Proof of Theorem \ref{thm:main_simple}}\label{ProofMainThm}
In this section we show that Proposition \ref{MainBootstrapProp} implies the main result, Theorem \ref{thm:main_simple}. 

\begin{proof}[Proof of global existence and uniqueness]
Let $\bar\varepsilon>0$ and $\delta>0$ as in Proposition \ref{MainBootstrapProp} be given, and let $f_0$ satisfying the assumptions \eqref{eq:init} (resp.\ \eqref{bootinit}) be given. Then by standard local existence theory we can find $T>0$ and a unique solution $f\in C^1(\R^3\times\R^3\times [0,T])$ to \eqref{NVP}, \eqref{NVP.2}. Moreover, by Corollary \ref{rhodeco} we can decompose the associated density $\rho(x,t)$ as in \eqref{sug31}, and by continuity (note that $\rho^{osc}(x,0)=0$) we can assume that $T>0$ is chosen such that the smallness assumption \eqref{YW12} of Proposition \ref{MainBootstrapProp} is met.

Noting that $f$ can be recovered exactly by integrating along the characteristics, i.e.\ that by \eqref{Lan2} we have
\begin{equation}\label{eq:sol_chars}
\begin{split}
 f(x,v,t)=f_0(X(x,v,0,t),V(x,v,0,t))+\left[M_0(V(x,v,0,t))-M_0(v)\right],
\end{split}
\end{equation}
and the characteristics in turn are determined by $\rho$ and its derivative through the system of ODEs \eqref{Lan1}, we can apply the conclusion \eqref{YW12improved} of Proposition \ref{MainBootstrapProp} to extend the solution $f$ to a larger time interval $[0,T']$ with $T'>T$, on which the splitting of the density according to Corollary \ref{rhodeco} and the corresponding smallness condition \eqref{YW12} still hold. Hence we obtain a unique global solution $f\in C^1(\R^3\times\R^3\times\R_{+})$ with associated density $\rho=\rho^{stat}+\Re\{e^{-it}\rho^{osc}\}$ satisfying for all $t>0$ that
\begin{equation}\label{eq:density_concl}
 \norm{\rho^{stat}}_{Stat_\delta}+\norm{\rho^{osc}}_{Osc_\delta}\lesssim \varepsilon_0.
\end{equation}
The corresponding splitting \eqref{eq:E-decomp} of the electric field follows directly by \eqref{Laga5}.
\end{proof}

We prove next the scattering statement \eqref{eq:lin_scatter} of Theorem \ref{thm:main_simple}:
\begin{proposition}\label{ScatProp}
Assume $f$ is a global solution as constructed above, so that in particular the associated density satisfies \eqref{eq:density_concl}. Then there exists $f_\infty\in L^\infty_{x,v}$ such that 
\begin{equation}\label{LinftyLinearScattering}
\begin{split}
\Vert f(x,v,t)-f_\infty(x-tv,v)\Vert_{L^\infty_{x,v}}\lesssim\varepsilon_1\langle t\rangle^{-\frac{1}{2}}.
\end{split}
\end{equation}
\end{proposition}

\begin{remark}
Using formula \eqref{eq:sol_chars} and with additional simple work on the characteristics, one could obtain the formula
\begin{equation*}
\begin{split}
f_\infty(x,v):=(f_0\circ\Phi)(x,v)+\left[M_0(\Phi^2(x,v))-M_0(v)\right]
\end{split}
\end{equation*}
where $\Phi=(\Phi^1,\Phi^2)$ is a $C^1$ symplectic diffeomorphism.
\end{remark}

Proposition \ref{ScatProp} essentially follows from the following lemma:
\begin{lemma}\label{ConvYV}
Under the assumptions of Proposition \ref{ScatProp}, there exist $(\Y_\infty,\W_\infty):\R^3\times\R^3\times\R_{+}\to\R^3$ such that for each fixed $s$, we have that
\begin{equation}\label{eq:charconv}
\begin{split}
\langle t\rangle^{\frac{1}{2}}\langle v\rangle^{-1}\left[\vert \widetilde{Y}(x,v,s,t)-\Y_\infty(x,v;s)\vert+\vert \widetilde{W}(x,v,s,t)-\W_\infty(x,v;s)\vert\right]\lesssim \varepsilon_1,
\end{split}
\end{equation}
uniformly in $0\le s\le t<\infty$. 
\end{lemma}

We first show how Proposition \ref{ScatProp} follows from Lemma \ref{ConvYV}.

\begin{proof}[Proof of Proposition \ref{ScatProp} assuming Lemma \ref{ConvYV}]
%
Write for simplicity $\Y_\infty(x,v):=\Y_\infty(x,v;0)$ and $\W_\infty(x,v):=\W_\infty(x,v;0)$ and let
\begin{equation*}
\begin{split}
G(y,u)&:=f_0(y+\Y_\infty(y,u),u+\W_\infty(y,u))+\left[M_0(u+\W_\infty(y,u))-M_0(u)\right]\\
\end{split}
\end{equation*}
be the (putative) scattering data from Lemma \ref{ConvYV}. By \eqref{eq:sol_chars} we can rewrite
\begin{equation*}
\begin{split}
f(x,v,t)&-G(x-tv,v)=f_0(x-tv+\widetilde{Y}(x-tv,v,0,t),v+\widetilde{W}(x-tv,v,0,t))\\
&-f_0(x-tv+\Y_\infty(x-tv,v),v+\W_\infty(x-tv,v))\\
&+M_0(v+\widetilde{W}(x-tv,v,0,t))-M_0(v+\W_\infty(x-tv,v))\\
&=\delta Y(x,v,t)\cdot A_x(x,v,t)+\delta W(x,v,t)\cdot A_v(x,v,t)+\delta W(x,v,t)\cdot B(x,v,t),
\end{split}
\end{equation*}
where
\begin{align*}
\delta Y & :=\widetilde{Y}(x-tv,v,0,t)-\Y_\infty(x-tv,v), & Y_0 & :=\widetilde{Y}(x-tv,v,0,t), & Y_\infty &:=\Y_\infty(x-tv,v)\\
\delta W & :=\widetilde{W}(x-tv,v,0,t)-\W_\infty(x-tv,v), &  W_0 & :=\widetilde{W}(x-tv,v,0,t), & W_\infty &:=\W_\infty(x-tv,v)
\end{align*}
and
\begin{align*}
A_x&:=\int_{\theta=0}^1\nabla_xf_0(x-tv+\theta Y_0+(1-\theta)Y_\infty,v+\theta W_0+(1-\theta)W_\infty)d\theta, \\
A_v&:=\int_{\theta=0}^1\nabla_vf_0(x-tv+\theta Y_0+(1-\theta)Y_\infty,v+\theta W_0+(1-\theta)W_\infty)d\theta,\\
B&:=\int_{\theta=0}^1\nabla_vM_0(v+\theta W_0+(1-\theta)W_\infty)d\theta.
\end{align*}

By our choice of background $M_0$ in \eqref{NVP.2} and the assumptions on the initial data \eqref{bootinit}, together with the estimate \eqref{cui6} in Lemma \ref{derichar2}, which ensures that $\langle v \rangle \sim \langle v + \widetilde{W}\rangle,$ we directly see that
\begin{equation*}
\begin{split}
\Vert \langle v\rangle A_x\Vert_{ L^\infty_{x,v}}+\Vert \langle v\rangle A_v\Vert_{ L^\infty_{x,v}}&\lesssim\varepsilon_0,\qquad
\Vert \langle v\rangle B\Vert_{L^\infty_{x,v}}\lesssim 1.
\end{split}
\end{equation*}
Letting $f_\infty:=G$, together with the bound \eqref{eq:charconv} for $\delta Y$ and $\delta W$ the claim follows.
\end{proof}

It remains to prove Lemma \ref{ConvYV}.

\begin{proof}[Proof of Lemma \ref{ConvYV}]

Given $0\le s\le t\le t^\prime$, we can write
\begin{equation*}
\begin{split}
\Delta_{s,t,t^\prime}(x,v)&:=\widetilde{Y}(x,v,s,t^\prime)-\widetilde{Y}(x,v,s,t)\\
&=\int_{\tau=t}^{t^\prime}(\tau-s)E(x+\tau v+\widetilde{Y}(x,v,\tau,t^\prime),\tau)d\tau\\
&\quad+ \int_{\tau=s}^t(\tau-s)\left[E(x+\tau v+\widetilde{Y}(x,v,\tau,t^\prime),\tau)-E(x+\tau v+\widetilde{Y}(x,v,\tau,t),\tau)\right]d\tau\\
&=\mathfrak{e}_{s,t,t^\prime}(x,v)+\mathfrak{r}_{s,t,t^\prime}(x,v).
\end{split}
\end{equation*}
A crude estimate gives
\begin{equation*}
\begin{split}
\vert \mathfrak{r}_{s,t,t^\prime}(x,v)\vert&\lesssim \int_{\tau=s}^t \tau\Vert \nabla_xE(\tau)\Vert_{L^\infty_x}\cdot\vert \Delta_{\tau,t,t^\prime}(x,v)\vert d\tau
\end{split}
\end{equation*}
and observing that
\begin{equation*}
g(\tau):=\tau \Vert \nabla_xE(\tau)\Vert_{L^\infty_x}\lesssim \varepsilon_1\langle \tau\rangle^{2\delta-2}
\end{equation*}
belongs to $L^1_\tau$ and using Gr\"onwall's Lemma, we see that it suffices to bound $\mathfrak{e}_{s,t,t^\prime}(x,v)$ by $\varepsilon_1\ip{v}\ip{t}^{-\frac{1}{2}}$ to prove the claim for $\Y-\Y_\infty$. To obtain this, we can decompose
\begin{equation*}
\begin{split}
\mathfrak{e}_{s,t,t^\prime}(x,v)&=\mathfrak{e}^{stat}_{s,t,t^\prime}(x,v)+\Re(\mathfrak{e}^{osc}_{s,t,t^\prime}(x,v)),\\
\mathfrak{e}^{(\ast)}_{s,t,t^\prime}(x,v)&=\int_{\tau=t}^{t^\prime}(\tau-s)E^\ast(x+\tau v+\widetilde{Y}(x,v,\tau,t^\prime),\tau)d\tau.
\end{split}
\end{equation*}
On the one hand, we see that
\begin{equation*}
\begin{split}
\vert \mathfrak{e}^{stat}_{s,t,t^\prime}(x,v)\vert&\lesssim  \int_{\tau=t}^{t^\prime} \tau\Vert E^{stat}(\tau)\Vert_{L^\infty_x}d\tau\lesssim \varepsilon_1 \int_{\tau=t}^{t^\prime}\langle \tau\rangle^{2\delta-2}d\tau\lesssim\varepsilon_1 \ip{t}^{2\delta-1}.
\end{split}
\end{equation*}
We now consider
\begin{equation*}
\begin{split}
\mathfrak{e}^{osc}_{s,t,t^\prime}(x,v)&= \int_{\tau=t}^{t^\prime} (\tau-s) e^{-i\tau}E^{osc}(x+\tau v+\widetilde{Y}(x,v,\tau,t^\prime),\tau)d\tau
\end{split}
\end{equation*}
and do a direct normal form transformation to get
\begin{equation*}
\begin{split}
\mathfrak{e}^{osc}_{s,t,t^\prime}(x,v)&= \left[ i(\tau-s) e^{-i\tau}E^{osc}(x+\tau v+\widetilde{Y}(x,v,\tau,t^\prime),\tau)\right]_{\tau=t}^{t^\prime}\\
&\quad-i\int_{\tau=t}^{t^\prime} e^{-i\tau}E^{osc}(x+\tau v+\widetilde{Y}(x,v,\tau,t^\prime),\tau)d\tau\\
&\quad-i\int_{\tau=t}^{t^\prime} e^{-i\tau}(\tau-s) \partial_\tau E^{osc}(x+\tau v+\widetilde{Y}(x,v,\tau,t^\prime),\tau)d\tau\\
&\quad-i\int_{\tau=t}^{t^\prime}e^{-i\tau} (\tau-s) \left((v^j+\partial_\tau\widetilde{Y}^j(x,v,\tau,t^\prime))\partial_{x^j}\right)E^{osc}(x+\tau v+\widetilde{Y}(x,v,\tau,t^\prime),\tau)d\tau.
\end{split}
\end{equation*}
Another crude integration gives that
\begin{equation*}
\begin{split}
\vert \mathfrak{e}^{osc}_{s,t,t^\prime}(x,v)\vert&\lesssim\varepsilon_1\langle v\rangle\langle t\rangle^{2\delta-1}.
\end{split}
\end{equation*}
To sum up, we have proved that 
\be
\big|\Delta_{s,t,t^\prime}(x,v)\big|\lesssim \varepsilon_1\langle v\rangle\langle t\rangle^{2\delta-1}.
\ee

\medskip

We can control the $\W$ variation using a similar decomposition and the above obtained estimate of $\Delta_{s,t,t^\prime}(x,v)$:
\begin{equation*}
\begin{split}
\vert \widetilde{W}(x,v,s,t^\prime)-&\widetilde{W}(x,v,s,t)\vert
\le \int_{\tau=t}^{t^\prime} \left\vert E(x+\tau v+\widetilde{Y}(x,v,\tau,t^\prime),\tau)\right\vert d\tau\\
& + \int_{\tau=s}^t\left\vert E(x+\tau v+\widetilde{Y}(x,v,\tau,t^\prime),\tau)-E(x+\tau v+\widetilde{Y}(x,v,\tau,t),\tau) \right\vert d\tau\\
&\lesssim \varepsilon_1 \int_{\tau=t}^{t^\prime}\langle \tau\rangle^{2\delta-2}d\tau+ \int_{\tau=s}^t \Vert \nabla_x E(\tau)\Vert_{L^\infty}\cdot \vert \Delta_{\tau,t,t^\prime}(x,v)\vert d\tau\\
&\lesssim \varepsilon_1 \langle t\rangle^{2\delta-1}+\varepsilon_1^2 \langle v\rangle\langle t\rangle^{2\delta-1} \int_{\tau=s}^t\langle \tau\rangle^{2\delta-3} d\tau\lesssim \varepsilon_1\langle v\rangle\langle t\rangle^{2\delta-1}.
\end{split}
\end{equation*}
 Hence finishing the proof of  Lemma \ref{ConvYV}.
\end{proof}

 \bibliographystyle{abbrev}

\end{document}